\documentclass[a4wide,10pt]{article}%
\usepackage{geometry}                
\usepackage{tikz} 
\usetikzlibrary{positioning}
\usepackage{graphicx}
\usepackage{authblk}
\usepackage{amssymb}
\usepackage{epstopdf}
\usepackage{subfigure}  
\usepackage[sans]{dsfont}
\usepackage[applemac]{inputenc}
\usepackage[english]{babel}
\usepackage{latexsym}
\usepackage{mathrsfs}
\usepackage{amscd}
\usepackage{graphicx}
\usepackage{color}
\usepackage{float}
\usepackage{amsmath}
\usepackage{amsfonts}
\numberwithin{equation}{section}
\usepackage{enumerate}
\usepackage{amsthm}
\usepackage{algorithm}
\usepackage{algorithmic}
\usepackage[numbers,sort]{natbib}
\usepackage[bookmarks=true,colorlinks=true,linkcolor={blue},urlcolor={blue}, citecolor={blue},pdfstartview={XYZ null null 1.22}]{hyperref}%

\providecommand{\keywords}[1]{\textbf{{Key words.}} #1}

\newcommand{\R}{\mathbb R}

\def\be#1\ee{\begin{equation}#1\end{equation}}

\newtheorem{corollary}{Corollary}

\newtheorem{proposition}{Proposition}

\theoremstyle{definition}
\newtheorem{alg}{Algorithm}[section]
\newtheorem{remark}{Remark}

\def\RR{\mathbb R}

\def\T{\mathcal T}

\def\be{\begin{equation}}
	\def\ee{\end{equation}}
\def\bea{\begin{eqnarray}}
	\def\eea{\end{eqnarray}}

\title{Kinetic based optimization enhanced by genetic dynamics}

\author{Giacomo Albi\footnote{	Dipartimento di Informatica, Universit\`a di Verona, Verona, Italy, e-mail: giacomo.albi@univr.it}, Federica Ferrarese\footnote{		Dipartimento di Matematica, Universit\`a di Trento, e-mail: federica.ferrarese@unitn.it}, Claudia Totzeck\footnote{		School of Mathematics and Natural Sciences, University of Wuppertal, e-mail: 	totzeck@uni-wuppertal.de}}

\begin{document}
	\date{}
	\maketitle

\begin{abstract}
 We propose and analyse a variant of the recently introduced kinetic based optimization method that incorporates ideas like survival-of-the-fittest and mutation strategies well-known from genetic algorithms. Thus, we provide a first attempt to reach out from the class of consensus/kinetic-based algorithms towards genetic metaheuristics. Different generations of genetic algorithms are represented via two species identified with different labels, binary interactions are prescribed on the particle level and then we derive a mean-field approximation in order to analyse the method in terms of convergence. Numerical results underline the feasibility of the approach and show in particular that the genetic dynamics allows to improve the efficiency, of this class of global optimization methods in terms of computational cost.
\end{abstract}

\keywords{Global optimization; Mean-field limit; Boltzmann equations; Particle-based methods;  Consensus-based optimization.}

\section{Introduction}\label{sec:intro}
In recent years a new perspective on gradient-free methods for global optimization of non-convex high-dimensional functions was established based on collective motion of particles. The mathematical depiction of self-organizing dynamics has significant implications in various applications, including traffic flow, pedestrian dynamics, flock behavior of birds, cell migration, and its primary focus lies in comprehending emergent properties of large interacting systems across different scales, see for example \cite{albi2019vehicular,bellomo2021life,bellomo2022towards,carrillo2010particle,motsch2014heterophilious}.
Hence, this new class of methods aims to exploit the collective dynamics of swarms to find global optimizers as consensus points emerging from a system of interacting agents, where each agent represents an evaluation of the objective function. 
In this context, Consensus-based optimization (CBO) is a global optimization method allowing rigorous proofs that the swarms concentrate arbitrarily close to the unique global minimizer of the objective function,  we refer to ~\cite{carrillo2018analytical,fornasier2020consensus,fornasier2021consensus,grassi2021particle,pinnau2017consensus} and the overview in \cite{totzeck2021trends}.
In CBO, differently from gradient-based methods, such as stochastic gradient descent\cite{bottou2018optimization},  the landscape of the objective functions is explored via function evaluations only. Indeed, the objective value at the current position and the current position of the agents is exchanged, with the help of a weighted mean value which is constructed such that the Laplace principle \cite{dembo1998zeitouni} applies. The dynamics is tailored such that the agents confine towards the weighted-mean on the one hand, and randomly explore the landscape on the other hand. This already indicates that the two components of the dynamics need to be well-balanced to obtain desirable results.
CBO was a first step towards the mathematical understanding of metaheuristics for global optimization, such as particle swarm optimization (PSO) methods, where second order dynamics is used for the evolution of the particles \cite{kennedy1995particle,poli2007particle}. Recently, the gap between the first-order method CBO and PSO was bridged in \cite{grassi2021particle}, further extensions were provided for systems with memory or momentum effects \cite{totzeck2020memory,Chen2022adamCBO}, for constrained optimization \cite{borghi2023constrained}, multi-objective problems \cite{borghi2022consensus,klamroth2022consensus} and jump-diffusion processes \cite{kalise2023jump}.
More recently, kinetic-based optimization (KBO) methods have been proposed in\cite{benfenati2022binary}, where
each agent with position $x$ moves subject to the following interaction rules  	
\begin{equation}\label{eq:bin_1pop}
	x' = x + \nu_F  (\hat{x}(t)-x)+\sigma_F   D(x) \xi,
\end{equation}
where $x'$ denotes the post iteration position,  $\sigma_F$, $\nu_F$ are positive parameters which allow to balance the exploitation and exploration of the swarm, $\xi$ is a random perturbation term, $D(x)$ is a diffusion matrix and $\hat{x}(t)$ denotes the global estimate of the position of the global minimizer at time $t$. In addition to this dynamics a drift towards the local best and a local diffusion term is proposed in \cite{benfenati2022binary}. The corresponding dynamics is described by a multidimensional Boltzmann equation and can be simulated with the help of Monte Carlo algorithms \cite{albi2013binary,pareschi2013interacting}.

With this work, we aim at extending KBO algorithm reaching out towards genetic algorithms (GA) \cite{zbigniew1993ga,toledo2014global}, a very popular class of metaheuristics that is widely used in engineering.  
The GA models a natural selection process based on biological evolution \cite{golberg1989genetic,mitchell1995genetic}. To this end, individuals (parents) from the current population are selected and their objective values (gene information) is combined to generate the next generation (children). The selection process is usually driven by a survival-of-the-fittest idea, hence over successive generations, the system is assumed to evolve towards an optimal solution. Agents in promising positions, i.e., with small objective values, are labeled as parents and the others are labeled as  children. Parents do not modify their position, hence they survive the iteration like in a survival-of-the-fittest strategy. In contrast, a child in position $x$ interacts with a randomly chosen parent in position $x_*$ and updates the position according to the rules
\begin{align}\label{eq:GA}
		x' &= x_*, \; \quad \text{with rate} \; \nu_F, \nonumber \\
		x' &= x, \; \quad \text{with rate} \; 1-\nu_F.
\end{align}
Here $x'$ denotes the post interaction position and $\nu_F >0$ is the jump rate. Furthermore, mutations can occur, that means a child in position $x$ encounters a random perturbation of the form
\begin{equation}\label{eq:GA_1}
	x' = x + \sigma_F  \xi,
\end{equation}
where $\sigma_F$ is a positive parameter and $\xi$ random vector drawn from a normal distribution.   Slight modifications of the mutation process can be considered, assuming that all children encounter a random perturbation of the type
\begin{equation}\label{eq:GA_mod}
	x' = x +\sigma_F D(x) \xi,
\end{equation}
with diffusion matrix $D(x)$. 

Thus, to establish a relation between KBO and GA, we divide the swarm into two species called followers (children) and leaders (parents) via a labeling strategy, evolving according to different transition processes \cite{albi2019leader,loy2021boltzmann}. 
 In this setting, leaders are selected analogously to parents in genetic algorithms, whereas  the process of child's generation is substituted by the relaxation of the followers toward the leaders' position.
Then, the dynamics of KBO \eqref{eq:bin_1pop} is split between the leaders, which move toward the global optimum estimate $\hat x$, and the followers, which explore the minimization landscape. Furthermore, we aim to improve the performance of these algorithms by relying on the higher flexibility of the particle dynamics.

More details are presented in the following sections which are organized as follows:
in Section \ref{sec:GKBO} we introduce the  Genetic Kinetic based optimization methods focusing on the description of the binary interactions rules which describes the dynamics. In Section \ref{sec:mean field}, we derive the mean field, in particular the evolution equations of the density functions of the two species. In Section \ref{sec:leaders}, we discuss different strategies of how to assign labels. In Section \ref{sec:moments_convergence}, we provide a theoretical analysis including the exponential decay of the variance and  the convergence of the method to the global minimum. In Section \ref{sec:num_methods}, we describe Nanbu's algorithm which is used to obtain the numerical results presented in Section \ref{sec:validation}. Here, we show different numerical experiments, testing the efficiency in terms of success rate and   number of iterations and compare the results  of the GKBO algorithm, to KBO and genetic algorithms.

\section{Genetic kinetic based optimization (GKBO)} \label{sec:GKBO} 
The GKBO method we propose in the following enhances kinetic based optimization, which belongs to the class of consensus based algorithms, with ideas from genetic algorithms. To this end, we assume to have a population divided into two groups, similar to the parent and children populations in  genetic algorithms. The two groups are specified with the help of labels leading to a modified KBO dynamic with followers and leaders. The dynamics is tailored in such a way that the the population clusters at the unique global minimum of the possibly non-convex objective function $\mathcal{E} \colon \RR^d \rightarrow \RR$. Hence, in the long time limit the dynamics solves the global optimization problem given by
\begin{equation}\label{eq:optProblem}
	\min\limits_{x\in\RR^d} \mathcal{E}(x),
\end{equation}
where $\mathcal E$ is assumed to have a unique global minimizer. In more detail, each agent is described by its position $x \in \mathbb{R}^d$ varying continuously
and a binary variable for the leadership-level $\lambda \in \{0,1\}$.  In the following we identify leaders with  $\lambda = 1$ and followers with $\lambda = 0$. We are interested in the
evolution of the density function \begin{equation}\label{eq:def_f}
	f=f(x,\lambda,t), \qquad f: \R^d\times\left\{0,1\right\} \times \R_+\rightarrow \R_+
\end{equation}
where $t\in\RR^+$ denotes as usual the time variable. In the rest of the paper we denote $f(x,\lambda,t)$ as $f_\lambda(x,t)$ and define
\begin{equation}
	g(x,t) = \sum_{\lambda \in \{0,1\}}   f_\lambda(x,t), 
\end{equation}
to be the density of the whole population at time $t$. We assume that $g(x,t)$ is normalized, hence a probability measure and introduce the fractions $\rho_\lambda\in[0,1]$ with $\lambda\in \{0,1\}$ s.t. $\rho_0 + \rho_1 = 1$ and $f_\lambda(x,t)/\rho_\lambda$ are probability measures as well. 

\subsection{Binary interaction between agents}
A binary interaction of agents with state $(x,\lambda)$ and $(x_*,\lambda_*)$ is described by their post-interaction positions given by
\begin{align}\label{eq:bin_2pop}
		x' & = x + \left( \nu_F (x_*-x) + \sigma_F D(x)\xi \right)\lambda_* (1-\lambda)+ \nu_L(\hat{x}(t)-x)\lambda, \nonumber \\
		x_*' &= x_*,
\end{align}
where $\sigma_F,$ $\nu_F$, $\nu_L$, are positive parameters, $\xi$ is a normally distributed random number and $D(x)$ is the diffusion matrix, defined to be either 
\begin{equation}\label{eq:diffusion_iso}
	D(x) = \vert \hat{x}(t)-x\vert Id_d,
\end{equation}
in the case of isotropic diffusion \cite{pinnau2017consensus}, or 
\begin{equation}\label{eq:diffusion_ani}
	D(x) = diag\{(\hat{x}(t)-x)_1,\ldots (\hat{x}(t)-x)_d\} ,
\end{equation}
in the case of anisotropic diffusion \cite{carrillo2021consensus}.
In equation \eqref{eq:bin_2pop}- \eqref{eq:diffusion_iso}-\eqref{eq:diffusion_ani} the term $\hat{x}(t)$ is the global estimate of the best position of the minimizer.  The term $\hat{x}(t)$ is computed as a convex combination of particle
locations weighted by the cost function according to Laplace principle \cite{dembo1998zeitouni}. In case we consider the whole population, we have
\begin{equation}\label{eq:x_tot}
	\hat{x}(t) = \frac{\int_{\mathbb{R}^d}x e^{-\alpha \mathcal{E}(x)}g(x,t)dx}{\int_{\mathbb{R}^d} e^{-\alpha \mathcal{E}(x)}g(x,t)dx},
\end{equation} 
and Laplace principle yields
\begin{equation}\label{eq:laplace}
	\lim_{\alpha \to \infty} \Big( -\frac{1}{\alpha} \int_{\mathbb{R}^d} e^{-\alpha \mathcal{E}(x)} g(x,t) dx \Big) = \inf_{x\in \text{supp }  g(x,t)} \mathcal{E}(x).
\end{equation} 
In the section on numerical results we will also consider variants, where the weighted mean is computed with information of leaders or followers only.
In \eqref{eq:bin_2pop} we use a compact form to describe both the followers and leaders dynamics. In particular, if we assume $\lambda =0$ (and $\lambda_*=1$) we retrieve the followers dynamics, 
	\begin{align*}
		x' & = x + \left( \nu_F (x_*-x) + \sigma_F D(x)\xi \right), \\
		x_*' &= x_*,
	\end{align*}
where $x,x_*$ denote the pre-interaction positions of a follower and a leader, respectively.
Followers are attracted toward a randomly selected leader and explore the space, searching for a possible position  of the global minimizer. On the contrary, if $\lambda = 1$, we retrieve the leaders dynamics,
\begin{align*}
			x' & = x +  \nu_L(\hat{x}(t)-x), \nonumber \\
	x_*' &= x_*,
\end{align*}
where $x,x_*$ denote the pre-interaction positions of two leaders.
 Hence, leaders do not explore the space but they relax their position toward the estimated position of the global minimizer at time $t$, which is given by $\hat{x}(t)$, defined in \eqref{eq:x_tot}. 
\begin{remark}
	Note that \eqref{eq:bin_2pop} implies that no follower-follower interactions are considered, since if both $\lambda$ and $\lambda_*$ are equal to zero the agents keep their positions.
\end{remark}
\begin{remark}
	 The goal of these algorithms is to find the global minimizer by letting agents concentrate on it. The exploration part slows down the convergence, especially for high value of the diffusion parameters. In the GKBO algorithm, just one of the two populations, the followers, is involved in the exploration process while the other agents, the leaders, are reaching quite immediately consensus on the estimated position of the global minimizer. Under these assumptions, the convergence process is speeded up, as we will see in the numerical experiments in Section~\ref{sec:test1}.
	\end{remark}
\subsection{Emergence of leaders and followers}\label{sec:leaders}
The emergence of leaders and followers is realized with the help of a transition operator which acts as follows
	\begin{align}\label{eq:master_0}
		\mathcal{T}[f_0](x,t) &= \pi_{L\to F}(x,\lambda;f)f_1(x,t)-  \pi_{F\to L}(x,\lambda;f) f_0(x,t), \nonumber \\
		\mathcal{T}[f_1](x,t) & =  \pi_{F\to L}(x,\lambda;f)f_0(x,t)-  \pi_{L\to F}(x,\lambda;f)f_1(x,t),
	\end{align}
where $\pi_{F\to L}(\cdot)$ and $\pi_{L\to F}(\cdot)$ are certain transition rates, possibly depending on the current states.
In the simplest case, if we assume that leaders emerge with fixed rate $\pi_{FL}>0$ and return to the followers status with fixed rate $\pi_{LF}>0$ then the transition rates reduce to
\begin{equation}\label{eq:rates_test_0}
	\pi_{L\to F}  =\pi_{LF},\qquad
	\pi_{F\to L} =\pi_{FL}.
\end{equation}
However, we also cover more general cases, as for example proposed in \cite{haskovec2013flocking}, where each agent in position $x$ is associated with a weight  
\begin{equation*}
	\begin{split}
		\omega(x,t) =\frac{1}{N} &\# \left\lbrace y\in \mathcal{A}(t): \vert  \mathcal{E}(x_{min}) -  \mathcal{E}(y)\vert < \vert  \mathcal{E}(x_{min}) -  \mathcal{E}(x)\vert  \right\rbrace \\
		= &  \frac{1}{N}\sum_{\lambda \in \{0,1\}}\int_{\mathbb{R}^{d}} \chi_{[0,1)}\left(\frac{\vert \mathcal{E}(x_{min}(t)) - \mathcal{E}(y) \vert}{\vert \mathcal{E}(x_{min}(t)) -  \mathcal{E}(x) \vert}\right) f(y,\lambda,t) dy, 
	\end{split}
\end{equation*} 
with $\chi_{[0,1)}(\cdot)$ denoting the characteristic function of the interval $[0,1)$ and
\begin{equation*}
	x_{min}(t) = \arg \min_{x \in \mathcal{A}(t)} \mathcal{E}(x),
\end{equation*}
where $\mathcal{A}(t)$ is the set of agents at time $t$. Assuming that agents with weight smaller than a certain threshold $\bar{\omega}$, which depends on the amount of leaders that we would like to generate, are in the leaders status while the others are in the followers status, then we can write the transition rates as follows 
\begin{align}\label{eq:rates_test_1}
	\pi_{L\to F}  =\begin{cases}
		1, \qquad \text{if } \omega(x,t)>\bar{\omega},\\
		0, \qquad \text{if } \omega(x,t)\leq \bar{\omega},\\
	\end{cases}   \qquad
	\pi_{F\to L}  =\begin{cases}
		0, \qquad \text{if } \omega(x,t) \geq \bar{\omega},\\
		1, \qquad \text{if } \omega(x,t)<\bar{\omega}.\\
	\end{cases}
\end{align}
The evolution of the emergence and decay of leaders can be described by the master equation
\begin{equation}\label{eq:Ldef2}
	\dfrac{d}{dt}\rho_\lambda(t) + \int_{\RR^{2d}} \T[f](x,\lambda,t)\,dx=0,
\end{equation}
for $\lambda \in \{0,1\}$, with $\rho_\lambda(t) = \int_{\RR^d} f_\lambda(x,t) dx$.
From the above definition of the transition operator $\T[\cdot]$, it follows that 
\begin{equation}
	\dfrac{d}{dt}\sum_{\lambda\in\{0,1\}}\rho_\lambda(t)=0. 
	\label{eq:tnc}
\end{equation}
In case of constant transition rates 
\begin{equation}\label{eq:constant_rates}
	\pi_{L\to F}(\cdot) = \pi_{LF},\qquad \pi_{F\to L}(\cdot) = \pi_{FL},
\end{equation}
we can rewrite equation \eqref{eq:Ldef2} as 
\begin{align}\label{eq:master_constant}
		&\partial_{t} \rho_1(t) =\pi_{FL} \rho_0(t) - \pi_{LF}\rho_1(t),\nonumber \\
		&\partial_{t} \rho_0(t) =\pi_{LF} 	\rho_1(t)-\pi_{FL} \rho_0(t).
\end{align}
which allows us to calculate its stationary solution explicitly as
\begin{equation}\label{eq:stationary}
	\rho_1^\infty = \frac{\pi_{FL}}{\pi_{LF} + \pi_{FL}},\qquad 	\rho_0^\infty = \frac{\pi_{LF}}{\pi_{LF} + \pi_{FL}}.
\end{equation}

\begin{remark}\label{remark} 
		The weighted strategy is inspired from the selection criterion of GA, where parents are chosen to be the agents in best position w.r.t.~the cost function. In the numerical experiments we will also consider a mixed strategy, assuming that a certain percentage $\bar{p}$ of the total amount of leaders change their label according to the weighted strategy and the remaining ones changes their label randomly.

\end{remark}

\section{Derivation of the mean-field equation}\label{sec:mean field}
Combining the interaction and transition dynamic described in the previous section, we obtain the evolution of the density function $f_\lambda(x,t)$ which is described by  the integro-differential equation of Boltzmann-type 
\begin{equation}\label{eq:boltz_lin}
	\partial_t f_\lambda(x,t)  -\T[f_\lambda](x,t)= \sum_{ \lambda_*\in\{0,1\}} Q(f_\lambda,f_{\lambda_*})(x,t),
\end{equation}
where $\T[\cdot]$ is the transition operator and $Q(\cdot,\cdot)$ is the binary interaction operator defined as follows
\begin{equation}	\label{eq:Bo}
	Q(f_\lambda,f_{\lambda_*}) =\eta  \int_{\mathbb{R}^{2d}}\left(\dfrac{1}{J}f_\lambda(x',t)f_{\lambda_*}(x_*',t)-f_\lambda(x,t)f_{\lambda_*}(x_*,t)\right)dx dx_*, 
\end{equation}
where $(x',x_*')$ are the pre-interaction positions generated by the couple $(x,x_*)$ after the interaction \eqref{eq:bin_2pop}. The term $J$ denotes the Jacobian of the transformation $(x,x_*)\rightarrow (x',x_*')$ and  $\eta>0$ is a constant relaxation rate representing the interaction frequency.
To obtain a weak-formulation, we consider a test function $\phi(x)$  and rewrite the collision operator  
\begin{equation}\begin{split}\label{eq:collisional_op}
		\int_{\RR^{d}} & Q(f_\lambda,f_{\lambda_*})(x,t)\phi(x)dx =\eta  \int_{\RR^{2d}}\ \left(\phi(x')-\phi(x)\right) f_{\lambda_*}(x_*,t) f_\lambda(x,t) dx dx_* .
\end{split}\end{equation} 
Hence, the weak form of \eqref{eq:boltz_lin} reads
\begin{align}\label{weak_formulation}
			\frac{\partial}{\partial t}\int_{\mathbb{R}^d} f_\lambda(x,t)& \phi(x) dx-\int_{\mathbb{R}^d} \mathcal{T}[f_\lambda](x,t) \phi(x) dx = \nonumber \\ &\eta \sum_{\lambda_*\in\{0,1\}} \left\langle \int_{\mathbb{R}^{2d}}\Big[\phi(x')-\phi(x)\Big] f_\lambda(x,t) f_{\lambda_*}(x_*,t) dx dx_* \right\rangle.
\end{align}
To simplify the computations, we assume to have constant transition rates \eqref{eq:constant_rates} and to be in the quasi-stationary state $\rho_\lambda^\infty$ i.e. $\rho_\lambda \approx  \rho_\lambda^\infty$ for any $\lambda\in\{0,1\}$ as in \eqref{eq:stationary}. Moreover, we introduce the scaling parameter $\varepsilon > 0$ and consider
\begin{equation}\label{eq:scaling}
	\begin{split}
		&\nu_F \rightarrow \frac{\nu_F}{\rho_1}\varepsilon,\qquad \nu_L \rightarrow \frac{\nu_L}{\rho_1}\varepsilon, \qquad \sigma_F \rightarrow \frac{\sigma_F}{\sqrt{\rho_1}}\sqrt{\varepsilon}, \qquad \eta \rightarrow \frac{1}{\varepsilon}.
	\end{split}
\end{equation}
This scaling corresponds to the case where the interaction kernel concentrates on binary interactions producing very small changes in the agents position but at the same
time the number of interactions becomes very large.

To obtain the mean-field equation, we consider the Taylor expansion of the test function $\phi(x')$ centred in $x$ given by
\begin{equation}
	\phi(x')-\phi(x) =\nabla_x\phi(x)\cdot  (x'-x)  + \frac{1}{2} \Delta_x \phi(x)  (x'-x)^2 + \mathcal{O}(\varepsilon^2),
\end{equation}
and use it to rewrite \eqref{weak_formulation} as follows
\begin{align}\label{eq:weak_formulation_foll}
		&\frac{\partial}{\partial t}\int_{\mathbb{R}^d} f_\lambda(x,t)\phi(x) dx -\int_{\mathbb{R}^d} \mathcal{T}[f_\lambda](x,t) \phi(x) dx =\nonumber \\
		& \sum_{\lambda_*\in\{0,1\}} \left\lbrace  \int_{\RR^{2d}} \left( \frac{\nu_F}{\rho_1} (x_*-x)  \lambda_* (1-\lambda) +\frac{\nu_L}{ \rho_1} (\hat{x}(t)-x) \lambda \right)\cdot \nabla_x \phi(x) df_\lambda df_{\lambda_*}\right.\nonumber \\
		&+\left.\frac{\sigma^2_F}{2} \int_{\RR^{2d}} D^2(x)(1-\lambda)^2\lambda_*^2 \Delta_x \phi(x) df_\lambda df_{\lambda_*} \right\rbrace + \mathcal{O}(\varepsilon),
\end{align}
where for simplicity we write  $df_\lambda = f_\lambda(x,t) dx $ and $df_{\lambda_*} = f_{\lambda_*}(x_*,t)dx_*$. 
Now, taking  the limit $\varepsilon \to 0$, integrating by parts and rewriting the equation in strong form yields
\begin{align}\label{eq:boltz_strong_lead_foll}
		& \frac{\partial}{\partial t}  f_0(x,t) - \mathcal{T}[f_0](x,t)
		= \frac{\sigma_F^2}{2} \Delta_x\Bigl[D^2(x) f_0(x,t)\Bigr] - \nu_F \nabla_x \cdot \Bigl[\Bigr(\frac{m_1(t)}{\rho_1} -x \Bigl) f_0(x,t)\Bigr], \nonumber \\&
		\frac{\partial}{\partial t} f_1(x,t) -  \mathcal{T}[f_1 ](x,t) = -\frac{\nu_L}{ \rho_1}   \nabla_x \cdot \Bigl[\Bigr(\hat{x}(t) -x \Bigl) f_1(x,t)\Bigr],
\end{align}
where $D(x)$ is the diffusion matrix defined in \eqref{eq:diffusion_ani}-\eqref{eq:diffusion_iso}, $\hat{x}(t)$ is the global estimate of the global minimizer at time $t$ defined in equation \eqref{eq:x_tot} and
\begin{equation}\label{eq:mean_leaders} 
	m_1(t) = \int_{\RR^{2d}} x f_1(x,t) dx
\end{equation}
denotes the centre of mass of the leaders at time $t$.
\begin{remark}
	Multiplying both side of the second equation in   \eqref{eq:boltz_strong_lead_foll} by  $x/\nu_L$ integrating and taking the formal limit $\nu_L\to + \infty$, we get 
	\[
	\frac{m_1(t)}{\rho_1} = \hat{x}(t).
	\]
	Plugging it into the first equation in \eqref{eq:boltz_strong_lead_foll}, assuming $\mathcal{T}[f_0](x,t)=0$, we recover the equation that governs the dynamics in absence of leaders that is 
	\begin{equation}\label{eq:boltz_strong_NOlead}
		\frac{\partial}{\partial t}  f_0(x,t) 
		= \frac{\sigma_F^2}{2} \Delta_x \Bigl[D(x)^2 f_0(x,t)\Bigr] - \nu_F \nabla_x \cdot \Bigl[\Bigr(\hat{x}(t) -x \Bigl) f_0(x,t)\Bigr].
	\end{equation}
\end{remark}

To summarize, the diagram in Figure \ref{fig:diagram} describes the relation between the three algorithms at the particle and mean field level.	
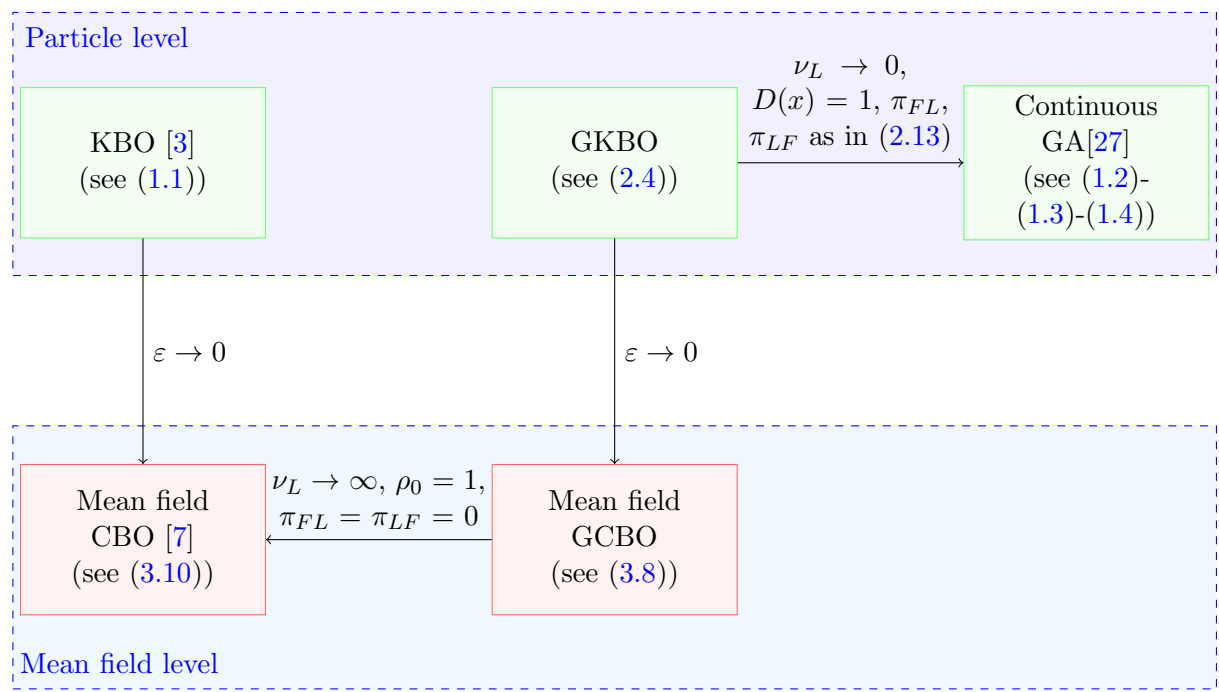
\begin{figure}[!htb] 
	\centering
\scalebox{.7}	{\begin{tikzpicture}[ 
		greenode/.style={rectangle, draw=green!60, fill=green!5, very thick, minimum size=7mm},rednot/.style={rectangle, draw=red!60, fill=red!5, very thick, minimum size=7mm},
		]
		\tikzstyle{greenode} = [ rectangle, 
		minimum width=2.5cm, 
		minimum height=2cm, 
		text centered, 
		text width=3cm, 
		draw=green!60, 
		fill=green!5,node distance=3cm]
		\tikzstyle{rednode} = [ rectangle, 
		minimum width=2.5cm, 
		minimum height=2cm, 
		text centered, 
		text width=3cm, 
		draw=red!60, 
		fill=red!5,node distance=3cm]
		\draw[blue,dashed,fill=blue!6] (-8,-1.5) rectangle (8,2) node[yshift=-2ex, xshift =- 91ex] {Particle level} ;
		\draw[blue,dashed,fill=cyan!6] (-8,-7) rectangle (8,-3.5 ) node[yshift=-20.5ex, xshift =- 88ex] {Mean field level} ;
		\node[greenode]      (GKBO)      {GKBO \\ (see \eqref{eq:bin_2pop})};
		\node[greenode]        (GA)       [right=of GKBO] {Continuous GA\cite{toledo2014global}\\(see \eqref{eq:GA}-\eqref{eq:GA_1}-\eqref{eq:GA_mod})};
		\node[greenode]     (KBO)       [left=of GKBO] {KBO \cite{benfenati2022binary} (see \eqref{eq:bin_1pop})};
		\node[rednode]        (MFCBO)       [below=of KBO] {Mean field
			
			CBO \cite{carrillo2018analytical} (see \eqref{eq:boltz_strong_NOlead})};
		\node[rednode]        (MFGCBO)       [below=of GKBO] {Mean field GCBO \\(see \eqref{eq:boltz_strong_lead_foll})};
		\draw[<-] (GA.west) -- (GKBO.east)node[text width=3cm,	text centered,midway,above ]   {$\nu_L \to 0$, $D(x)=1$, 
			$\pi_{FL}$, $\pi_{LF}$ as in \eqref{eq:rates_test_1}};
		\draw[<-] (MFCBO.north) -- (KBO.south) node[ pos=0.5, anchor=west]  {$\varepsilon \to 0$ };
		\draw[<-] (MFGCBO.north) -- (GKBO.south) node[ pos=0.5, anchor=west]  {$\varepsilon \to 0$ };
		\draw[<-] (MFCBO.east) -- (MFGCBO.west)node[text width=3cm,	text centered,midway,above ]   {$\nu_L \to \infty$, $\rho_0 = 1$,  
			
			$\pi_{FL} = \pi_{LF} = 0$};
	\end{tikzpicture}}
\vspace*{8pt}
	\caption{Diagram describing the relation between the KBO, the GKBO with weighted and random strategies and the GA. }
	\label{fig:diagram} 
\end{figure}

\section{Moments estimates and convergence to the global minimum}\label{sec:moments_convergence}
Following the idea introduced in \cite{benfenati2022binary,pinnau2017consensus} we provide moments estimates, showing that the variance decreases exponentially to zero, and we prove the convergence of the method toward the position of the global minimum. In this section we study the behaviour of the two population dynamic, we therefore assume throughout this section $\rho_0, \rho_1 >0$.
\subsection{Evolution of the moment estimates}\label{sec:moment} 
We define the first two moments of the total population by 
\begin{equation}\label{eq:2moments}
	m(t) = m_0(t) + m_1(t),\qquad e(t) = e_0(t) + e_1(t),
\end{equation}
respectively, where
\begin{align}\label{eq:moments_F_L_bis}
		&m_0(t) =  \int_{\mathbb{R}^{d}}x ~ f_0(x,t) dx,\qquad e_0(t) = \int_{\mathbb{R}^d}\vert x \vert^2 f_0(x,t) dx,\nonumber \\
		&m_1(t) =\int_{\mathbb{R}^{d}}x ~ f_1(x,t) dx,\qquad e_1(t) =  \int_{\mathbb{R}^d}\vert x \vert^2 f_1(x,t) dx,
\end{align}
are the first two moments of the subpopulations $f_\lambda(x,t)$ for $\lambda \in \{0,1\}$ and 
\begin{equation}\label{eq:variance}
	V(t) = v_0(t) + v_1(t),
\end{equation}	
the sum of the variances of the subpopulations given by 
\begin{equation}\label{eq:variance_F_L}
	v_0(t) =  \int_{\RR^{d}} \Big \vert x - \frac{m_0}{\rho_0} \Big \vert^2 f_0(x,t) dx, \qquad 	v_1(t) =\int_{\RR^{d}} \Big \vert x -\frac{m_1}{\rho_1} \Big \vert^2 f_1(x,t) dx.
\end{equation}
\begin{remark}
	For the following computations it is helpful to have in mind that
	$$m(t) = \int_{\RR^{d}} x (f_0(x,t) + f_1(x,t)) dx, \qquad e(t) = \int_{\RR^{d}} |x|^2 (f_0(x,t) + f_1(x,t)) dx,$$
	but due to the nonlinearity
	$$ V(t) \ne \int_{\RR^{d}} |x - m(t)|^2 (f_0(x,t) + f_1(x,t)) dx. $$
\end{remark}
\begin{proposition}\label{prop:1}Let us assume the transitions have equilibrated, that is, $\rho_0 \equiv \rho_0^{\infty}$ and $\rho_1 \equiv \rho_1^{\infty}$. 
	Furthermore let $\mathcal{E}(x)$ positive and bounded for all $x\in\mathbb{R}^d$, in particular, there exist constants $\underline{\mathcal{E}},\overline{\mathcal{E}} >0$ such that
	\begin{equation}\label{eq:assumption1}
		\underline{\mathcal{E}}:=\inf_x \mathcal{E}(x) \leq \mathcal{E}(x) \leq \sup_x \mathcal{E}(x) :=\overline{\mathcal{E}},
	\end{equation}
	and define $\tilde{\sigma} = k \sigma_F^2 b_{\underline{\mathcal{E}}}$, with  $b_{\underline{\mathcal{E}}} = \exp(\alpha(\bar{\mathcal{E}}-\underline{\mathcal{E}})),$ where $k=d$ in the case of isotropic diffusion and $k=1$ in the case of anisotropic diffusion.
	If 
	\begin{equation}\label{eq:condition_1} 
		\nu_F = \nu_L, \qquad \nu_F > \max \left\lbrace \frac{\tilde{\sigma}}{2}, \frac{\rho_1}{2}\right\rbrace,
	\end{equation}
	it holds
	\begin{align}\label{eq:estimates}
			&\frac{d}{dt}m(t) = \nu_F(\hat{x}-m)(t),\nonumber \\
			& \frac{d}{dt}V(t) \leq (-2\nu_F+\tilde{\sigma}) V(t) + (\tilde{\sigma} \rho_0\rho_1-\pi_{LF} \rho_1 -\pi_{FL}\rho_0 ) \left( \frac{m_0(t)}{\rho_0} - \frac{m_1(t)}{\rho_1}\right) ^2.
	\end{align} 
	
\end{proposition}
\begin{proof}
	Let us define for simplicity
	\begin{align}\label{eq:moments_F_L}
			M_\lambda(t) =  \frac{1}{\rho_\lambda}\int_{\mathbb{R}^{d}}x ~ &f_\lambda(x,t) dx,\qquad E_\lambda(t) =\frac{1}{\rho_\lambda} \int_{\mathbb{R}^d}\vert x \vert^2 f_\lambda(x,t) dx, \nonumber \\ &V_\lambda(t) = \frac{1}{\rho_\lambda} \int_{\RR^{d}} \Big \vert x -M_\lambda \Big \vert^2 f_\lambda(x,t) dx,
	\end{align}
	for any $\lambda \in \{0,1\}$	such that
	\begin{equation*}
		\begin{split}
			m(t) = \rho_0 M_0(t) &+ \rho_1 M_1(t),\qquad e(t) = \rho_0 E_0(t) + \rho_1 E_1(t),\\
			& V(t) = \rho_0 V_0(t) + \rho_1 V_1(t).
		\end{split}
	\end{equation*}
	We begin by computing the evolution of the first moments
	\begin{equation}\label{eq:proof_mean} 
		\frac{d}{dt}m(t) = \rho_0 \frac{d}{dt}M_0(t) + \rho_1 \frac{d}{dt} M_1(t).
	\end{equation}
	For the first term of \eqref{eq:proof_mean} we obtain  
	\begin{align}\label{eq:first_term_proof_mean}
			\rho_0 \frac{d}{dt} M_0(t) = & \int_{\RR^{d}} x \partial_{t} f_0 =\nonumber  \\
			& =\int_{\RR^{d}} x \Big[ -\pi_{LF} f_1 + \pi_{FL} f_0 + \Big.\nonumber \\& -\Big.\nabla_x \cdot \left( \nu_F (M_1-x) f_0 \right)+ \frac{\sigma_F^2}{2}  \Delta_x\left(  D^2(x) f_0\right) \Big] dx = \nonumber \\& = -\pi_{LF} \rho_1 M_1(t) + \pi_{FL} \rho_0 M_0(t)  + \rho_0 \nu_F (M_1-M_0)(t).
	\end{align}
	and for the second term in \eqref{eq:proof_mean}  it holds
	\begin{align}\label{eq:second_term_proof_mean}
			\rho_1 \frac{d}{dt} M_1(t) = & \int_{\RR^{d}} x \partial_{t} f_1 =\int_{\RR^{d}} x \Big[ \pi_{LF} f_1 - \pi_{FL} f_0  -\nabla_x \cdot \left( \frac{\nu_L}{ \rho_1} (\hat{x}-x) f_1 \right)\Big] dx = \nonumber \\& = \pi_{LF} \rho_1 M_1(t) - \pi_{FL} \rho_0 M_0(t)  +  \nu_L  (\hat{x}-M_1)(t).
			\end{align}
	Together this yields
	\begin{equation}
		\frac{d}{dt}m(t) = \nu_F \rho_0 (M_1-M_0)(t) +\nu_L  (\hat{x}-M_1)(t),
	\end{equation}
	and recalling the definition of $M_0(t)$ and $M_1(t)$ in \eqref{eq:moments_F_L} we get
	\begin{equation}
		\frac{d}{dt} m(t) = \nu_L \hat{x}(t) - \nu_F m(t) + \left( \nu_F \frac{\rho_0}{\rho_1} + \nu_F -\frac{\nu_L}{\rho_1} \right) m_1(t). 
	\end{equation}
	By the first assumption in \eqref{eq:condition_1} we can recover
	the first equation of the statement. 
	
	For $V(t)$ we have
	\begin{equation}\label{eq:proof_variance} 
		\frac{d}{dt} V (t) = \rho_0 \frac{d}{dt}V_0 (t)+ \rho_1 \frac{d}{dt}V_1 (t).
	\end{equation}
	We investigate the terms separately. First, we obtain 
	\begin{align}\label{eq:variance_V0} 
			\frac{d}{dt}V_0(t) = &\frac{1}{\rho_0} \frac{d}{dt} \int_{\RR^{d}} \vert x - M_0(t)\vert^2 df_0 =\nonumber \\& \underbrace{\frac{2}{\rho_0} \int_{\RR^{d}} \left( x-M_0(t), -\frac{d}{dt}M_0(t)\right) df_0}_{=: I_0} + \underbrace{\frac{1}{\rho_0} \int_{\RR^{d}} \vert x - M_0(t)\vert^2 \partial_t f_0}_{=: I_1}.
	\end{align}
	We note that $I_0$ vanishes, 
	since $2\rho_0^{-1} \int_{\RR^{d}} x - M_0(t) df_0 =0$.
	We divide $I_1$ into its drift, diffusion and transition parts to obtain
	\begin{equation}
		I_1 =: I_1^0 + I_1^1 + I_1^2, 
	\end{equation}
	with 
	\begin{align}
			I_1^0 = &\frac{1}{\rho_0} \int_{\RR^{d}} \vert x-M_0(t)\vert^2 \left(-\nu_F \nabla_x \cdot \left( \left( 	M_1(t)-x\right) f_0\right)\right)  dx 
			=\nonumber  \\& =\frac{2 \nu_F}{\rho_0} \int_{\RR^{d}} (x-M_0(t)) (M_1(t)-x) df_0 = \nonumber \\&= 2\nu_F \left( M_0(t) M_1(t) - E_0(t) - M_0(t)M_1(t) +M_0^2(t)\right)  = -2\nu_F V_0,
	\end{align}
	and, by an application of Jensen inequality we get
	\begin{align}
			I_1^1 = &\frac{1}{\rho_0}\int_{\RR^{d}} \vert x-M_0(t)\vert^2 \Delta_x\left(  \frac{\sigma_F^2}{2} D^2(x) f_0\right) dx =\nonumber \\& = \frac{\sigma_F^2}{2\rho_0} \int_{\RR^{d}} k \vert \hat{x}(t)-x\vert^2 df_0 = \nonumber\\& =\frac{\sigma_F^2}{2\rho_0} \int_{\RR^{d}} k 
			\int_{\RR^d} \Big \vert \frac{\int_{\RR^{d}}(y-x) e^{-\alpha \mathcal{E}(y)} g(y) dy}{\int_{\RR^{d}} e^{-\alpha \mathcal{E}(y)} g(y) dy}  \Big \vert^2 df_0 =\nonumber\\& \leq \frac{\tilde{\sigma}}{\rho_0} \int_{\RR^{2d}} \vert y-x \vert^2 g(y) f_0(x) dx dy =\nonumber \\ & = \frac{\tilde{\sigma}}{\rho_0} \left( \rho_0 E_0(t)+\rho_1 E_1(t) - (\rho_0 M_0(t) + \rho_1 M_1(t))^2\right) =\nonumber \\&= \frac{\tilde{\sigma}}{\rho_0} \left( V(t) + \rho_0\rho_1 \left( M_0(t)-M_1(t)\right) ^2\right),
	\end{align}
	and finally, 
	\begin{align}
			I_1^2 = &\frac{1}{\rho_0} \int_{\RR^{d}} \vert x-M_0(t)\vert^2 \left( -\pi_{LF}f_1 + \pi_{FL} f_0\right) dx = \nonumber\\& -\frac{\pi_{LF}}{\rho_0} \int_{\RR^{d}} \left( x^2+M_0^2(t)-2xM_0(t)\right) df_1 + \pi_{FL} V_0(t)  = \nonumber \\& -\frac{\pi_{LF}}{\rho_0}\left( \rho_1 V_1(t) + \rho_1 (M_0-M_1)^2(t)\right) +\pi_{FL}V_0(t),
	\end{align}
	where we add and subtract $\rho_1 M_1^2(t)$ in the last term of $I_1^2$.
	
	For $V_1(t)$  we have 
	\begin{align}\label{eq:variance_V1} 
			\frac{d}{dt}V_1(t) = &\frac{1}{\rho_1} \frac{d}{dt} \int_{\RR^{d}} \vert x - M_1(t)\vert^2 df_1 = \nonumber \\& \underbrace{\frac{2}{\rho_1} \int_{\RR^{d}} \left( x-M_1(t), -\frac{d}{dt}M_1(t)\right) df_0}_{=: I_2} + \underbrace{\frac{1}{\rho_1} \int_{\RR^{d}} \vert x - M_1(t)\vert^2 \partial_t f_1}_{=: I_3}.
		\end{align}
	Similarly to the case $I_0$ one can easily conclude that $I_2$ vanishes.
	We divide $I_3$ into the drift and transition part to obtain
	\begin{equation}
		I_3 = I_3^0 + I_3^1, 
	\end{equation}
	with 
	\begin{align}
			I_3^0 = &\frac{1}{\rho_1} \int_{\RR^{d}} \vert x-M_1(t)\vert^2 \left(- \frac{\nu_L}{  \rho_1} \nabla_x \cdot \left( \left( 	\hat{x}(t)-x\right) f_1\right)\right)  dx 
			=\nonumber \\& =\frac{2 \nu_L}{ \rho^2_1} \int_{\RR^{d}} (x-M_1(t)) (\hat{x}(t)-x) df_1 = \nonumber\\&= \frac{2 \nu_L}{ \rho_1} \left( M_1(t) \hat{x}(t) - E_1(t) - M_1(t)\hat{x}(t) +M_1^2(t)\right)  = -2\frac{ \nu_L}{ \rho_1} V_1,
	\end{align}
	and 
	\begin{align}
			I_3^1 = &\frac{1}{\rho_1} \int_{\RR^{d}} \vert x-M_1(t)\vert^2 \left( \pi_{LF}f_1 - \pi_{FL} f_0\right) dx =\nonumber\\& =  -\frac{\pi_{FL}}{\rho_1} \int_{\RR^{d}} \left( x^2+M_1^2(t)-2xM_1(t)\right) df_0 + \pi_{LF} V_1(t)  = \nonumber\\& = -\frac{\pi_{FL}}{\rho_1}\left( \rho_0 V_0(t) + \rho_0 (M_0-M_1)^2(t)\right) +\pi_{LF}V_1(t),
	\end{align}
	where we add and subtract $\rho_0 M_0^2(t)$ in the last term of $I_3^2$.
	
	Altogether, we get
	\begin{align}\label{eq:proof_V} 
			\frac{d}{dt}V(t)\leq& \rho_1 \left( 2\nu_F +\rho_0-\frac{2\nu_F}{\rho_1} \right) V_0(t) +\left( -2\nu_F +\tilde{\sigma} \right) V(t) +\nonumber\\
			&+ \left( \tilde{\sigma} \rho_0 \rho_1 -\pi_{LF} \rho_1 -\pi_{FL} \rho_0\right) \left( M_0-M_1\right)^2(t).
	\end{align}
	Using the assumptions, we recover the second inequality in \eqref{eq:estimates}.
\end{proof}

\begin{corollary}\label{corollary_1}
	Let the assumptions of Proposition \ref{prop:1} hold, and in addition suppose that 
	\begin{equation}\label{eq:conditions_cor} 
		\nu_F >  \max \left\lbrace \frac{\pi_{LF}\rho_1}{\rho_0 (1-b_{\bar{\mathcal{E}}}\rho_1)}, \frac{\pi_{FL}}{b_{\bar{\mathcal{E}}}\rho_0}\right\rbrace,  \qquad \text{ with } b_{\bar{\mathcal{E}}} = e^{\alpha(\underline{\mathcal{E}}-\bar{\mathcal{E}})}.
	\end{equation}
	Then it holds
	\begin{equation}
		\Big \vert \frac{m_0(t)}{\rho_0}-\frac{m_1(t)}{\rho_1}\Big \vert^2 \to 0, \quad V(t)\to 0, \qquad \text{ as }   t \to \infty.
	\end{equation}
	
\end{corollary}
\begin{proof}
	Let us first study the behavior of $\vert M_0-M_1\vert^2(t)$. 
	We have 
	\begin{align}\label{eq:derivative_M0M1}
			& \frac{d}{dt}  \vert M_0-M_1\vert^2(t) = 2 \left( M_0-M_1\right) (t) \frac{d}{dt} \left( M_0-M_1\right) (t) \leq \nonumber \\  
			&\leq -2\nu_F \vert M_0-M_1\vert^2(t) +2 (M_0-M_1)(t) \left( \mathcal{C}_{1} M_1(t) - \mathcal{C}_{0} M_0(t)\right) = \nonumber \\
			&=-2\nu_F \vert M_0-M_1\vert^2(t) -2 		\begin{bmatrix}
				M_0(t) \\ M_1(t)
			\end{bmatrix}^T
			\begin{bmatrix}
				\mathcal{C}_0 & -\mathcal{C}_0 \\
				-\mathcal{C}_1 & \mathcal{C}_1
			\end{bmatrix}
			\begin{bmatrix}
				M_0(t) \\ M_1(t)
			\end{bmatrix},
	\end{align}
	with
	\begin{equation}
		\mathcal{C}_0 = -\frac{\pi_{FL}-\nu_F \rho_0 b_{\bar{\mathcal{E}}}}{\rho_1},\qquad \mathcal{C}_1 = \frac{-\pi_{LF}\rho_1+\nu_F\rho_0(1-\rho_1 b_{\bar{\mathcal{E}}})}{\rho_0\rho_1},
	\end{equation}
	and we used equations \eqref{eq:first_term_proof_mean}-\eqref{eq:second_term_proof_mean} and the estimate
	\begin{equation}\label{eq:estimate_xhat}
		\hat{x}(t) = \frac{\int_{\mathbb{R}^d}x e^{-\alpha \mathcal{E}(x)}g(x,t)dx}{\int_{\mathbb{R}^d} e^{-\alpha \mathcal{E}(x)}g(x,t)dx} \geq \frac{e^{\alpha \underline{\mathcal{E}}}}{e^{\alpha \bar{\mathcal{E}}}} \int_{\mathbb{R}^{d}} x~ g(x,t) dx  :=  b_{\bar{\mathcal{E}}} ~ m(t).
	\end{equation}
	Note that, if condition \eqref{eq:conditions_cor} holds then $\mathcal{C}_0$, $\mathcal{C}_1>0$ and so
	\begin{equation}
		-2 		\begin{bmatrix}
			M_0(t) \\ M_1(t)
		\end{bmatrix}^T
		\begin{bmatrix}
			\mathcal{C}_0 & -\mathcal{C}_0 \\
			-\mathcal{C}_1 & \mathcal{C}_1
		\end{bmatrix}
		\begin{bmatrix}
			M_0(t) \\ M_1(t)
		\end{bmatrix}\leq 0,
	\end{equation}
	since the above matrix is weakly diagonal dominant and hence positive semidefinite.  
	Altogether, we obtain the estimate
	\begin{equation}\label{eq:derivative_M0M1_bis}
		\frac{d}{dt} \vert M_0-M_1\vert^2(t) \leq -2\nu_F\vert M_0-M_1\vert^2(t).
	\end{equation}   
	and an application of Gr{\"o}nwall lemma yields
	\begin{equation}\label{eq:sol_MOM1}
		\vert M_0-M_1 \vert^2(t) \leq 	\vert M_0-M_1 \vert^2(0) e^{-2\nu_F t},
	\end{equation}
	which allow us to conclude $\vert M_0-M_1 \vert^2(t) \to 0$ as $t\to \infty$. 
	\\
	\\
	In particular, this implies 
	\begin{equation}\label{eq:max_MOM1}
		\vert M_0-M_1 \vert^2(t) \leq \vert M_0-M_1 \vert^2(0), 
	\end{equation}
	which helps us to show the second statment. Indeed, we rewrite the second inequality in \eqref{eq:estimates} in integral form as \begin{equation}\label{eq:variance_integral_form}
		V(t) \leq V(0) + \mathcal{C}^0_v\vert M_0 -M_1\vert^2 (0) \int_{0}^{t}  ds- \mathcal{C}_v \int_{0}^{t}V(s)ds, 
	\end{equation} 
	with $\mathcal{C}^0_v =  \tilde{\sigma} \rho_0 \rho_1 -\pi_{LF} \rho_1 -\pi_{FL} \rho_0$ and $\mathcal{C}_v = 2\nu_F-\tilde{\sigma}$. Moreover, we note that \[
	t\to V(0)+\mathcal{C}^0_v \vert M_0 -M_1\vert^2 (0)t,
	\]
	is a non-decreasing function. Hence,
	again using Gr{\"o}nwall lemma, we get
	\begin{equation}\label{eq:sol_variance}
		V(t) \leq \Big[V(0)+\mathcal{C}^0_v \vert M_0 -M_1\vert^2 (0)t\Big] e^{-\mathcal{C}_vt},
	\end{equation}
	which implies $V(t)\to 0$ as $t\to \infty$. 
\end{proof}
The fact that $V(t)$ vanishes in the limit $t\rightarrow \infty$ allows us to conclude that the crowd concentrates. However, the position of the concentration point is unknown. This position is quantified in the following section.

\subsection{Convergence to the global minimum} \label{sec:convergence} 
In this section, we determine the conditions under which the mean value of the population is a reasonable approximation of the global minimizer.
\begin{proposition}	\label{prop:2}
	Suppose the assumptions of Proposition \ref{prop:1} hold. Further, we assume that $\mathcal{E}\in C^2(\mathbb{R}^d)$ and that there exist constants $c_1,c_2 >0$ such that
	\begin{equation}\label{eq:constants} 
		\sup_{y\in \mathbb{R}^2} \vert \nabla \mathcal{E}(y)\vert \leq c_1, \qquad 	\sup_{y\in \mathbb{R}^2} \vert \Delta \mathcal{E}(y)\vert \leq c_2,
	\end{equation} 
	and that the initial condition is well-prepared in the sense that the minimizer of $\mathcal E$ is in the support of the initial population and
	\begin{equation}\label{eq:assumption_mu} 
		\frac{\mu}{M_\alpha^2(0)}\leq \frac{3}{4},
	\end{equation}
	is satisfied  with 
	\begin{equation}
		M_\alpha(t) = \int_{\RR^{d}} e^{-\alpha \mathcal{E}(x)} g(x) dx,
	\end{equation}
	and
	\begin{align}\label{eq:mu}
			\mu =& 2 \alpha e^{-\alpha \underline{\mathcal{E}}} \left[ c_1 \sqrt{2} \left( \nu_F + \frac{\nu_F}{\rho_1}  \right) + c_2 \sigma_F^2 k  \right] \cdot \nonumber \\
			&\cdot \left[ \max \left\lbrace \frac{1}{\mathcal{C}_v} V(0)+ \gamma_m \left( \frac{\mathcal{C}_v^0}{\mathcal{C}_v^2} + \frac{\rho_0 \rho_1}{2\nu_F}\right), \frac{2}{\mathcal{C}_v} V(0) + \frac{4 \mathcal{C}^*}{\mathcal{C}_v} + \frac{\sqrt{\rho_0\rho_1\gamma_m}}{\nu_F} \right\rbrace  \right],
	\end{align}
	with 
	\[\mathcal{C}^0_v =  \tilde{\sigma} \rho_0 \rho_1 -\pi_{LF} \rho_1 -\pi_{FL} \rho_0,\qquad  \mathcal{C}_v = 2\nu_F-\tilde{\sigma}, \qquad 
	\gamma_m = \left( \frac{m_0(0)}{\rho_0}-\frac{m_1(0)}{\rho_1}\right) ^2,
	\]
	and	$\mathcal{C}^*$  is the maximal value of
	\begin{equation*}
		t\to e^{-\frac{\mathcal{C}_v}{4}t} \sqrt{\mathcal{C}_v^0 \gamma_m t}.
	\end{equation*}
	Then there exists $ \tilde{x} \in \R^d$ such that $m(t) \to \tilde{x}$ as $t \to \infty$ and
	\begin{equation}
		\mathcal{E}(\tilde{x}) = \underline{\mathcal{E}} + r(\alpha) + \frac{\log{2}}{\alpha},
	\end{equation}
	where $r(\alpha) = -\frac{1}{\alpha} \log (M_\alpha(0)) - \underline{\mathcal{E}}\to 0$,  as $\alpha \to \infty$. 
\end{proposition}
\begin{proof}
	First, we show
	\begin{equation}
		\Big	\vert \frac{d}{dt}m(t) \Big \vert \to 0, 
	\end{equation}
	as $t\to \infty$. 
	To this end, we rewrite
	\begin{equation}\label{eq:proof_m_1} 
		\Big\vert \frac{d}{dt}m(t) \Big\vert =  \Big\vert\nu_F \int_{\RR^{d}} \left( \frac{e^{-\alpha \mathcal{E}(x)}}{M_\alpha(t)} -1 \right) x ~ g(x) dx \Big \vert,
	\end{equation}
	where we use the first estimate in \eqref{eq:estimates} and the definition of $\hat{x}(t)$. Applying Jensen inequality and using the estimate
	\begin{equation*}
		\hat{x}(t)\leq e^{-\alpha(\underline{\mathcal{E}}-\bar{\mathcal{E}})}\int_{\RR^{d}} x g(x,t) dx := b_{\underline{\mathcal{E}}} m(t),
	\end{equation*}
	we get
	\begin{align}\label{eq:proof_m_2}
			\Big\vert \frac{d}{dt}m(t) \Big\vert & = \frac{\nu_F}{M_\alpha(t)}\Big\vert \int_{\RR^{2d}} x e^{-\alpha\mathcal{E}(x)} g(x) g(x_*) dx dx_* - \int_{\RR^{2d}} x_* e^{-\alpha\mathcal{E}(x)} g(x) g(x_*) dx dx_* \Big \vert = \nonumber \\ 
			& = \frac{\nu_F}{M_\alpha(t)}\Big\vert \int_{\RR^{2d}} (x-x_*) e^{-\alpha\mathcal{E}(x)} g(x) g(x_*) dx dx_* \Big \vert \nonumber \\
			&\leq \frac{\nu_F}{M_\alpha(t)} \int_{\RR^{2d}} \vert x-x_* \vert e^{-\alpha\mathcal{E}(x)} g(x) g(x_*) dx dx_*\leq \nonumber \\
			& \leq b_{\underline{\mathcal{E}}} \nu_F \left( \int_{\RR^{2d}} \vert x-x_*\vert^2 g(x) g(x_*) dx dx_*\right) ^{1/2} = \nonumber \\
			& = b_{\underline{\mathcal{E}}} \nu_F \sqrt{2} \left( \rho_0 E_0(t)+\rho_1 E_1(t) - (\rho_0 M_0(t) + \rho_1 M_1(t))^2 \right)^{1/2} = \nonumber \\ 
			&=  b_{\underline{\mathcal{E}}} \nu_F \sqrt{2} \left( V(t) + \rho_0 \rho_1 (M_0-M_1)^2(t) \right)^{1/2} \to 0,\qquad \text{ as } t \to \infty,
	\end{align} 
	since both, $V(t)$ and $(M_0-M_1)^2(t)$, go to zero as $t\to \infty$. 
	Thus, there exists $\tilde{x}\in \RR^d $ such that 
	\begin{equation}\label{eq:xtilde}
		\tilde{x} = m(0) + \int_{0}^{t} \frac{d}{ds}m(s)ds = \lim_{t\to \infty} m(t).
	\end{equation}
	\\ 
	\\
	Let us now focus on the term $M_\alpha(t)$
	\begin{equation}\label{eq:proof_Malpha} 
		\frac{d}{dt} M^2_\alpha (t) = 2 M_\alpha(t) \frac{d}{dt} M_\alpha(t) =  2 M_\alpha(t) \int_{\RR^{d}} e^{-\alpha \mathcal{E}(x)} \partial_t g(x,t) dx,
	\end{equation}
	with 
	\begin{align}
			\partial_{t} g(x,t) &= \partial_t f_0(x,t) + \partial_{t} f_1(x,t) = 
			-\nu_F \nabla_x \cdot \Big[ (M_1 - x) f_0(x,t) \Big] +\nonumber \\ &+ \frac{\sigma_F^2}{2} \Delta_x \Big[D^2(x) f_0(x,t)\Big] - \frac{\nu_F}{\rho_1} \nabla_x \cdot \Big[ (\hat{x}(t)-x)f_1(x,t)\Big],
	\end{align}
	where we recall that we assume $\nu_L= \nu_F$ and
	\[\sum_{ \lambda_\in\{0,1\}} \mathcal{T}[f_\lambda](x,t) = 0.\] 
	
	We consider the terms separately to obtain 
	\begin{align}\label{eq:I_1}
			I_1 = &-\nu_F \int_{\RR^{d}} e^{-\alpha \mathcal{E}(x)} \nabla_x \cdot \Big[(M_1-x) f_0 \Big] dx = \nonumber\\
			& = -\nu_F \alpha \int_{\RR^{2d}} e^{-\alpha \mathcal{E}(x)} \nabla \mathcal{E}(x) (x_*-x) df_0 df_1\geq \nonumber \\
			&\geq -\nu_F \alpha e^{-\alpha \underline{\mathcal{E}}} c_1 \frac{M_\alpha(t)}{M_\alpha(t)} \int_{\RR^{2d}} \vert x_*-x \vert dg dg_* \geq \nonumber \\
			& \geq -\nu_F \alpha \frac{e^{-2\alpha \underline{\mathcal{E}}}}{M_\alpha(t)} c_1 \left( \int_{\RR^{2d}} \vert x_*-x \vert^2 dg dg_* \right)^{1/2}\geq \nonumber \\
			& \geq   -\nu_F \alpha \frac{e^{-2\alpha \underline{\mathcal{E}}}}{M_\alpha(t)} c_1 \sqrt{2} \left[V(t) + \rho_0 \rho_1 (M_0-M_1)^2(t)\right]^{1/2},
	\end{align}
	\begin{align}\label{eq:I_2}
			I_2 = & \frac{\sigma_F^2}{2} \int_{\RR^{d}} e^{-\alpha \mathcal{E}(x)} \Delta_x \Big[ D^2(x) f_0\Big]dx = \nonumber \\
			& =- \frac{\sigma_F^2}{2} \alpha \int_{\RR^{d}} e^{-\alpha \mathcal{E}(x)}  \Delta \mathcal{E}(x) k \vert \hat{x}(t) -x \vert^2 df_0  + \nonumber\\
			& +\frac{\sigma_F^2}{2}\alpha^2\int_{\RR^{2d}} e^{-\alpha \mathcal{E}(x)} \nabla_x \mathcal{E}(x) \otimes  \nabla_x \mathcal{E}(x)  k \vert \hat{x}(t) -x \vert^2 df_0  \geq\nonumber \\
			& \geq -\frac{\alpha \sigma_F^2}{2} k c_2 e^{-\alpha \underline{\mathcal{E}}} \int_{\RR^{d}} \vert \hat{x}(t)-x \vert ^2 dg \geq \nonumber\\
			& \geq  -\frac{\alpha \sigma_F^2}{2} k c_2 e^{-\alpha \underline{\mathcal{E}}} \int_{\RR^{2d}} \int \vert x_*-x\vert^2 \frac{e^{-\alpha \mathcal{E}(x_*)}}{M_\alpha(t)} dg dg_*\geq \nonumber\\
			&  \geq   -\frac{\alpha \sigma_F^2}{2} k c_2 \frac{e^{-2\alpha \underline{\mathcal{E}}}}{M_\alpha(t)} \int_{\RR^{2d}} \vert x_*-x \vert^2 dg dg_* \geq  \nonumber \\
			& \geq  -\frac{\alpha \sigma_F^2}{2} k c_2 \frac{e^{-2\alpha \underline{\mathcal{E}}}}{M_\alpha(t)} \Big[ V(t) + \rho_0 \rho_1 (M_0-M_1)^2(t)\Big],
	\end{align}
	and 
	\begin{align}\label{eq:I_3}
			I_3 = & -\frac{\nu_F}{\rho_1} \int_{\RR^{d}} e^{-\alpha \mathcal{E}(x)} \nabla_x \cdot \Big[(\hat{x}(t)-x) f_1 \Big] dx \geq \nonumber \\
			& \geq  -\frac{\alpha \nu_F}{\rho_1}  c_1 e^{-\alpha \underline{\mathcal{E}}} \int  \vert \hat{x}(t)-x \vert dg \geq \nonumber \\
			& \geq -\frac{\alpha \nu_F}{\rho_1}  c_1 \frac{e^{-2\alpha \underline{\mathcal{E}}}}{M_\alpha(t)} \left( \int_{\RR^{2d}} \vert x_* - x \vert ^2 dg dg_*\right)^{1/2}\geq \nonumber \\
			& \geq  -\frac{\alpha \nu_F}{\rho_1}  c_1 \frac{e^{-2\alpha \underline{\mathcal{E}}}}{M_\alpha(t)}\Big[ V(t) + \rho_0 \rho_1 (M_0-M_1)^2(t)\Big]^{1/2},
	\end{align}
	where we use assumption \eqref{eq:constants}, we integrate by parts, use Jensen inequality and the previous estimates. 
	Altogether, we estimate \eqref{eq:proof_Malpha} as follows 
	\begin{align}\label{eq:sol_Malpha}
			\frac{d M_\alpha(t)}{dt} &\geq -2\alpha e^{-2\alpha \underline{\mathcal{E}}} \Big[ c_1 \sqrt{2}\nu_F \left(1+ \frac{1}{\rho_1}\right) \left(V(t) + \rho_0 \rho_1 (M_0-M_1)^2(t) \right)^{1/2} + \Big. \nonumber \\ & + \Big.c_2 \sigma^2_F k   \left(V(t) + \rho_0 \rho_1 (M_0-M_1)^2(t) \right)\Big].
	\end{align}
	Using the estimates for the mean and variance in \eqref{eq:sol_MOM1}-\eqref{eq:sol_variance} and integrating equation \eqref{eq:sol_Malpha} we get
	\begin{align}\label{eq:M_alpha2} 
			M^2_\alpha(t)\geq& M^2_\alpha(0)-2\alpha e^{-\alpha \underline{\mathcal{E}}} \left[ c_1 \sqrt{2}\nu_F \left(1+ \frac{1}{\rho_1}\right) + c_2  \sigma_F^2k \right]\cdot \nonumber\\\cdot\max&\left\lbrace \int_{0}^t \left[  V(0) + \mathcal{C}_v^0\gamma_m s \right] e^{-\mathcal{C}_vs} + \rho_0\rho_1 \gamma_m e^{-2\nu_Fs}ds,\right.\nonumber \\
			&\left.\int_{0}^t\sqrt{ \left[  V(0) + \mathcal{C}_v^0\gamma_m s \right] e^{-\mathcal{C}_vs} + \rho_0\rho_1 \gamma_m e^{-2\nu_Fs} }ds\right\rbrace.
	\end{align}
	We integrate the first integral in \eqref{eq:M_alpha2} by parts to get
	\begin{equation*}
		\begin{split}
			\int_{0}^t \left[  V(0) + \mathcal{C}_v^0\gamma_m s \right] e^{-\mathcal{C}_vs} + \rho_0\rho_1 \gamma_m e^{-2\nu_Fs}ds\leq 
			\frac{V(0)}{\mathcal{C}_v} + \gamma_m \left( \frac{\mathcal{C}_v^0}{\mathcal{C}_v^2} + \frac{\rho_0\rho_1}{2\nu_F} \right).
		\end{split}
	\end{equation*}
	Moreover, applying H{\"o}lder inequality to the second integral in \eqref{eq:M_alpha2} yields
	\begin{equation*} 
		\begin{split}
			&\int_{0}^t\sqrt{ \left[  V(0) + \mathcal{C}_v^0\gamma_m s \right] e^{-\mathcal{C}_vs} + \rho_0\rho_1 \gamma_m e^{-2\nu_Fs} }ds\leq\\
			&\leq \frac{2\sqrt{V(0)}}{\mathcal{C}_v} + \Vert \sqrt{\mathcal{C}_v^0\gamma_m s} e^{\frac{-\mathcal{C}_v s}{4}}\Vert_\infty \int_0^t e^{\frac{-\mathcal{C}_v s}{4}} ds + \frac{\sqrt{\rho_0 \rho_1 \gamma_m}}{\nu_F}\leq \\
			&\leq \frac{2\sqrt{V(0)}}{\mathcal{C}_v} + \frac{4\mathcal{C}^*}{\mathcal{C}_v} + \frac{\sqrt{\rho_0 \rho_1 \gamma_m}}{\nu_F},
		\end{split}
	\end{equation*}
	where $\mathcal{C}^* := \max\limits_{s\in \RR} \sqrt{\mathcal{C}_v^0\gamma_m s} e^{\frac{-\mathcal{C}_v s}{4}}$ and
	we use the fact that 
	\[
	\sqrt{a+b}\leq \sqrt{a} + \sqrt{b},
	\]
	for and $a,b\geq 0$.
	Altogether, using assumption \eqref{eq:assumption_mu} we obtain
	\begin{equation}\label{eq:M_alpha3} 
		M^2_\alpha(t)\geq M^2_\alpha(0)- \mu \geq \frac{1}{4} M^2_\alpha(0),
	\end{equation}
	with $\mu$ defined as in equation \eqref{eq:mu}.
	Thus 
	\begin{equation}\label{eq:M_alpha4} 
		M_\alpha(t)\geq \frac{1}{2} M_\alpha(0).
	\end{equation}
	In addition, since $m(t) \to \tilde{x}$ and $V(t)\to 0 $ as $t\to \infty$ it holds, 
	\begin{equation}\label{eq:M_alpha5}
		M_\alpha(t) =  \int_{\RR^{d}}e^{-\alpha \mathcal{E}(x)}g(x) dx \to e^{-\alpha \mathcal{E}(\tilde{x})},\end{equation}
	as $t\to\infty$ as a consequence of Chebishev inequality (see \cite{carrillo2018analytical}). 
	Thus 
	\begin{equation}
		\begin{split}
			0\geq e^{-\alpha \mathcal{E}(\tilde{x})} \geq \frac{1}{2}M_\alpha(0) \quad \Longleftrightarrow \quad 0  \geq -\alpha \mathcal{E}(\tilde{x}) \geq \log \left( \frac{M_\alpha(0)}{2}\right), 
		\end{split}
	\end{equation}
	that is 
	\begin{equation}
		0\leq \mathcal{E}(\tilde{x}) \leq -\frac{1}{\alpha} \log (M_\alpha(0)) + \frac{\log(2)}{\alpha}.
	\end{equation}
	Finally, $0 \leq  \mathcal{E}(\tilde{x}) \leq \underline{\mathcal{E}}$ as $\alpha \to \infty$, since the first term tends to $\underline{\mathcal{E}}$ thanks to Laplace principle and $\log(2)/\alpha$ vanishes in the limit. 
\end{proof}
\begin{remark} We emphasize the following observations:
	\begin{itemize}
		\item In order to satisfy condition \eqref{eq:assumption_mu}, $V(0)$ and $m(0)$ need to be small. 
		\item  Note that if we assume to have anisotropic diffusion the convergence is guaranteed independently of the parameters choice and, in particular, of the dimension $d$. For this reason, all
		numerical examples of the next section consider the anisotropic noise. 
	\end{itemize}
\end{remark}

\section{Numerical methods} \label{sec:num_methods}
In order to approximate the time evolution of the density $f_\lambda(x,t)$ we sample $N_s$ particles $(x_i^0,\lambda_i^0), i=1,\dots,N_s$ from the initial distribution. We consider a time interval $[0, T]$ discretized
in $N_t$ intervals of length $h$.
The interaction step is solved by means of binary interaction algorithms, see \cite{pareschi2001introduction,pareschi2013interacting} for details.   

We denote the approximation of $f_\lambda(x,nh)$ at time $t^n$ by $f_\lambda^{n}(x)$. For any $\lambda \in \{0,1\}$ fixed, the next iterate is given by
\begin{equation}\label{eq:collision}
	f_\lambda^{n+1}(x) =\left( 1-\frac{h}{\varepsilon} \right)  f_\lambda^n(x) + \frac{h}{\varepsilon} \sum_{\lambda_*\in\{0,1\}} Q_\alpha^+(f_\lambda^n,f_{\lambda_*}^n )(x),
\end{equation} 
where $\varepsilon>0$ is a frequency parameter and $Q^+(f_\lambda^n,f_{\lambda_*}^n)$ is the gain part of the collision operator defined in \eqref{eq:collisional_op}.
Equation \eqref{eq:collision} can be interpreted as follows: with probability $1-h /\varepsilon$ an individual in position $x$  does not interact with other individuals and with probability $h /\varepsilon$ it interacts with another randomly selected individual. In the following we will assume $h = \varepsilon$.   

In order to simulate changes of the label $\lambda$, we discretize equation \eqref{eq:Ldef2}. For any fixed $x\in \RR^{d}$, we obtain
\begin{align}\label{eq:lambda_evolution}
		f_0^{n+1}(x) = (1-\varepsilon ~\pi_{F\to L})~ f_0^n(x) + \varepsilon ~\pi_{L\to F} ~f_1^n(x),\nonumber\\
		f_1^{n+1}(x) = (1-\varepsilon ~\pi_{L\to F})~ f_1^n(x) + \varepsilon ~\pi_{F\to L}~ f_0^n(x),
\end{align}
where $\pi_{F\to L}(\cdot)$ and $\pi_{L\to F}(\cdot)$ are the transition rates as defined in \eqref{eq:rates_test_0}-\eqref{eq:rates_test_1}. The details of the numerical scheme are summarized in Algorithm \ref{alg_binary}. Here, the parameters $\delta_{stall}$ and $j_{stall}$ are used to check if consensus has been reached in the last $j_{stall}$ iterations within a tolerance $\delta_{stall}$. In more detail, we stop the iteration if the distance of the current and previous mean $\hat x$ is smaller then the tolerance $\delta_{stall}$ for at least $j_{stall}$ iterations. In this case, the evolution is stopped before the total number of iterations has been reached. 
\\
\begin{alg}~[GKBO]\label{alg_binary}
	\begin{enumerate}
		\item[\texttt 1.] Draw $(x_i^0,\lambda_i^0)_{i=1,\dots,N_s}$ from the initial distribution $f^0_\lambda(x)$ and set $n=0$, $j=0$.
		\item[\texttt 2.] Compute $\hat{x}^0$ as in equation \eqref{eq:x_tot}.
		\item[\texttt 3.] \texttt{while} $n<N_t$ \texttt{and} $j<j_{stall}$
		\begin{enumerate}
			\item \texttt{for} $i=1$ \texttt{to}  $N$
			\begin{itemize}
				\item Select randomly a leader with position $y^{n}_k$, $k\neq i$.
				\item Compute the new positions
				\begin{align}\label{eq:bin_mod_GA_alg}
					x_i^{n+1} &= x_i^n +\nu_F \varepsilon \left(   y_k^{n+1}-x_i^{n}\right) + \sigma_F \sqrt{\varepsilon} D \xi  \left( 1-\lambda_i^n\right) + \varepsilon\nu_L (\hat{x}^{n}-x_i^n)\lambda_i^n, \nonumber \\
						y^{n+1}_k &= y^n_k + \nu_L \varepsilon (\hat{x}^{n}-y^n_k).
				\end{align}
				\item Compute the following probabilities rates 
				\[
				p_{L} =\varepsilon ~\pi_{F\to L}(x_i^{n+1},\lambda_i^n), \qquad 	p_{F}=\varepsilon ~\pi_{L\to F}(x_i^{n+1},\lambda_i^n). 
				\]  
				\item \texttt{if} $\lambda_i^n = 0$,\\  with probability $p_{L}$ agents $i$ becomes a leader: $\lambda_i^{n+1} = 1$.
				\item \texttt{if} $\lambda_i^n = 1$,\\  with probability $p_{F}$ agents $i$ becomes a follower: $\lambda_i^{n+1} = 0$.
			\end{itemize}
			\texttt{end for}
			\item Compute $\hat{x}^{n+1}$ as in equation \eqref{eq:x_tot}.
			\item  \texttt{if} $\Vert \hat{x}^{n+1}-\hat{x}^n\Vert_\infty\leq \delta_{stall}$
			\begin{enumerate}
				\item [] $j\leftarrow j+1$
			\end{enumerate}
			\texttt{end if} 
			\item [] $n\leftarrow n+1$
		\end{enumerate}
		\texttt{end while} 
	\end{enumerate}
\end{alg}
The above algorithm is inspired from Nanbu's method\cite{nanbu1980direct}, for larger class of direct simulation Monte-Carlo algorithm for interacting particle dynamics we refer to\cite{albi2013binary,pareschi2013interacting}. 

\section{Validation tests}\label{sec:validation} 
In this section we test the performance of the GKBO algorithm in terms of success rate and number of needed iterations. We consider the translated Rastrigin function with global minimum in $\bar{x} = 1$ for the vast majority of the tests. In the last experiment we compare the results for different benchmark functions (see \cite{jamil2013literature} for a complete list).  If not explicitly specified, we run $M=20$ simulations and, according to \cite{benfenati2022binary,carrillo2018analytical}, we consider a simulation successful if 
\begin{equation}\label{eq:success} 
	\Vert \hat{x}(t) - \bar{x}\Vert_\infty \leq 0.25.
\end{equation}
We set $\alpha = 5\cdot 10^6$ and we adopt the numerical trick described in \cite{fornasier2021consensus} to allow for arbitrary large values of $\alpha$. We assume $N=200$ and that agents are initially uniformly distributed in the hypercube $[-4.12,0]^d$, which does not contain the global minimum.  At time $t=0$ we suppose all agents are in the followers status and they change their label according to equation \eqref{eq:Ldef2}. For the GKBO algorithm we set the total percentage of leaders is $\rho_1^\infty=0.5$, if not specified explicitly. Hence, the transition rates are defined as in equation \eqref{eq:rates_test_0}, with $\pi_{LF} = \pi_{FL} = 0.2$, if the emergence of leaders is random or as defined in equation \eqref{eq:rates_test_1} if the labels change according to the weighted criterion defined in Section \ref{sec:leaders}. We will consider also a mixed strategy with $\bar{p} =0.5$, that is, among the total amount of generated leaders, $50\%$ change their labels according to the weighted strategy and the remaining ones change their labels randomly. We let the dynamics in \eqref{eq:bin_2pop} to evolve for $N_t=10000$ iterations with $\varepsilon=0.1$, where differently specified. We set $j_{stall} = 1000$, $\delta_{stall}= 10^{-4}$.  We assume $\nu_F=1$, $\nu_L= 10$ while the diffusion parameter and the dimension change in the different tests and will be specified later.

\subsection{Test 1: Comparison of different followers / leaders ratios}\label{sec:test1} 
Suppose $\sigma_F=4$, $d=20$. Table \ref{tab:rastrigin_NLvaries_succ_rate_iterations} reports the mean of the number of iterations and success rate (in parenthesis) for the GKBO algorithm tested on the translated Rastrigin function as the leaders mass at the equilibrium $\rho_1^\infty$, defined as in equation \eqref{eq:stationary}, varies. The success rate and number of iterations for the KBO algorithm are $1$ and $10000$ respectively. GKBO outperforms KBO in terms of the number of iterations. However, the success rate of GKBO with random leader emergence deteriorates for $\rho_1^\infty = 0.75$.

\begin{table}[ht]\centering
	{\caption{Mean of the number of iterations (success rate) for the GKBO algorithm tested on the translated Rastrigin function as the leaders mass at the equilibrium $\rho^\infty_1$ varies.}
		\begin{tabular}{cccc} 
			& GKBO random & GKBO $\bar{p}=0.5$ & GKBO weighted \\ 
			$\rho_1^\infty=0.25$ & 3008 (1) & 3588 (1) &  6421 (1)\\
			$\rho_1^\infty=0.5$ & 2898 (1)  &  3477 (1) & 6612 (1) \\
			$\rho_1^\infty=0.75$ & 4252 (0.6)  & 4566 (1)  & 7741 (1) \\
	\end{tabular}
	\label{tab:rastrigin_NLvaries_succ_rate_iterations}}
\end{table}


\subsection{Test 2: GKBO for different choices of $\hat{x}$. }\label{sec:test2}
We compare the results of the algorithm with $\hat{x}$ as in \eqref{eq:x_tot} and slight modifications given by 
\begin{equation}\label{eq:x_tot_F} 
	\hat{x}_F =  \frac{\int_{\mathbb{R}^d}x e^{-\alpha \mathcal{E}(x)}f_0(x,t)dx}{\int_{\mathbb{R}^d} e^{-\alpha \mathcal{E}(x)}f_0(x,t)dx}, \qquad \text{or} \qquad
	\hat{x}_L =  \frac{\int_{\mathbb{R}^d}x e^{-\alpha \mathcal{E}(x)}f_1(x,t)dx}{\int_{\mathbb{R}^d} e^{-\alpha \mathcal{E}(x)}f_1(x,t)dx},
\end{equation}
which corresponds to the cases where the weighted mean depends only on the followers or only on the leaders, respectively.

In Figure \ref{fig:xhat_varies} the success rate and number of iterations as $\sigma_F$ and $d$ varies for $\hat{x}$ (left), $\hat x_f$ (middle) and $\hat x_L$(right). In the first row, results for the case with random leaders generation are shown, in the second row the mixed leaders generation with $\bar{p}=0.5$ and in the third row the case with weighted leaders generation. Note that the performance of the random strategy, especially for large values of the dimension $d$ is higher if $\hat{x}_F(t)$ is used for the estimate of the global minimizer. This can be explained by a better exploration phase of the particles during the evolution, whereas the leaders position $\hat x _L$ may result in a less accurate estimate, since labels change randomly. The weighted strategy with $\hat{x}_L(t)$  has computational advantages  since leaders are chosen to be the agents with optimal position and the computation of the $\hat{x}_L(t)$ requires a lower number of evaluations of the cost function. This may be advantageous in particular if the evaluation of the cost function is numerically expensive.
\begin{figure}[!htb]
	\centering
	\includegraphics[width=0.327\linewidth]{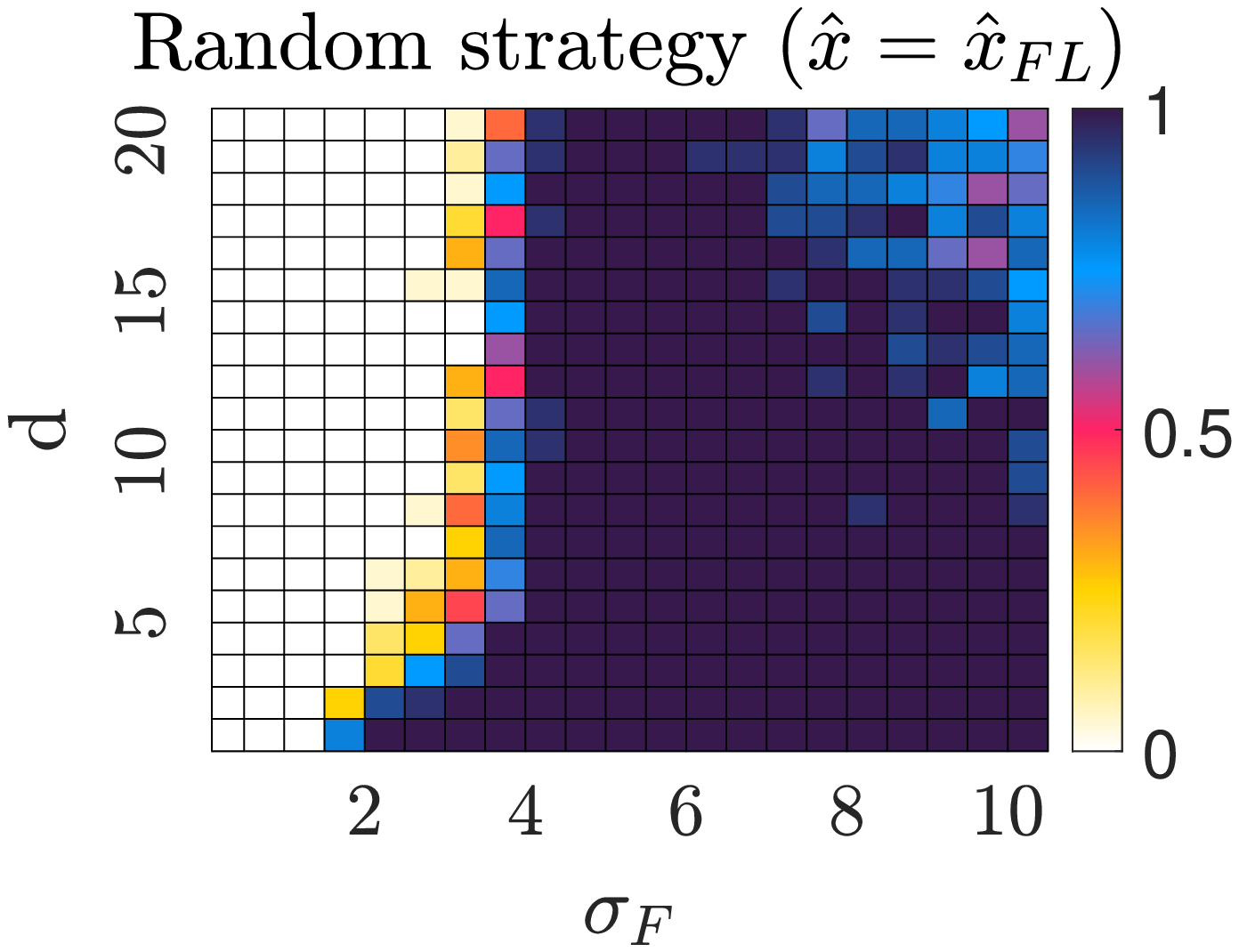}
	\includegraphics[width=0.327\linewidth]{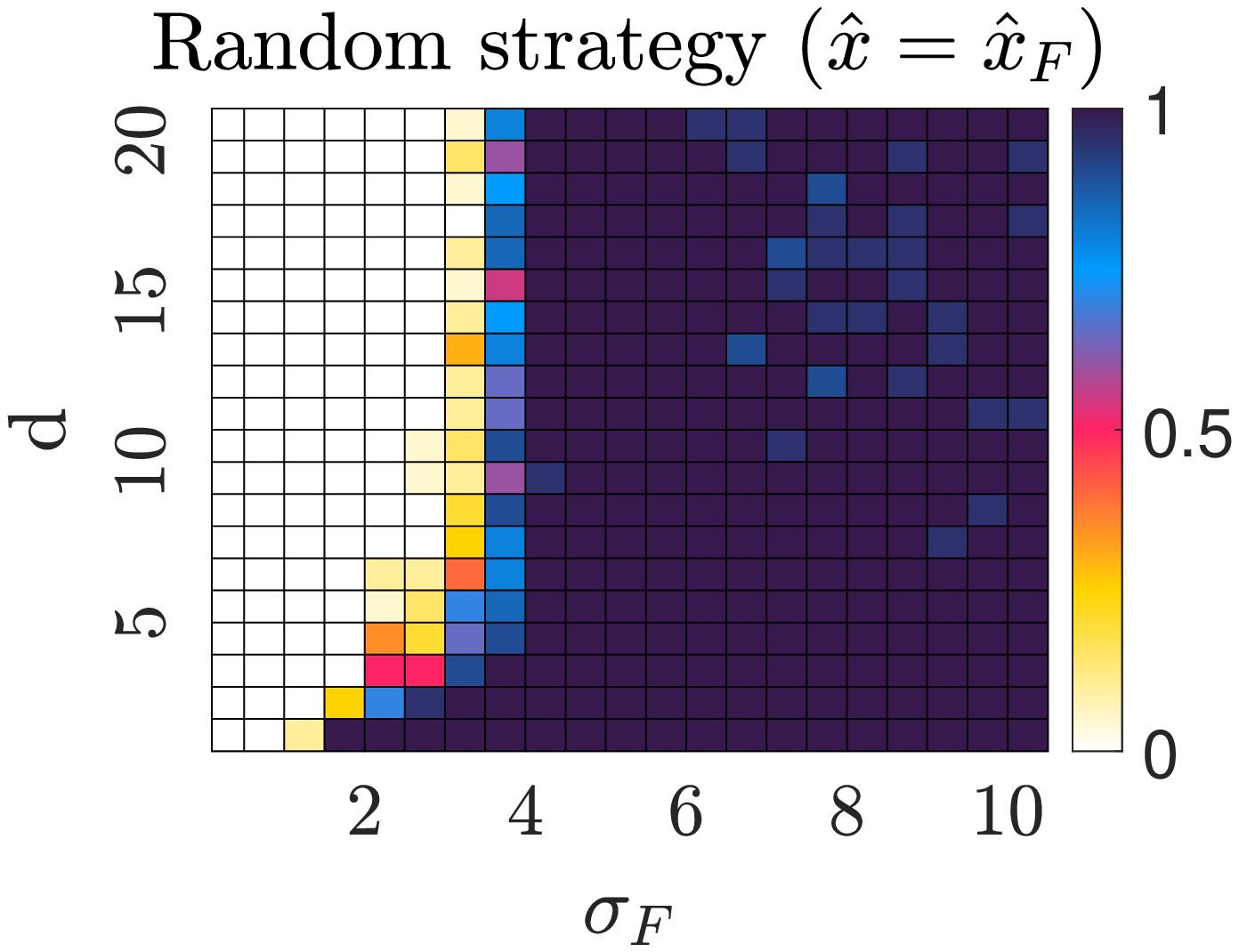}
	\includegraphics[width=0.327\linewidth]{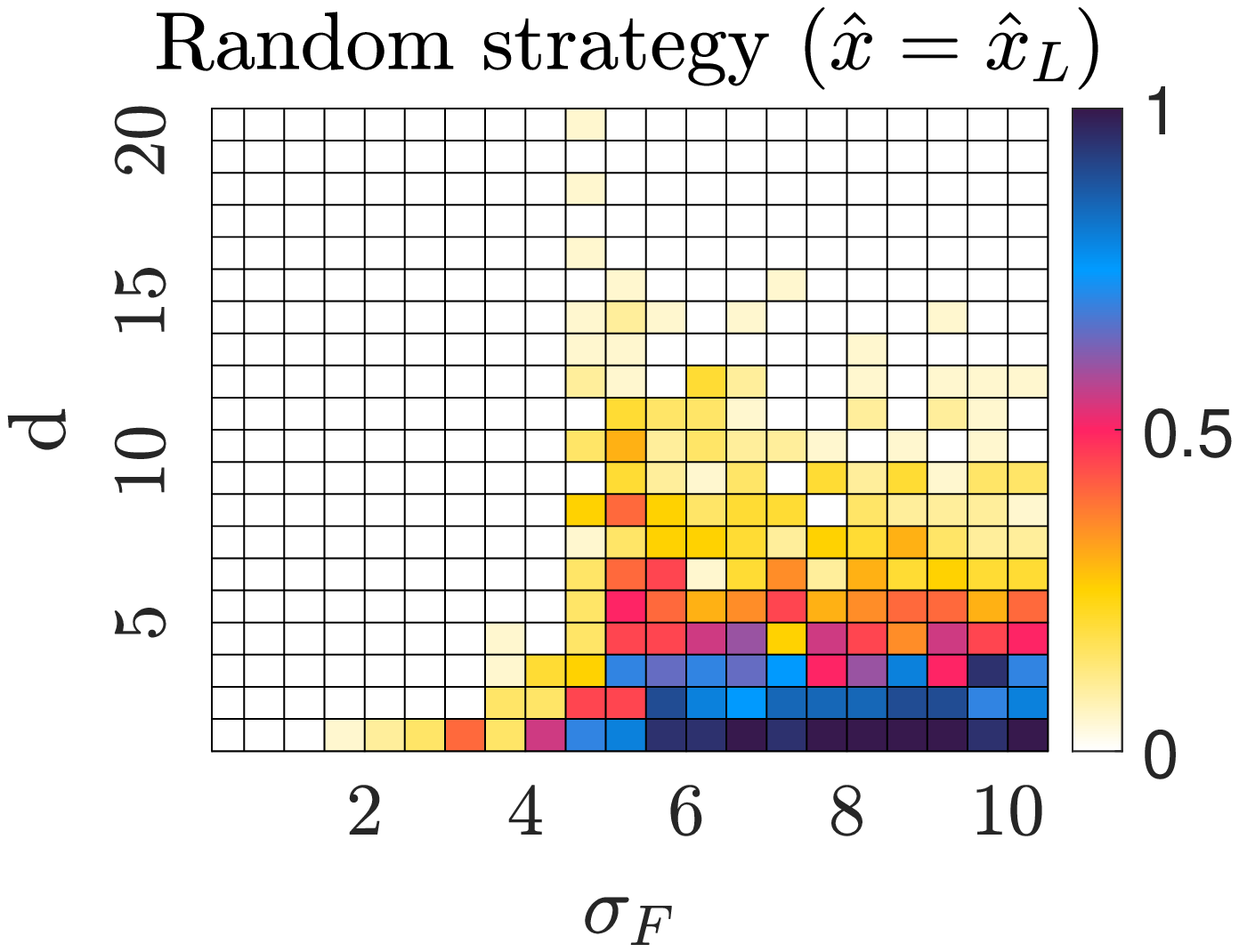}\\
	\includegraphics[width=0.327\linewidth]{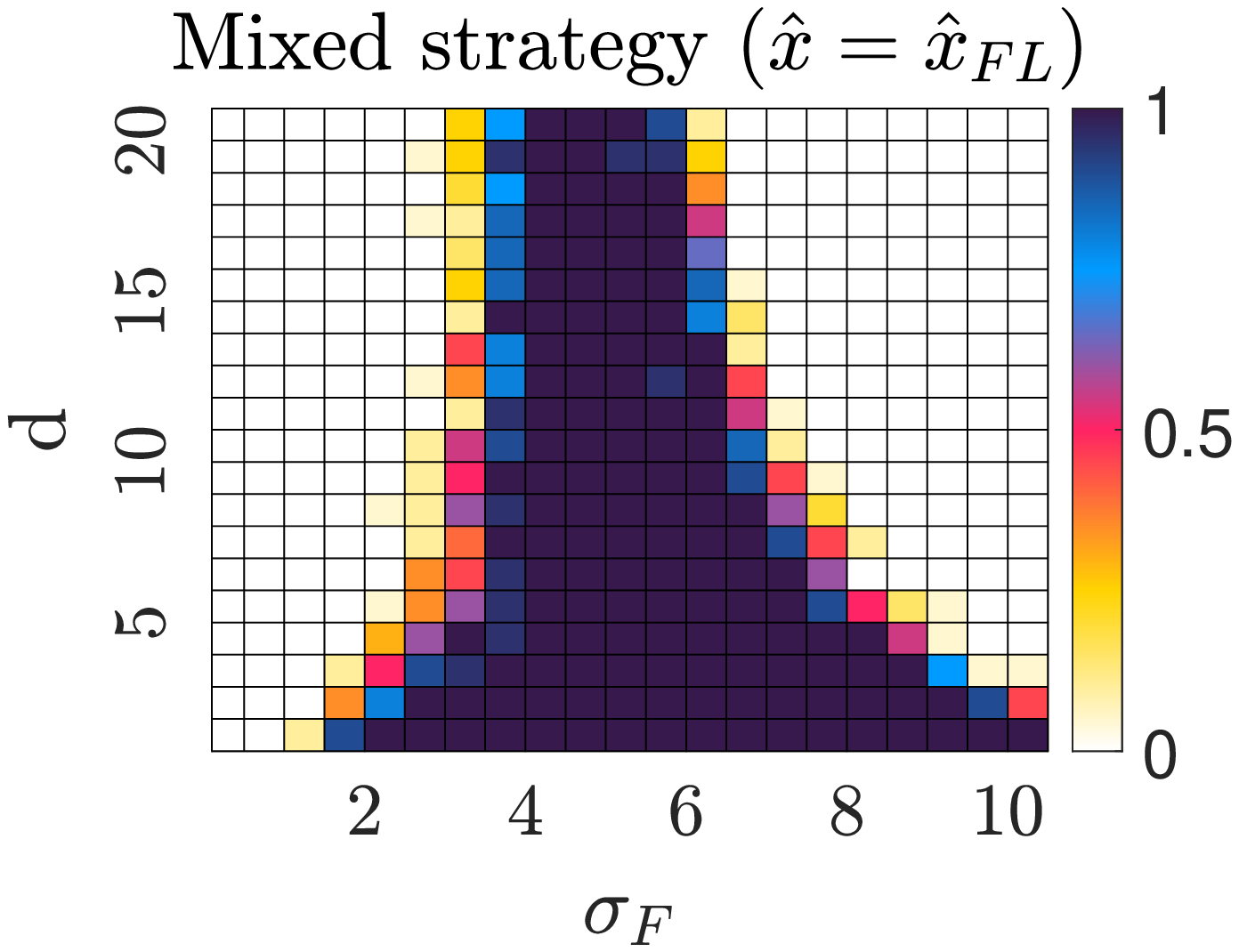}
	\includegraphics[width=0.327\linewidth]{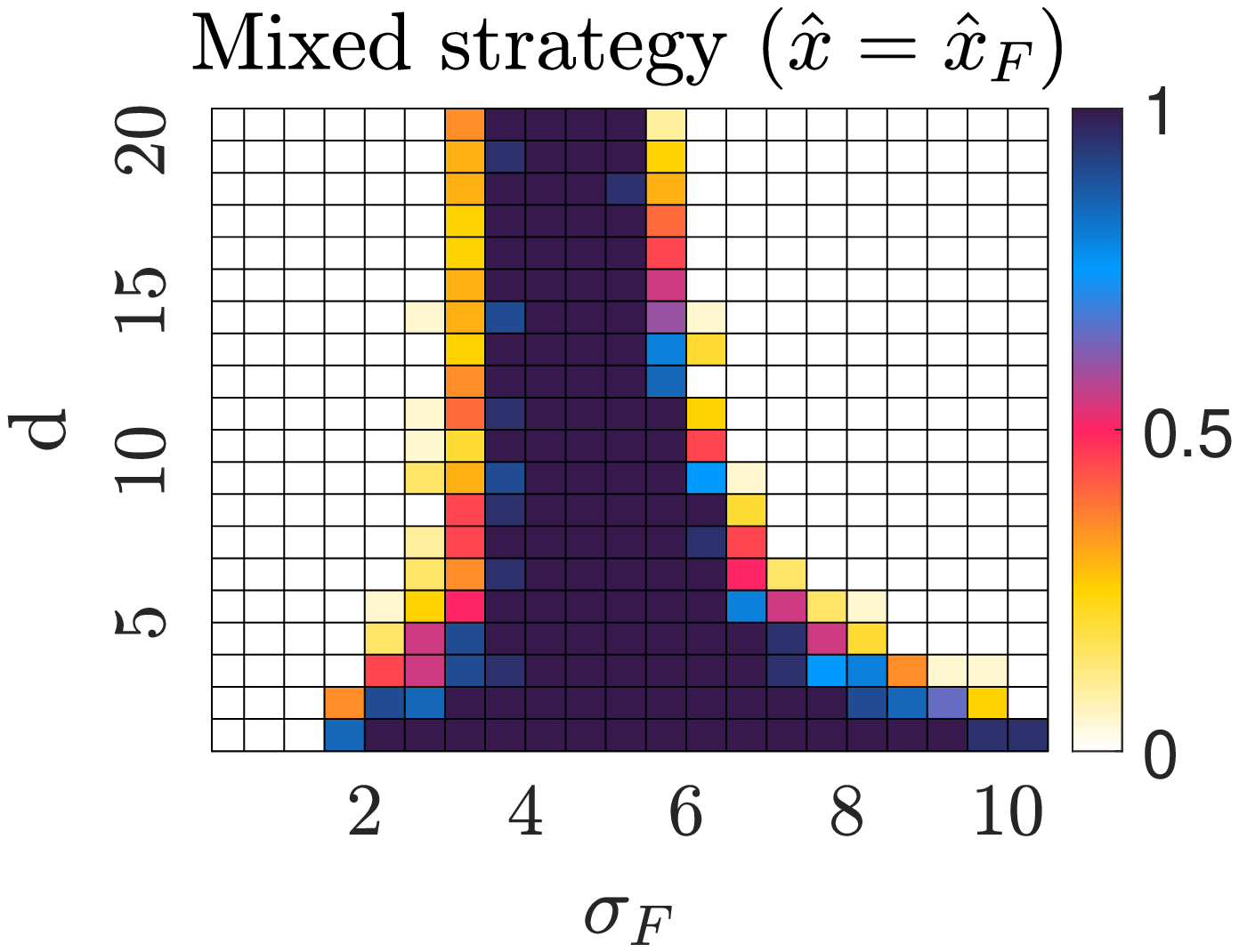}
	\includegraphics[width=0.327\linewidth]{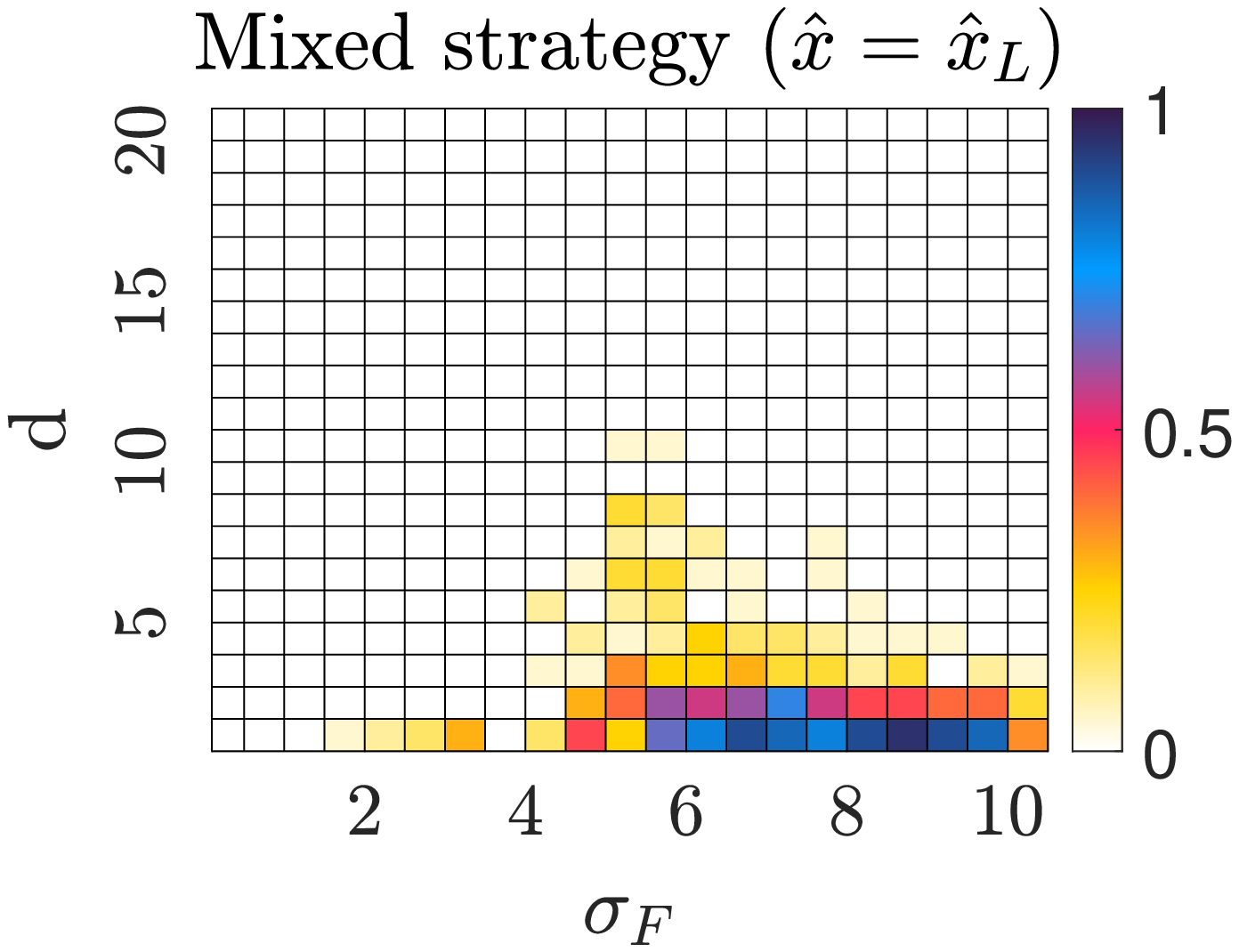}\\
	\includegraphics[width=0.327\linewidth]{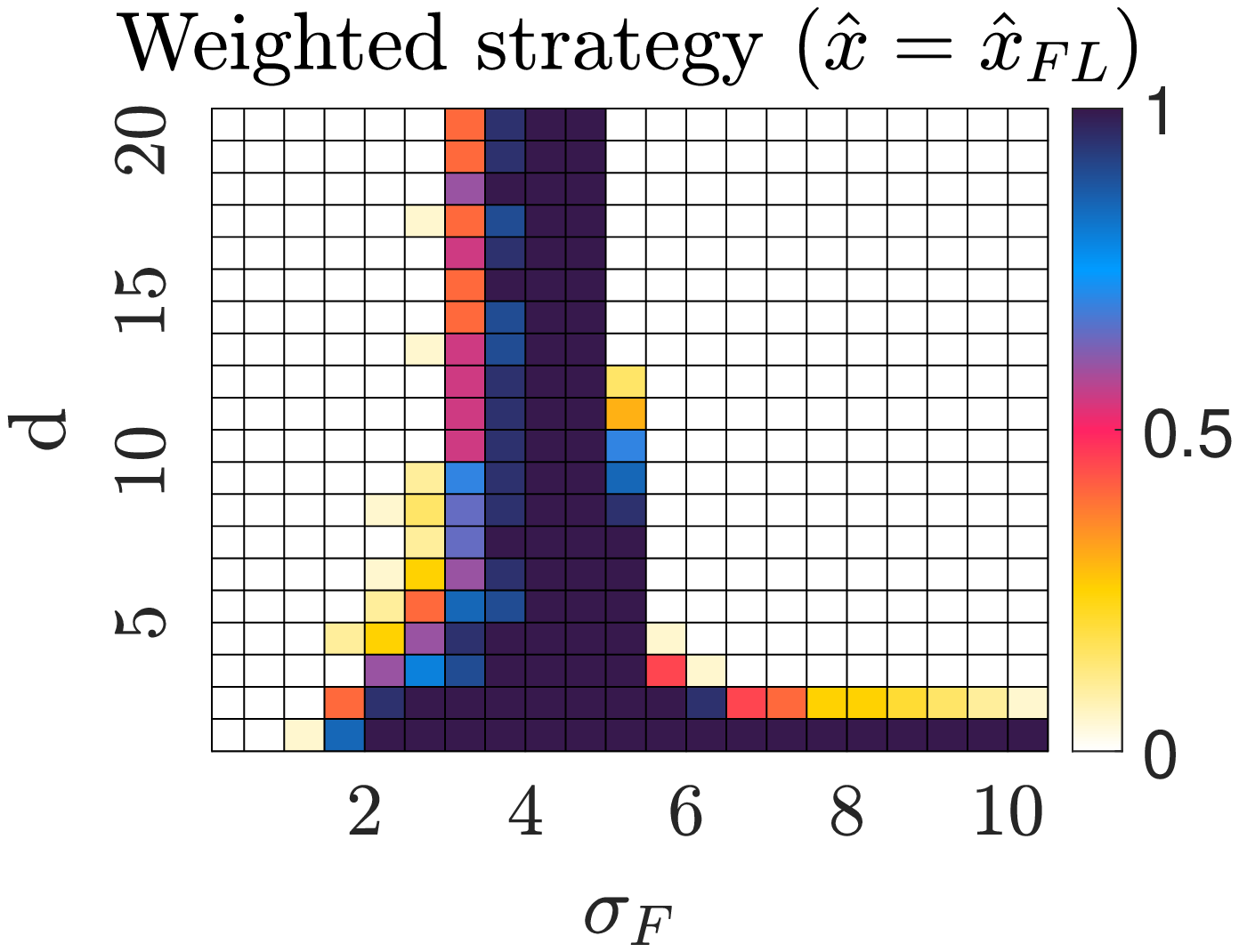}
	\includegraphics[width=0.327\linewidth]{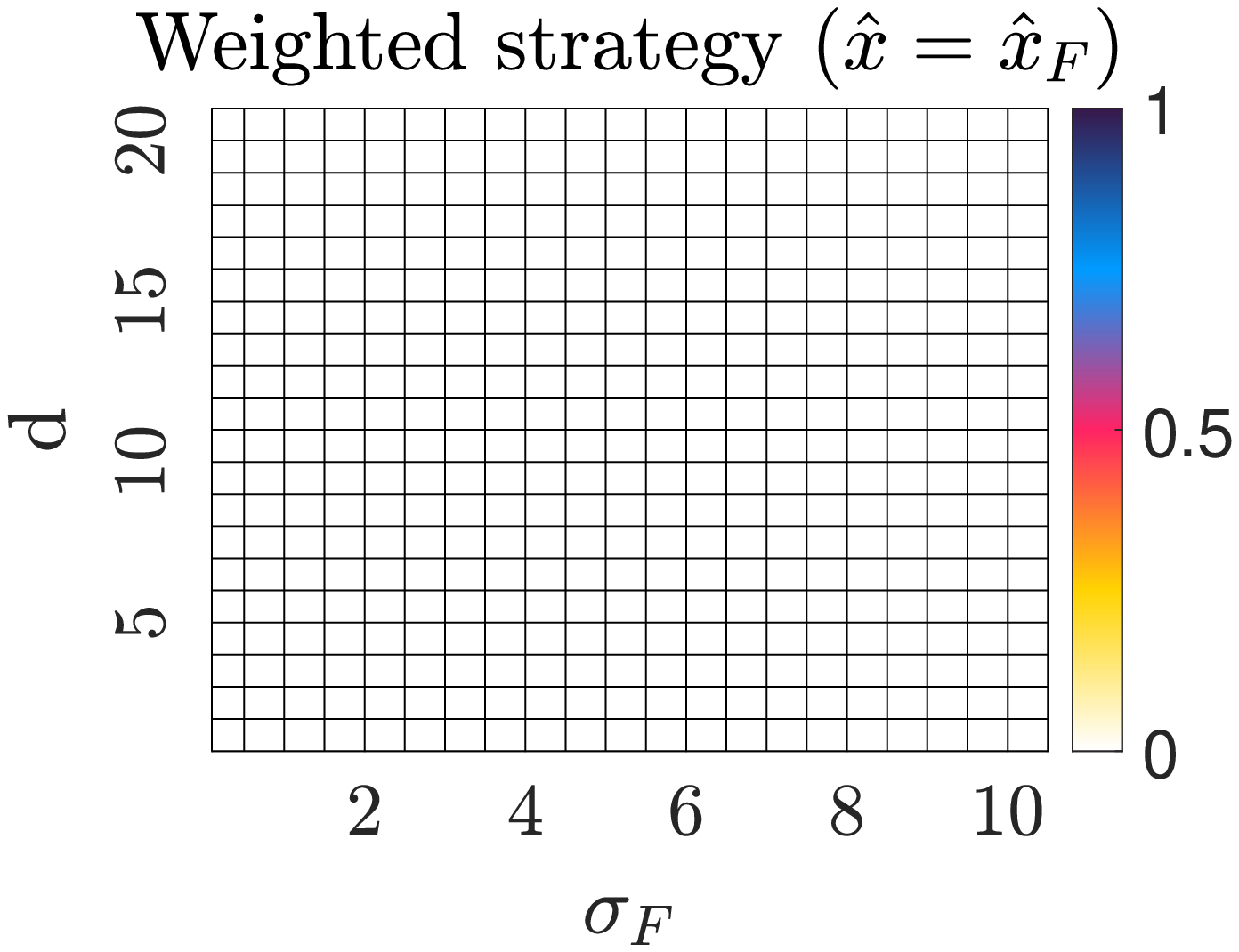}
	\includegraphics[width=0.327\linewidth]{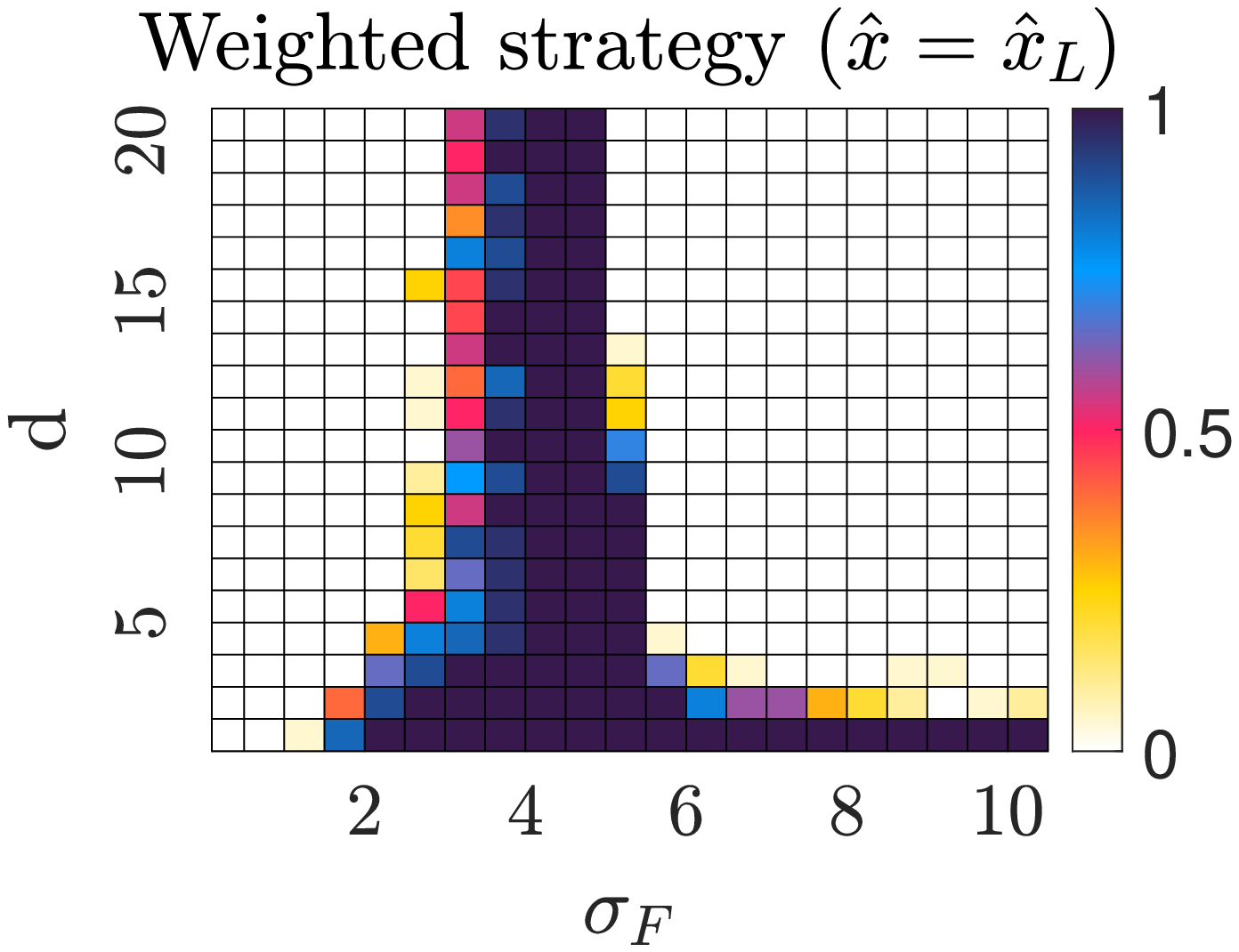}
	\vspace*{8pt}
	\caption{Success rates for varying $\sigma_F$ and $d$ for the translated Rastrigin function with dynamics simulated with the GKBO method for $\hat{x}$ (left), $\hat{x}_F$ (middle), $\hat x_L$ (right). The first row is with random leader emergence, second row with mixed strategy $\bar{p} = 0.5$ and third row with weighted leader emergence. 
	}
	\label{fig:xhat_varies}
\end{figure}

\subsection{Test 3: Comparison in $d=20$ dimensions for varying $\sigma_F$ } \label{sec:test3}
We fix $d=20$ and let $\sigma_F$ vary from $\sigma_F=0.1$ to $\sigma_F = 10$ to compare the performance of GKBO (equation \eqref{eq:bin_2pop}), standard GA (equation \eqref{eq:GA}-\eqref{eq:GA_1}), the modified GA (equation \eqref{eq:GA}-\eqref{eq:GA_mod}) and the KBO (equation \eqref{eq:bin_1pop}).

In Figure \ref{fig:sigma_range_in0_p01} the success rates and means of the number of iterations obtained with the different algorithms in the case of the translated Rastrigin function is shown. Here, test GKBO with $\hat{x}$, $\hat{x}_F$ and $\hat{x}_L$ as defined above and study random leader emergence (left), mixed leader emergence with $\bar{p}=0.5$ (middle), and weighted leader emergence (right). Altough the success rates of KBO and the variants of GKBO behave similar, the GKBO versions required less iterations. Moreover, we remark that the behavior of the GKBO with weighted leaders generation and with $\hat{x}$ as in equation \eqref{eq:x_tot} and of the KBO is similar, as expected from our analysis.
\begin{figure}[!htb]
	\centering
	\includegraphics[width=0.327\linewidth]{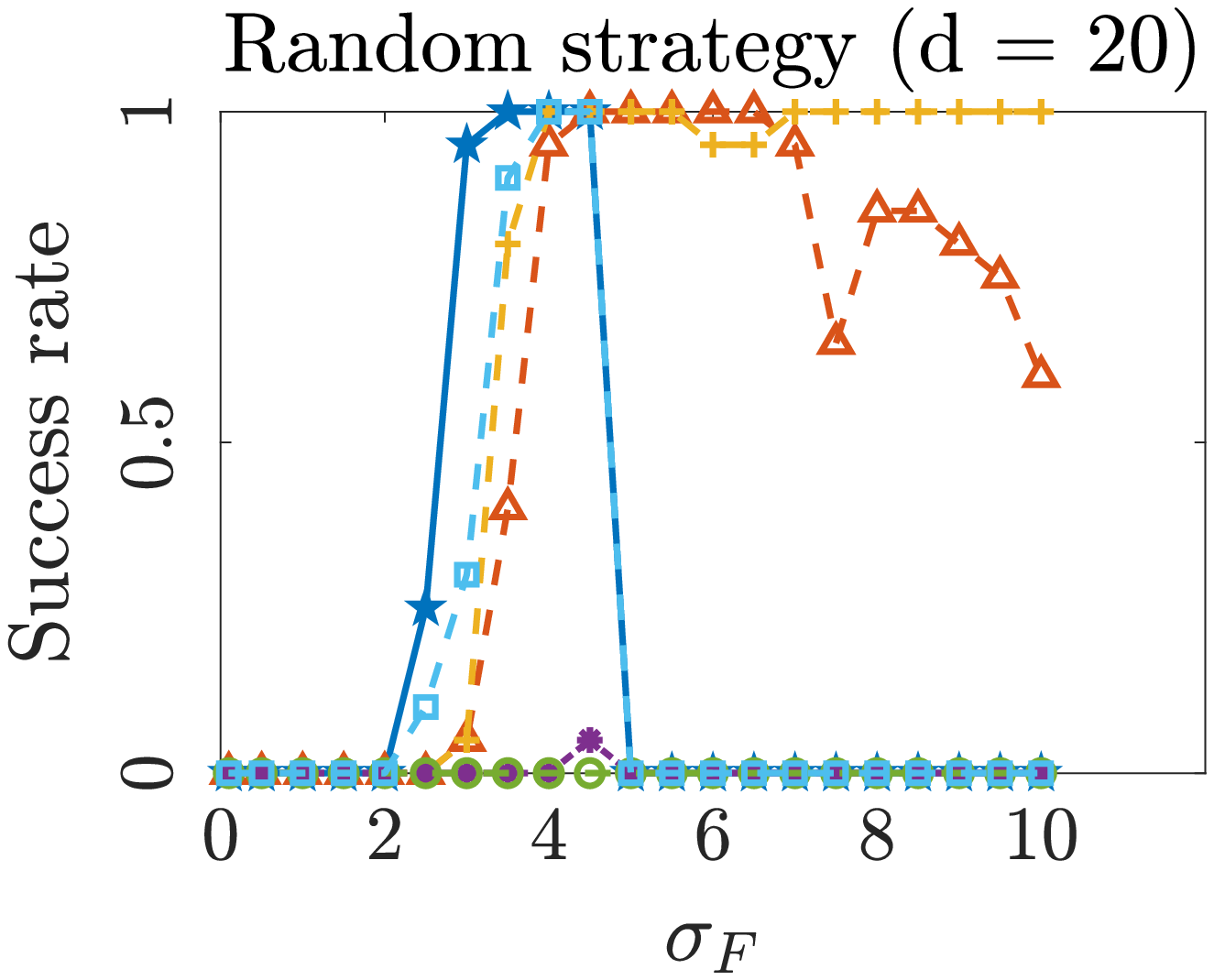}
	\includegraphics[width=0.327\linewidth]{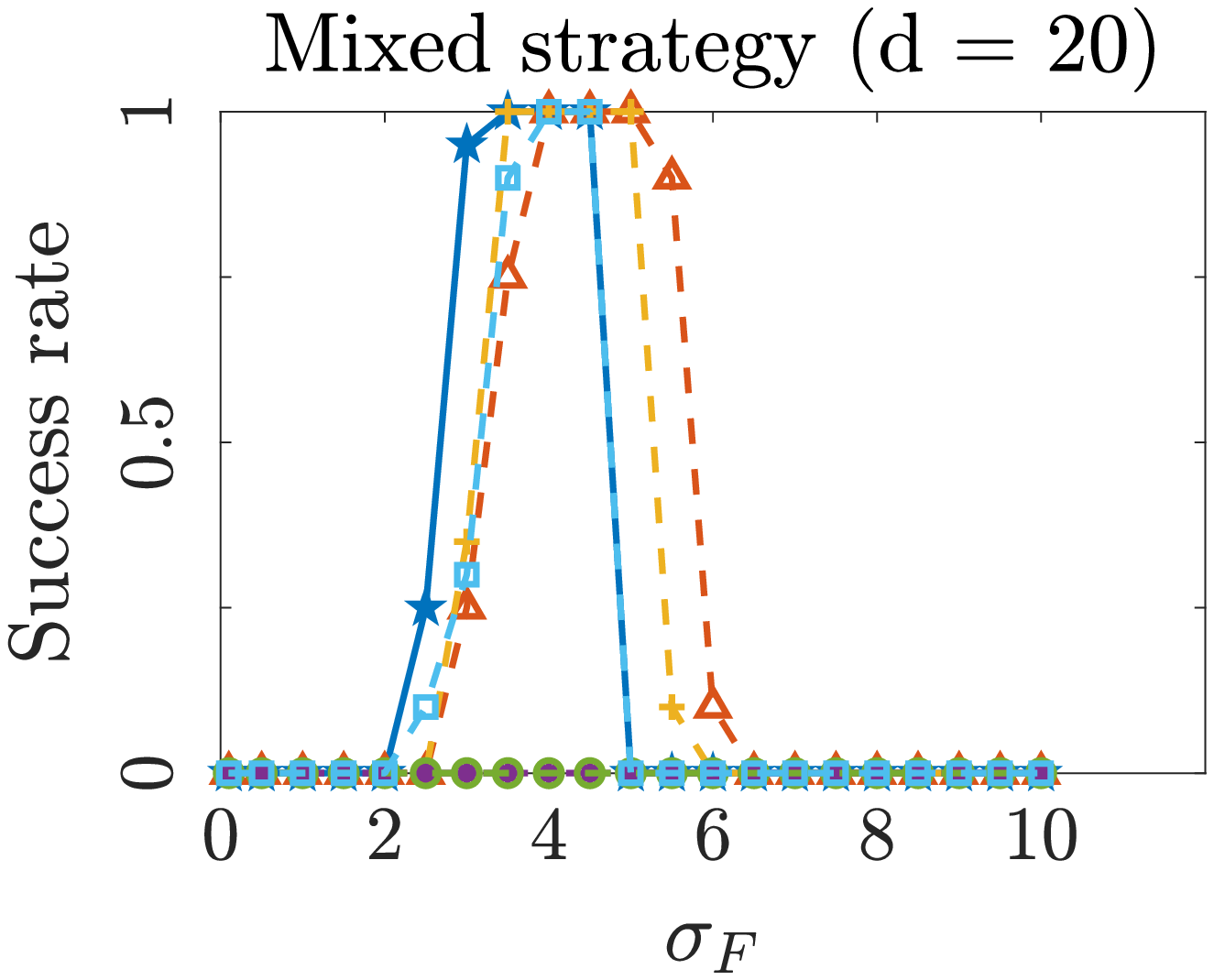}
	\includegraphics[width=0.327\linewidth]{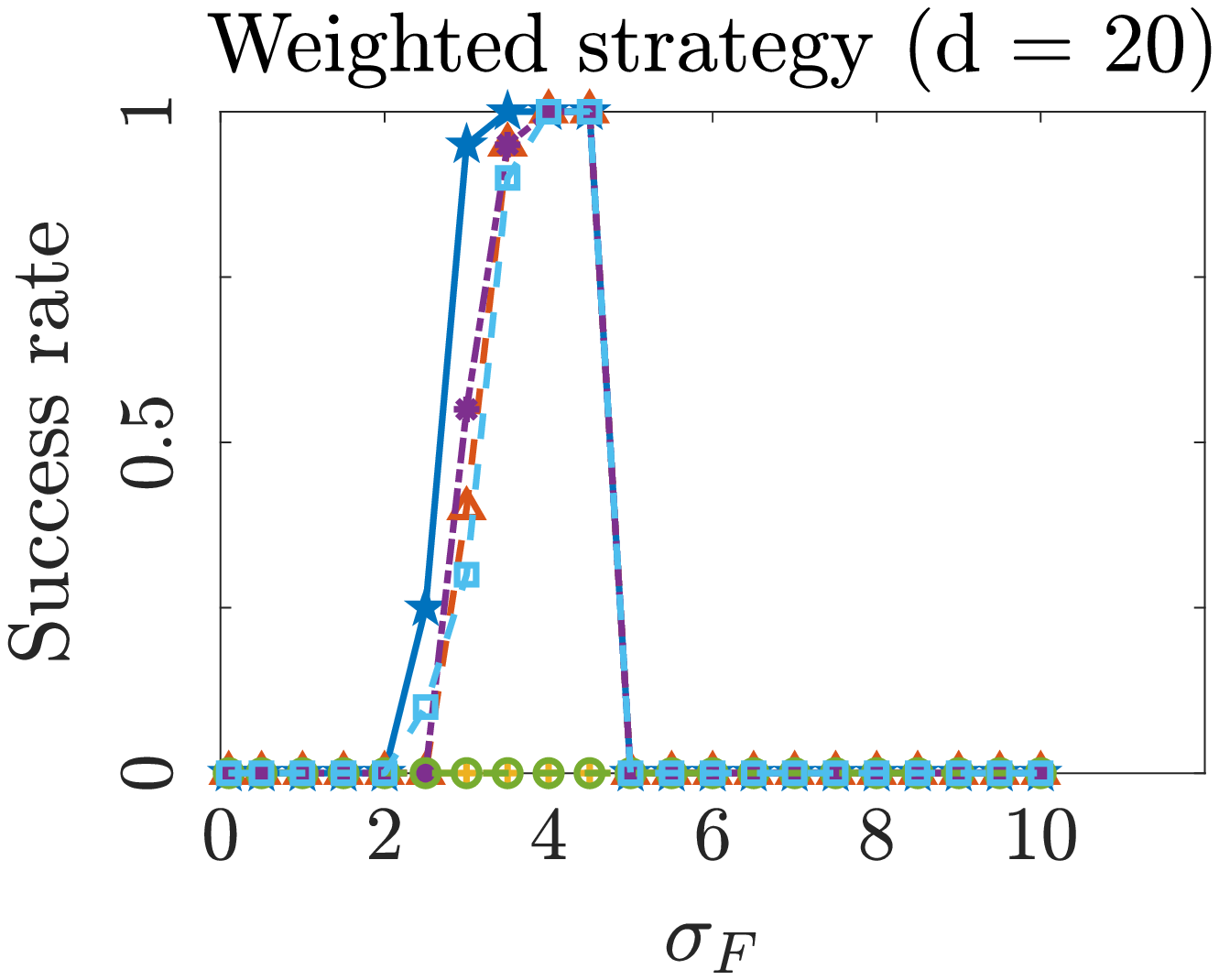}\\
	\vspace{18pt}
	\includegraphics[width=0.327\linewidth]{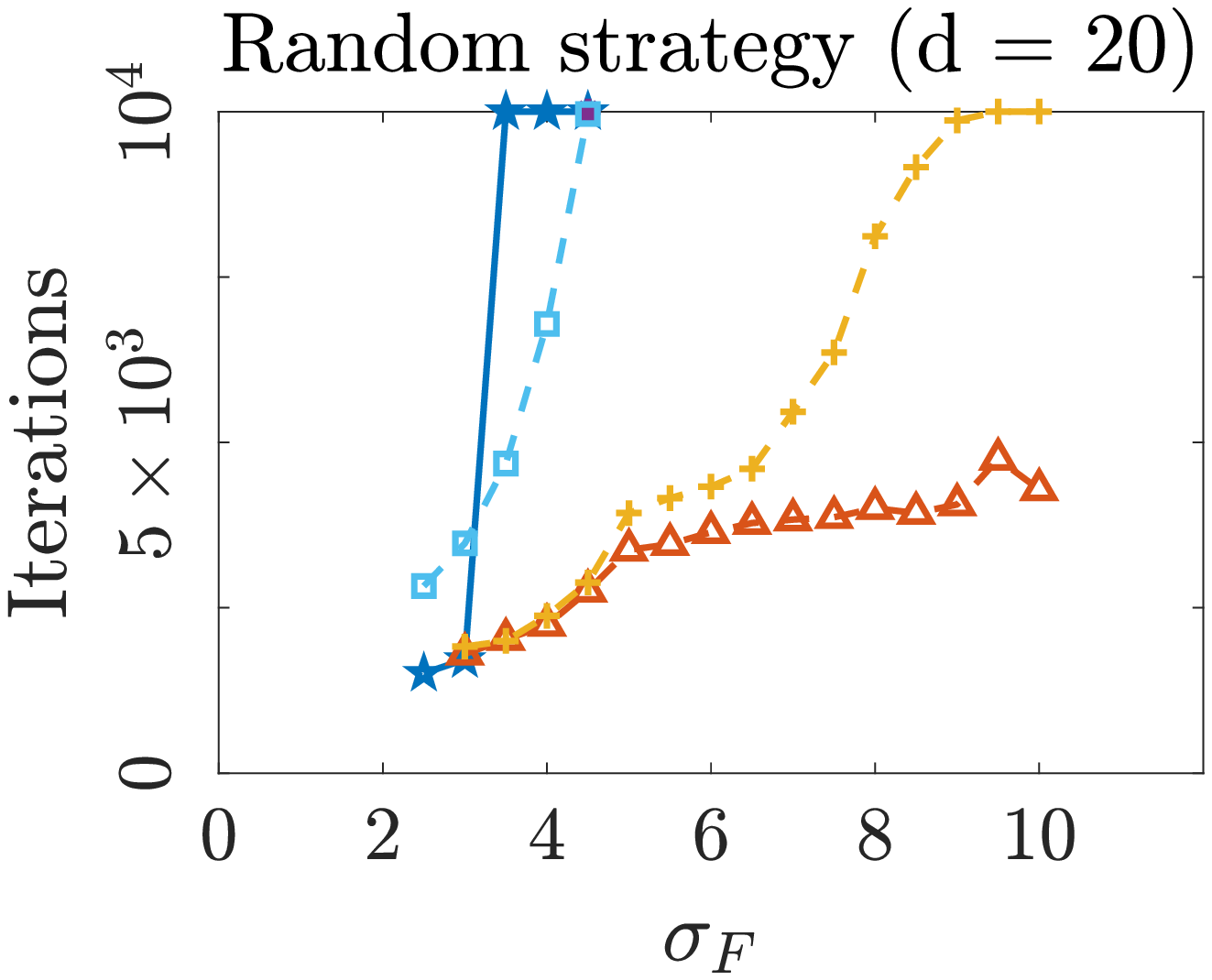}
	\includegraphics[width=0.327\linewidth]{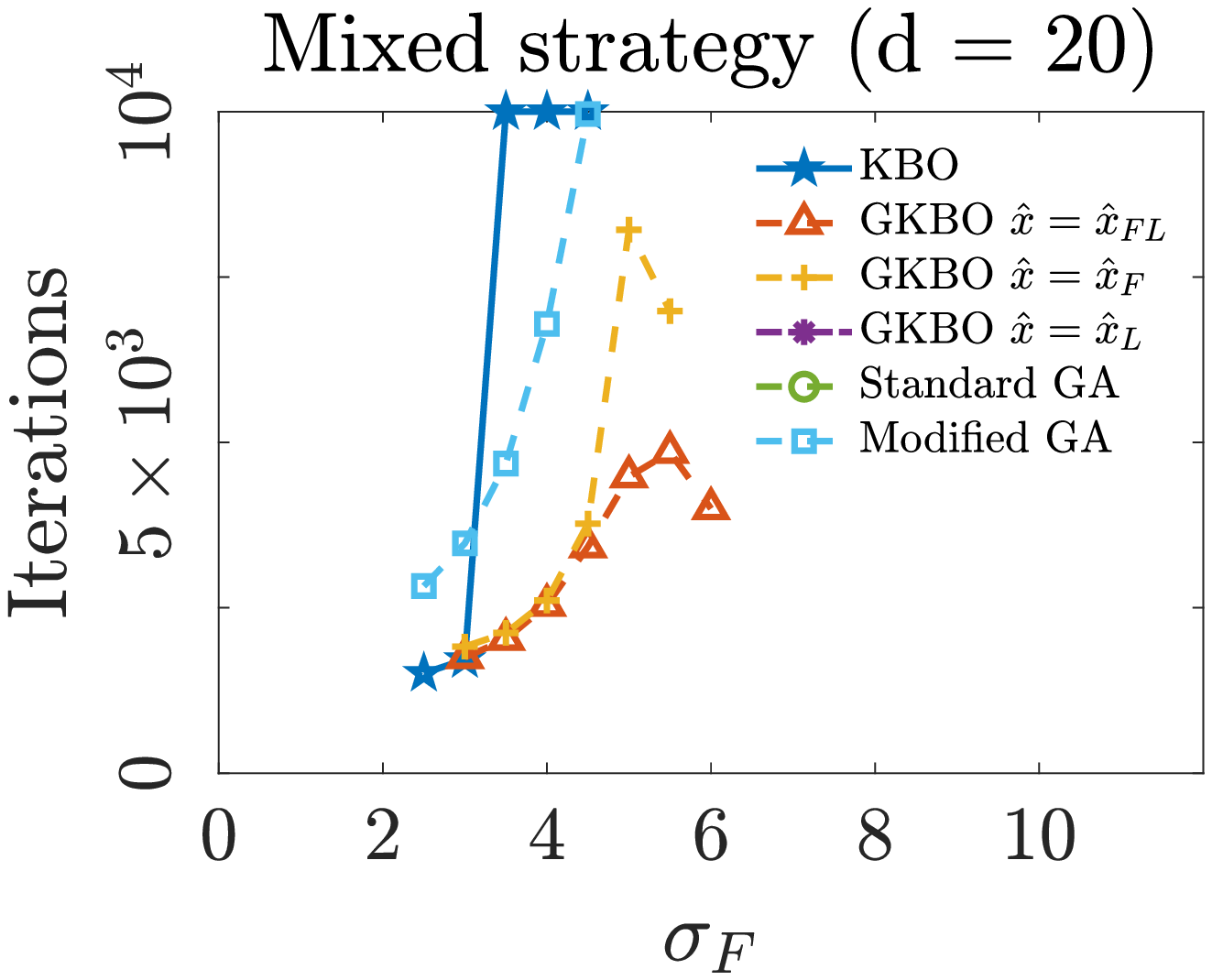}
	\includegraphics[width=0.327\linewidth]{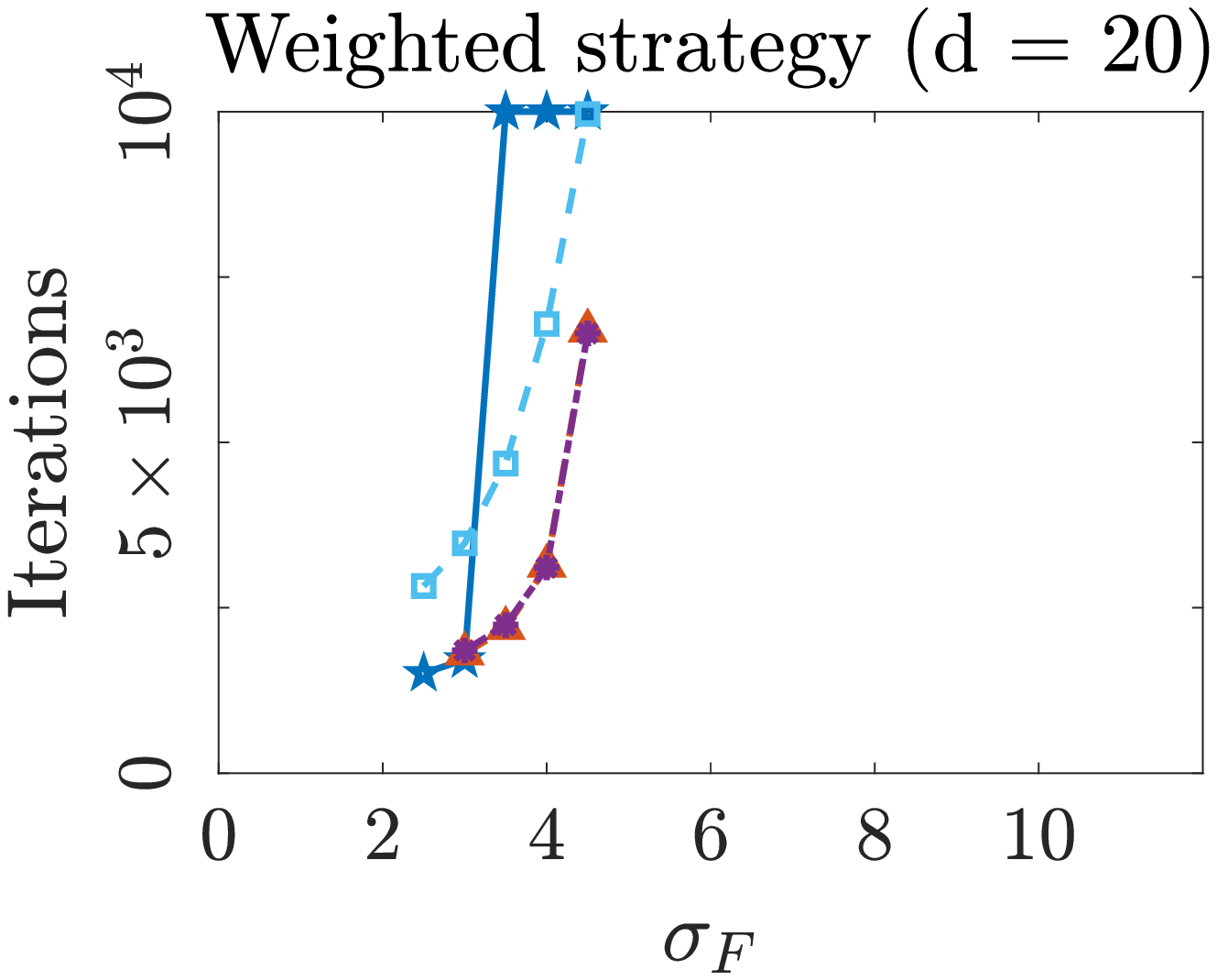}
	\vspace*{8pt}
	\caption{Success rates and means of the number of iterations as $\sigma_F$ varies and $d=20$ for the translated Rastrigin function obtained with the different algorithms. On the left leaders emerge randomly, in the, we consider mixed leader emergence with $\bar{p}=0.5$, and on the right, we have weighted leader emergence. The markers denote the value of the success rates and numbers of iterations for different $\sigma_F$.  
	}
	\label{fig:sigma_range_in0_p01}
\end{figure}

\subsection{Test 4: Comparison of different leader emergence strategies.}\label{sec:test4}
Let us fix $d = 20$ and consider the mixed leader emergence strategies as discussed in Remark~\ref{remark}. 
In Figure \ref{fig:weight_varies} on the left we see the success rates for different values of $\sigma_F$ and $\bar{p}$, on the right
the number of iterations for different values of $\bar{p}$ and for $\sigma_F= 4,5$.  In Figure \ref{fig:weight_varies_1} the success rate and minimum, maximum and mean iterations number for the GKBO method with $\hat{x}$ is shown for $d=20$ as $\bar{p}$ and $\sigma_F$ vary.  
\begin{figure}[!htb]
	\centering
	\includegraphics[width=0.495\linewidth]{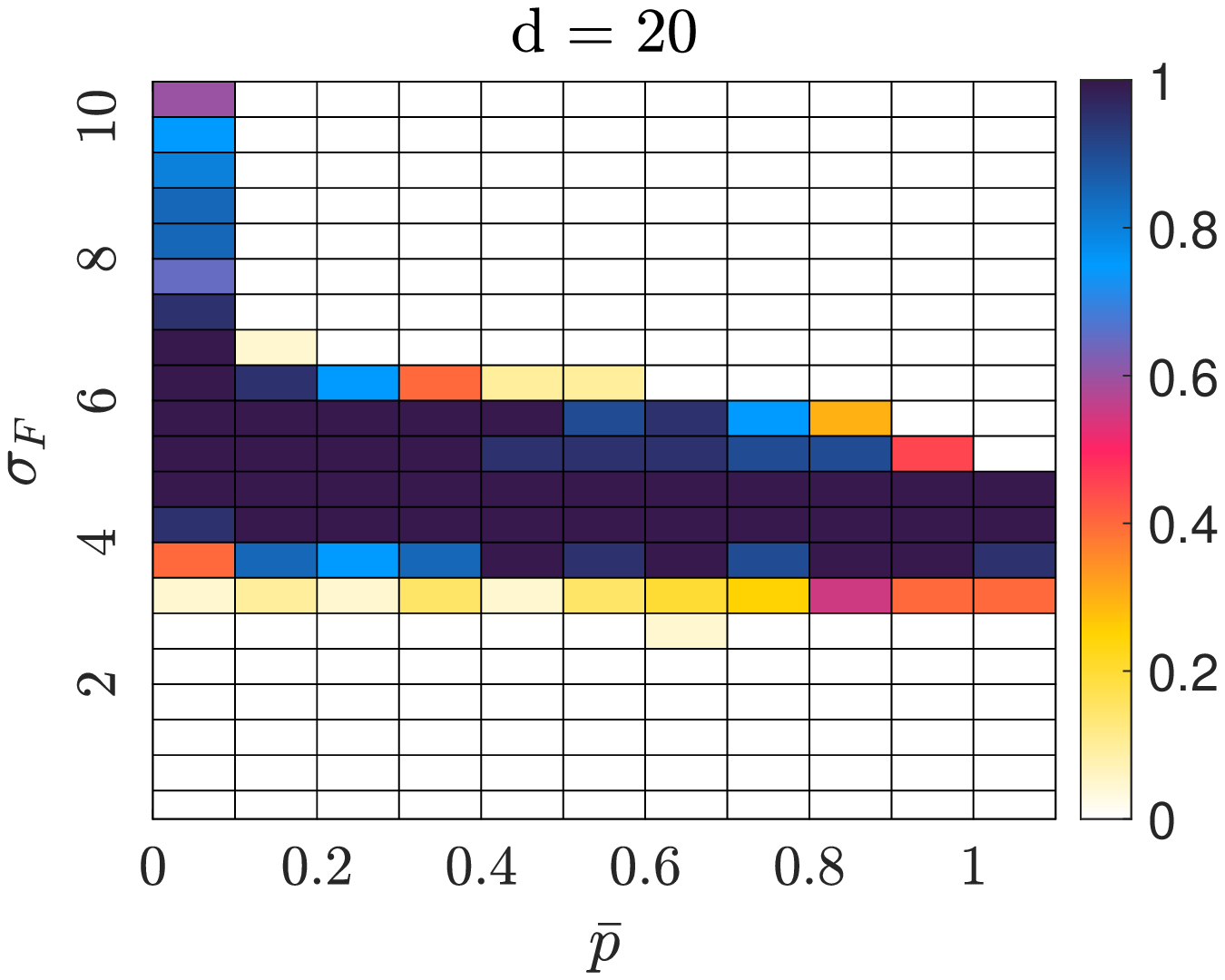}
	\includegraphics[width=0.495\linewidth]{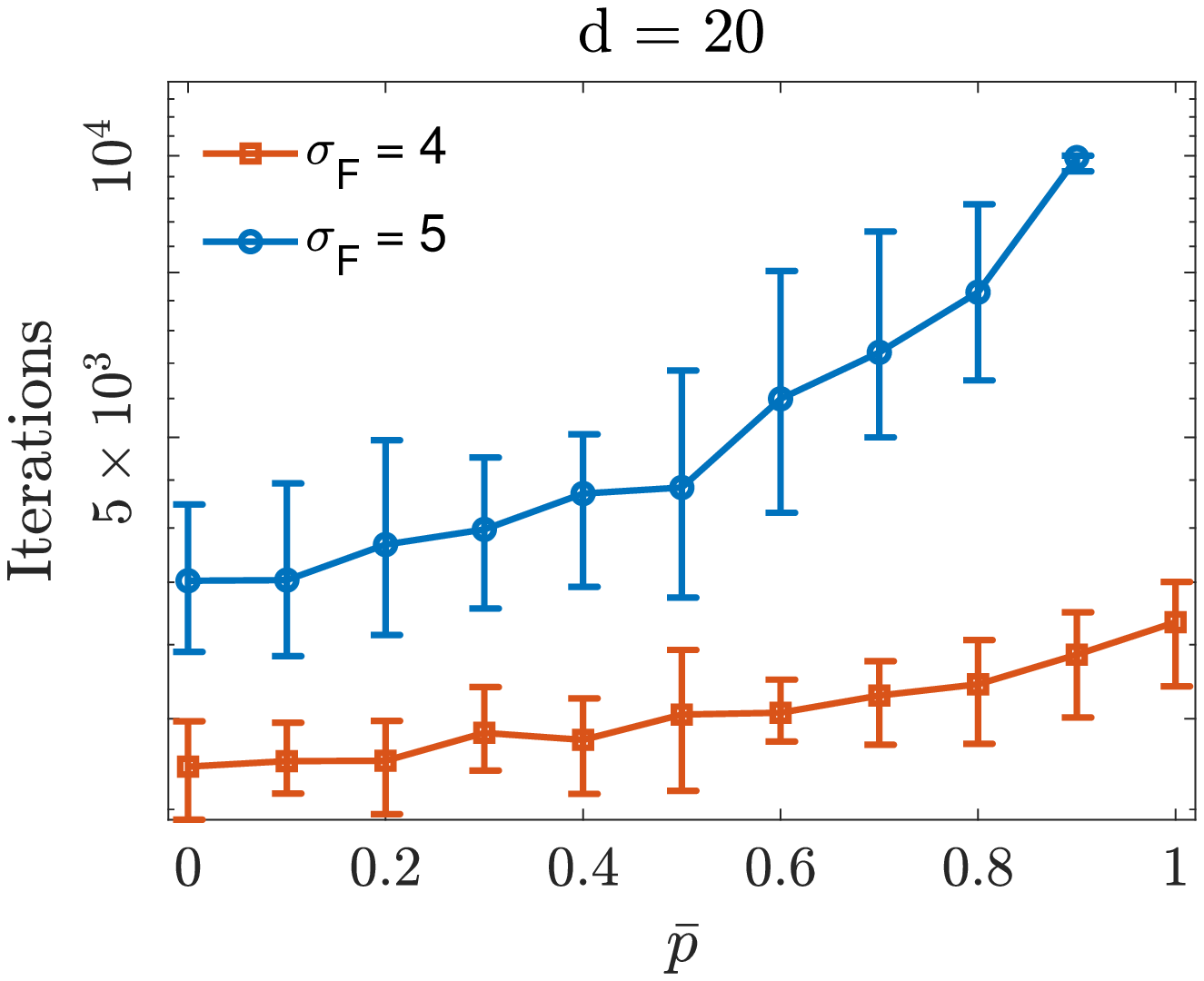}
	\vspace*{8pt}
	\caption{Different leader emergence strategies. On the left, success rates for varying $\sigma_F$ and $\bar{p}$ and $d=20$. On the right, max, min and mean number of iterations obtained in the different simulations for $d =20$ and $\sigma_F=4,5$ as $\bar{p}$ varies. The markers denote the number of iterations needed for different $\bar{p}$ and tested on the translated Rastrigin function.  
	}
	\label{fig:weight_varies}
\end{figure}

\begin{figure}[!htb]
	\centering
	\includegraphics[width=0.495\linewidth]{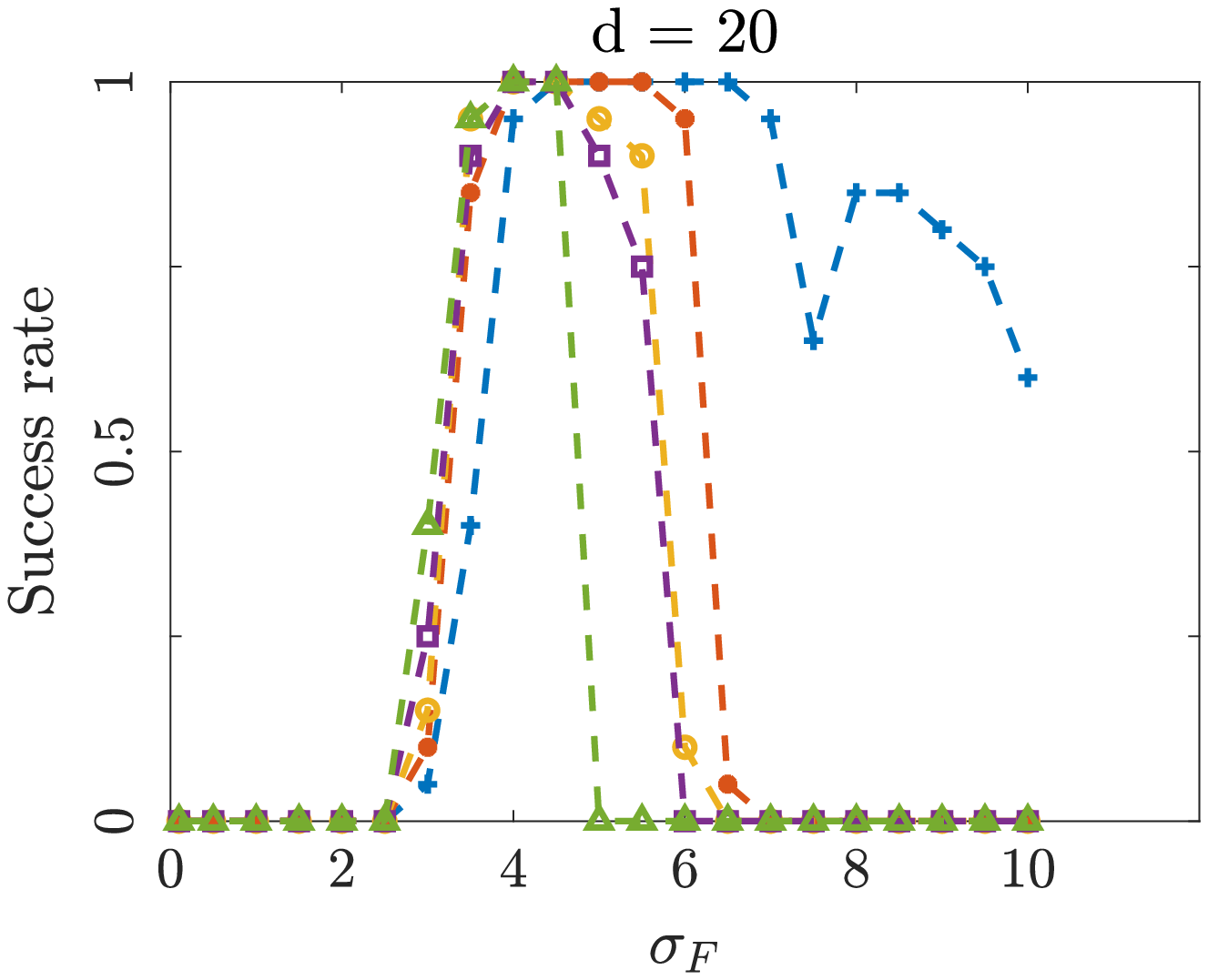}
	\includegraphics[width=0.495\linewidth]{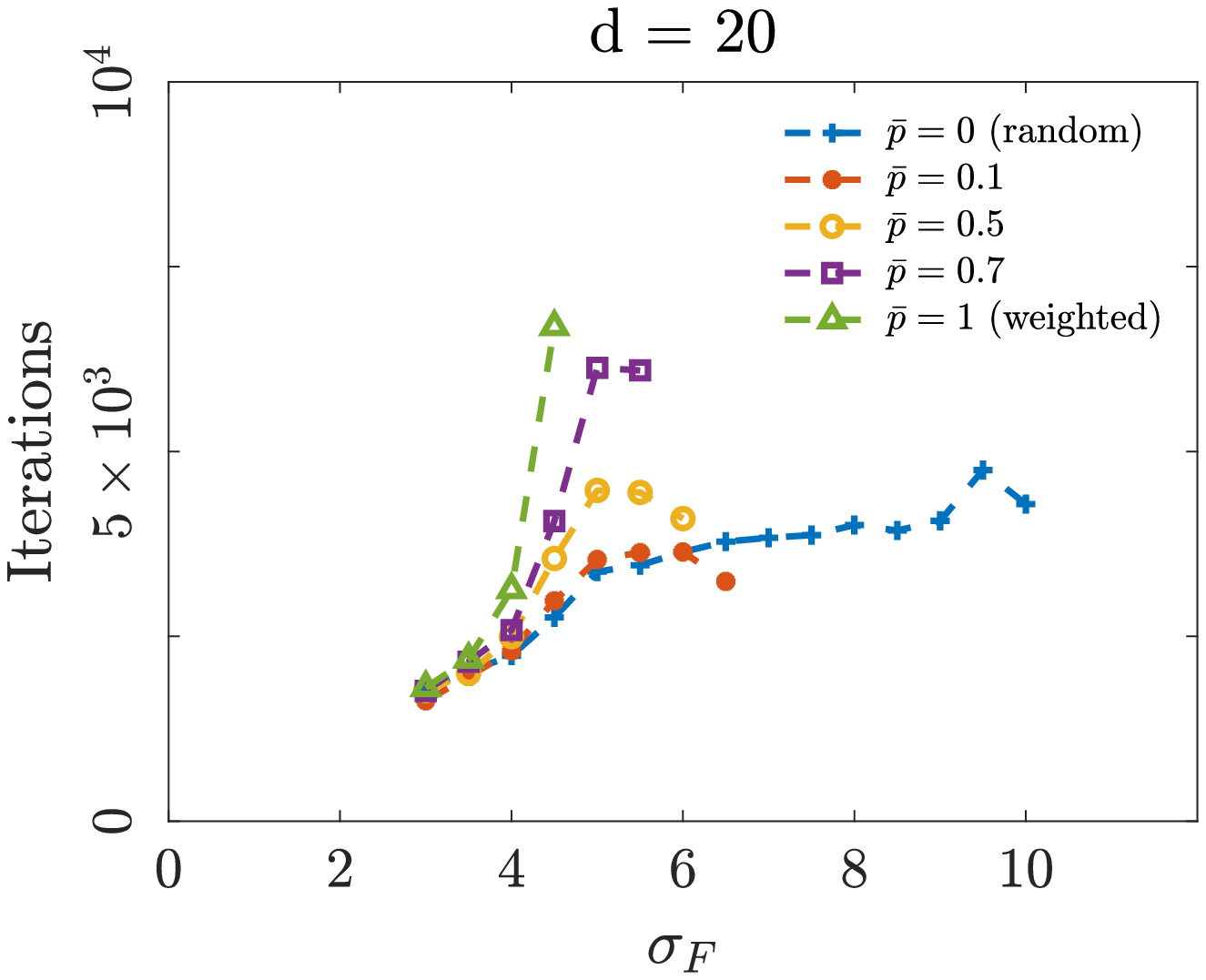}
	\vspace*{8pt}
	\caption{Different leader emergence strategies. Success rates and mean number of iterations for $d=20$ as $\sigma_F$ varies and for different values of $\bar{p}$, tested on the translated Rastigin function. The markers denote the value of the success rates and number of iterations for different $\sigma_F$.  
	}
	\label{fig:weight_varies_1}
\end{figure}

\subsection{Test 5: Comparison of different methods for varying $d$}\label{sec:test5}
We fix $\sigma_F=4$ and vary the dimension $d$ from $1$ to $20$. Figure~\ref{fig:d_range_in0_p01} shows the success rates and means of the number of iterations of the different methods in the case of the translated Rastrigin function. GKBO uses $\hat{x}$ as in \eqref{eq:x_tot}. 
\begin{figure}[!htb]
	\centering
	\includegraphics[width=0.495\linewidth]{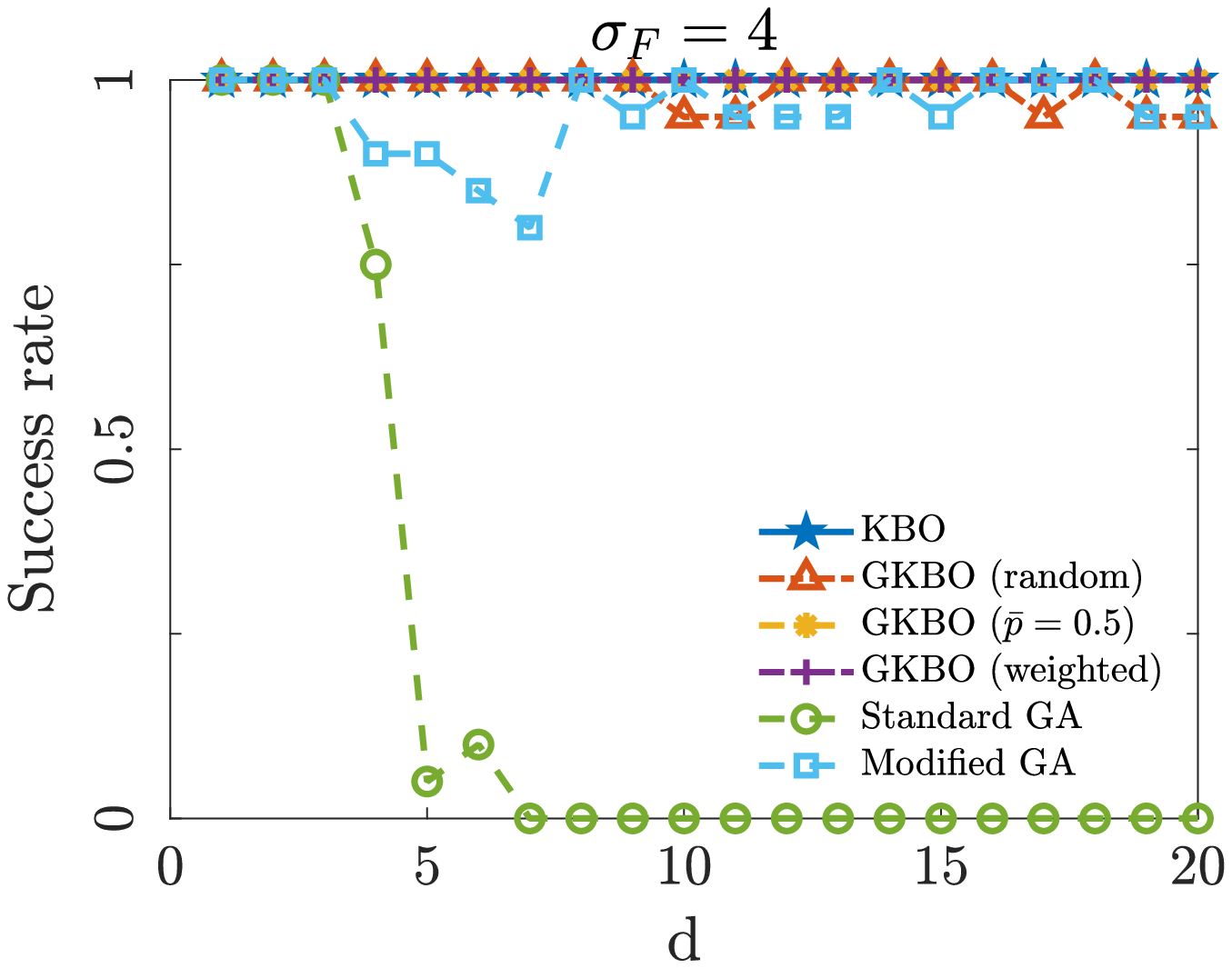}
	\includegraphics[width=0.495\linewidth]{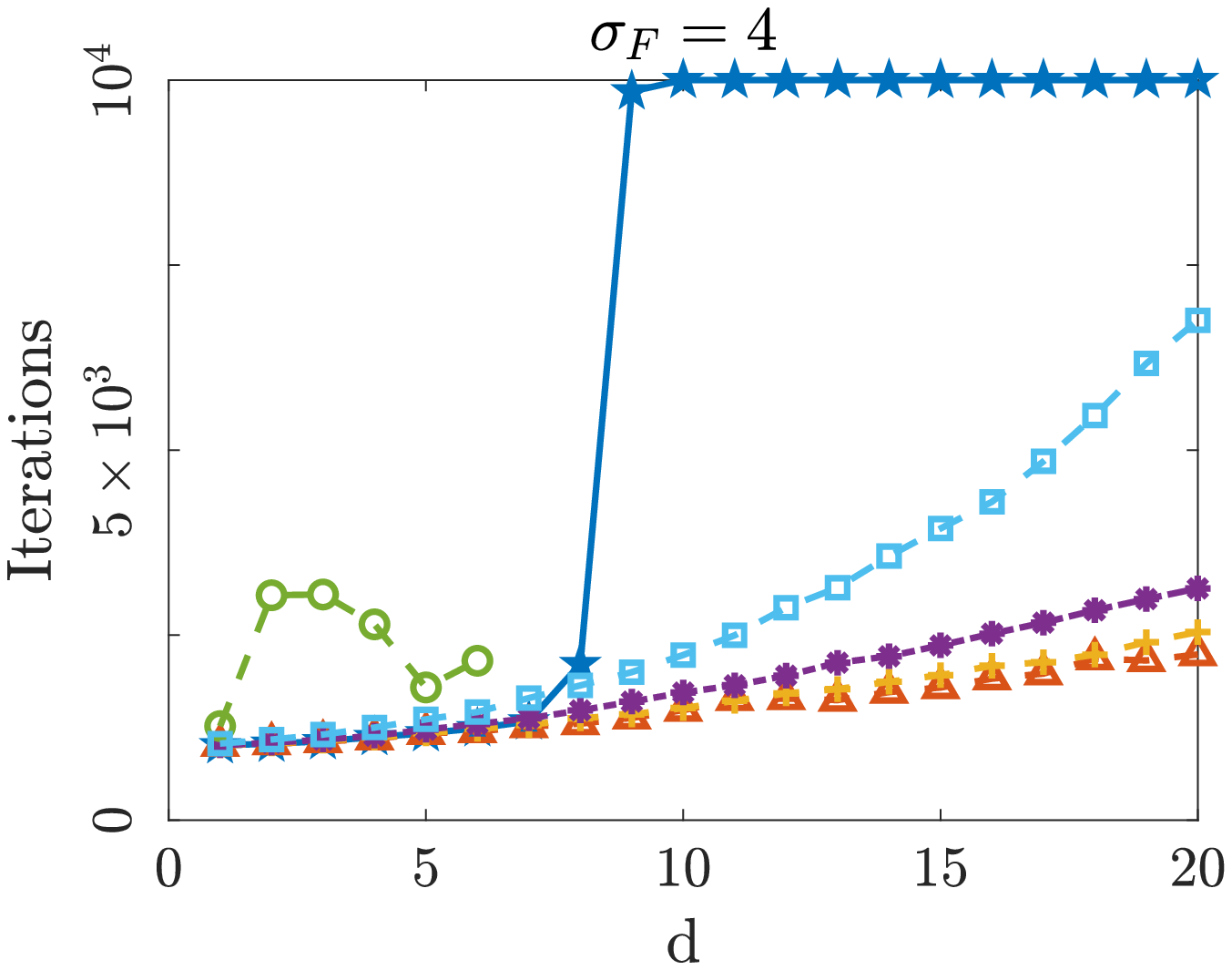}
	\vspace*{8pt}
	\caption{Success rates and number of iterations as $d$ varies for the translated Rastrigin function obtained with the different algorithms. The markers denote the value of the success rates and mean number of iterations.  
	}
	\label{fig:d_range_in0_p01}
\end{figure}
\subsection{Test 6: Comparison of the accuracy  for varying frequency $\varepsilon$}\label{sec:test6}
Here we study the influence of the frequency parameter $\varepsilon$ by comparing the accuracy of the KBO and GKBO with weighted and random leader emergence. We run the test for $M=100$ simulations assuming the initial data to be normally distributed in the hypercube $[-4.12,0]^d$, $d = 20$. The accuracy is computed as 
\begin{equation}\label{eq:xhat_plot}
	\Vert \hat{x}(t)-\bar{x}\Vert_\infty,
\end{equation}
where $\bar{x}$ is the actual value of the minimum.
In Figure \ref{fig:accuracy} the accuracy of the KBO algorithm (left) for GKBO algorithm with random leader emergence (middle) and weighted leader emergence (right) with $\varepsilon=0.01$ (first row) and $\varepsilon=0.1$ (second row).   Note that in both cases, the values of $\sigma_F$ for which the method converges with the weighted GKBO and the KBO algorithm is almost the same. If $\varepsilon=0.01$ the accuracy of the weighted GKBO is higher than the one of the KBO. If $\varepsilon=0.1$ the random strategy performs better than the other methods since the algorithm converges for almost all the values of $\sigma_F$ considered. Furthermore, if we look at the case $\sigma_F=4$, all the methods converge but the random strategy reaches higher levels of accuracy.  

\begin{figure}[!htb]
\centering
	\includegraphics[width=0.327\linewidth]{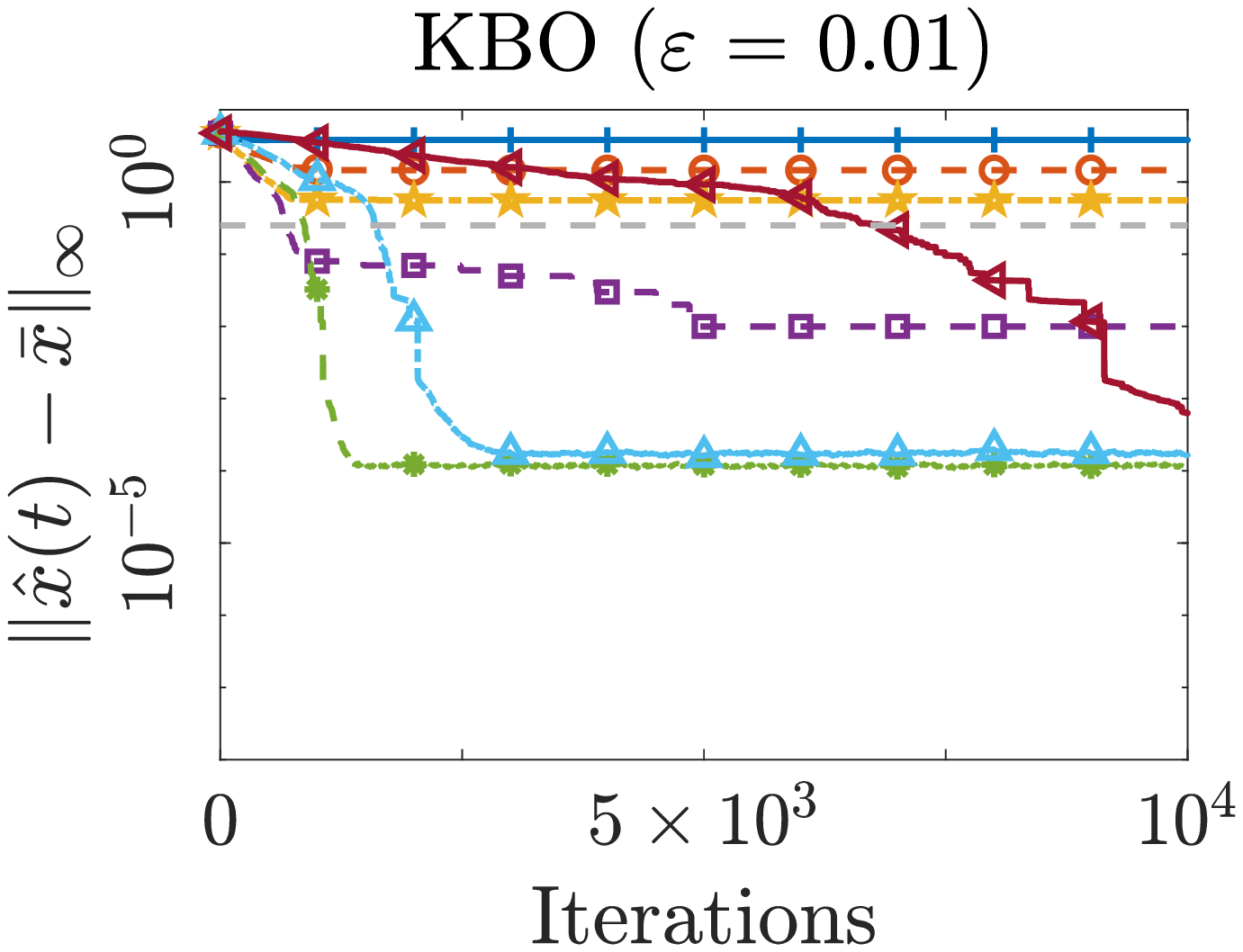}
	\includegraphics[width=0.327\linewidth]{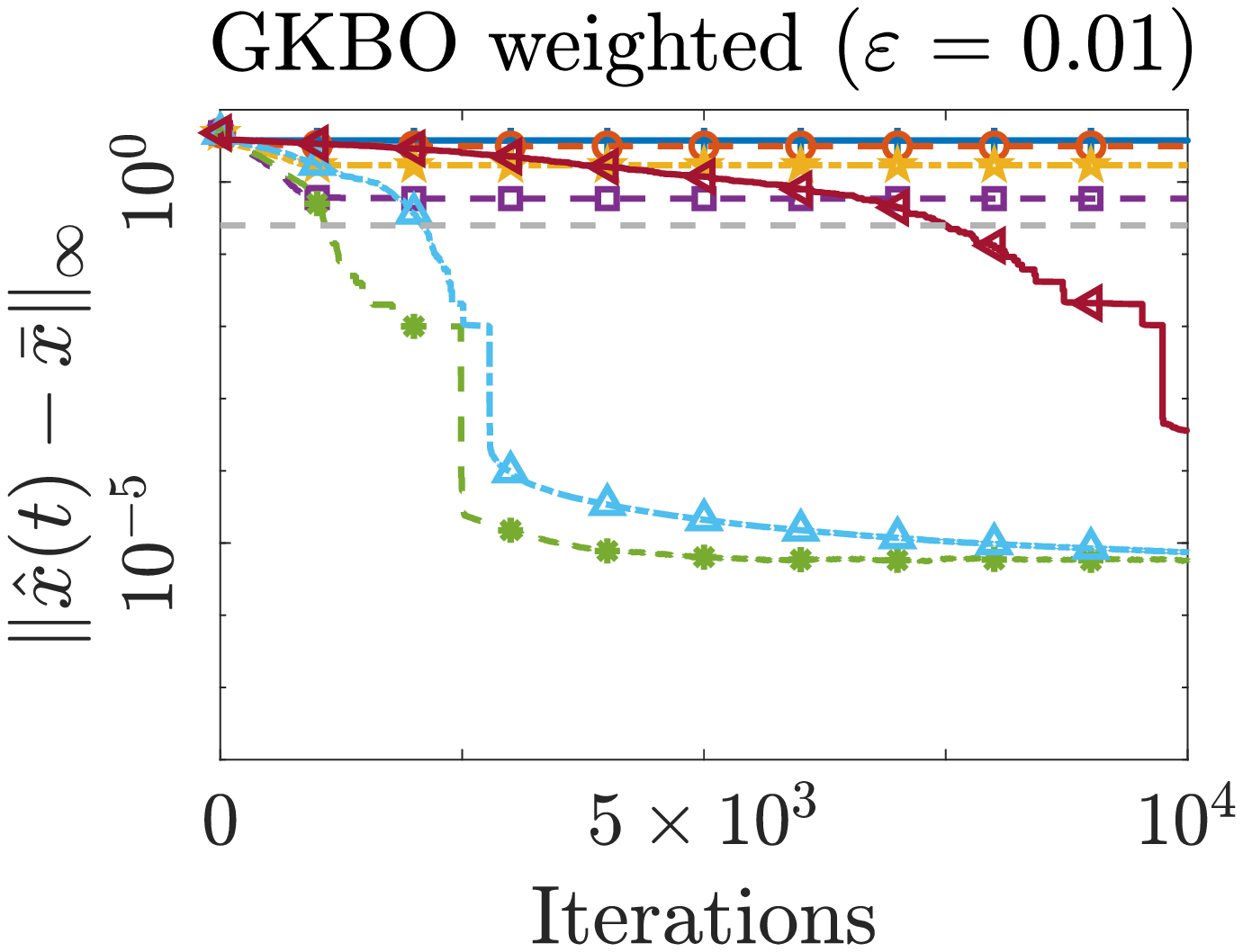}
	\includegraphics[width=0.327\linewidth]{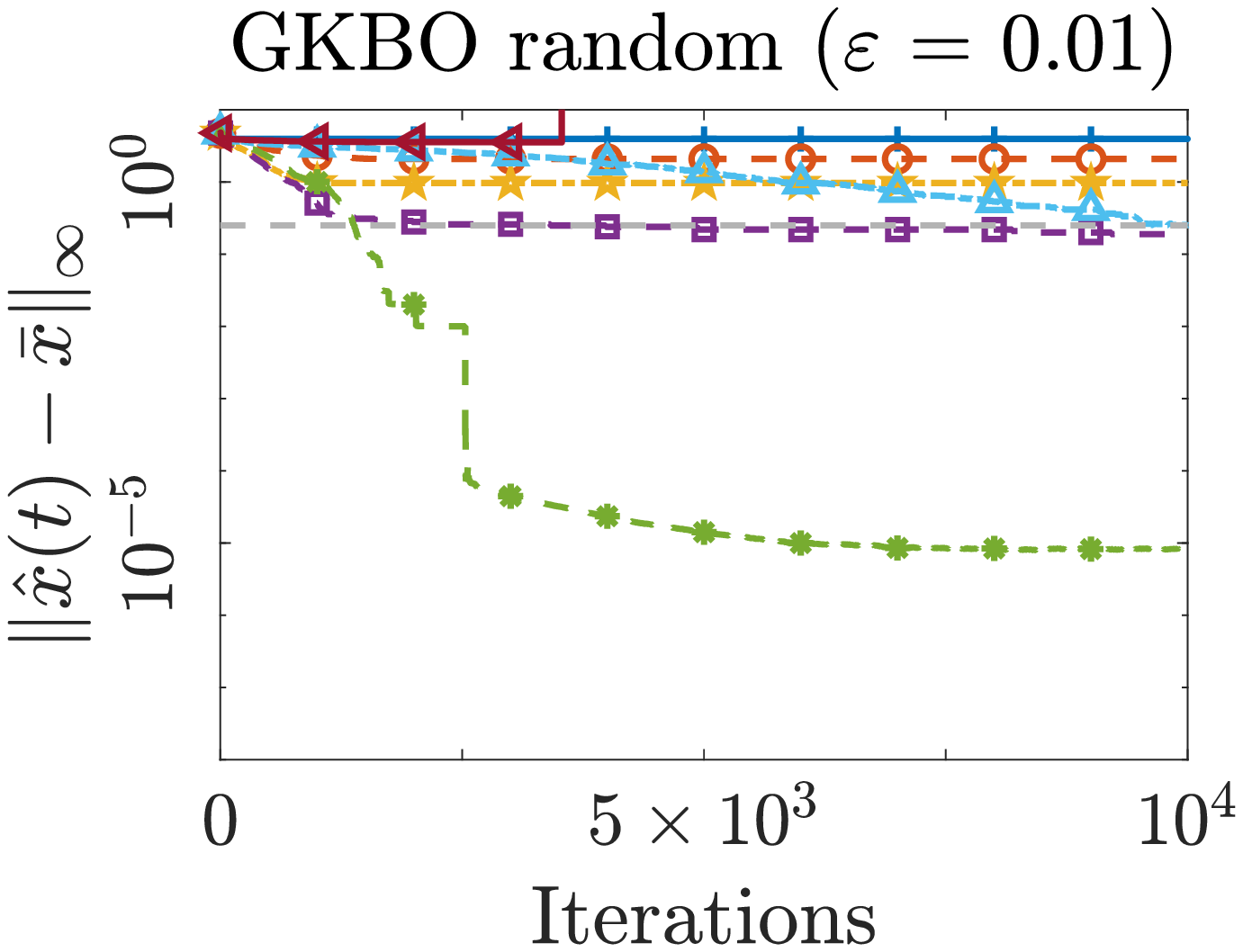}\\
	\vspace{.3cm}
	\includegraphics[width=0.327\linewidth]{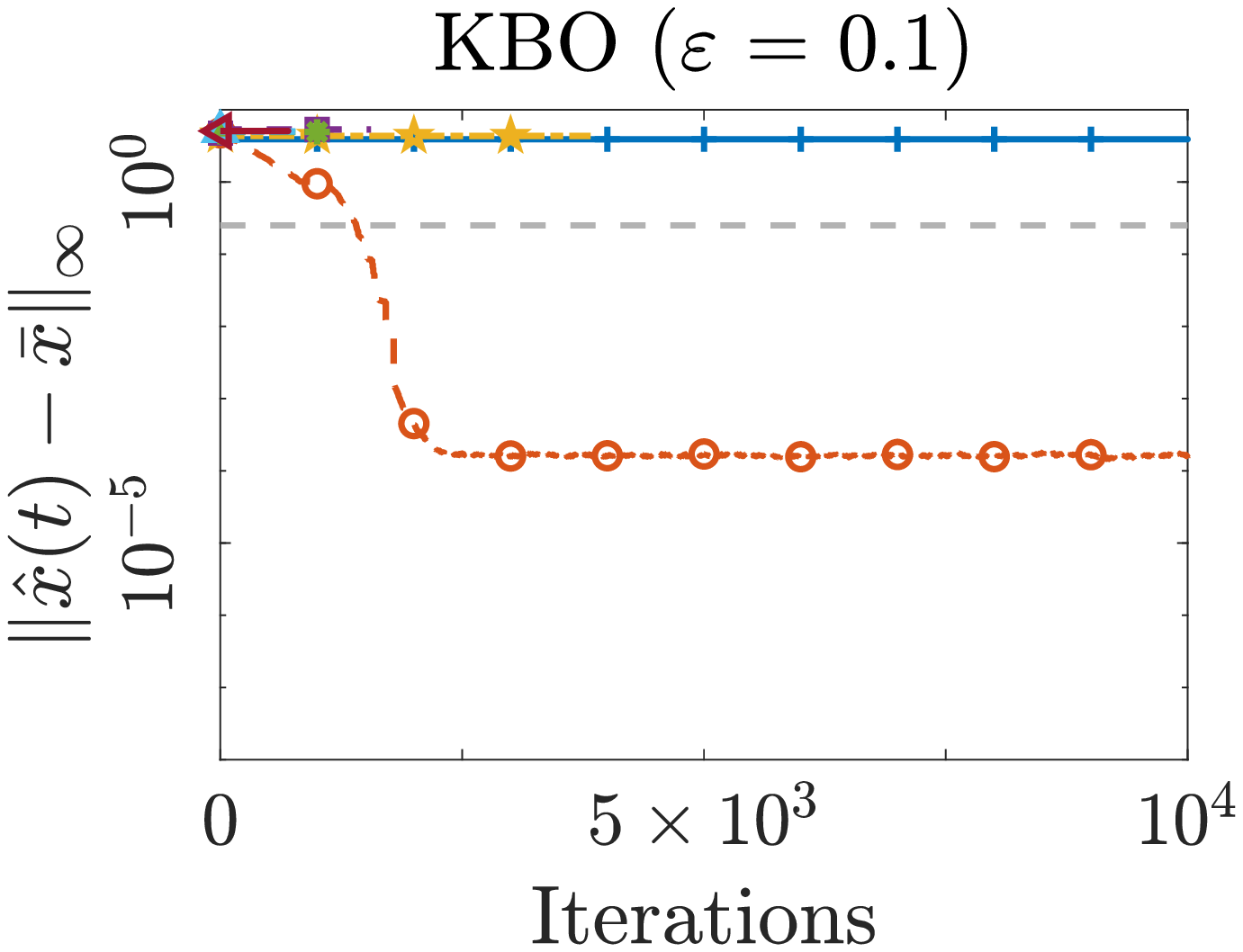}
	\includegraphics[width=0.327\linewidth]{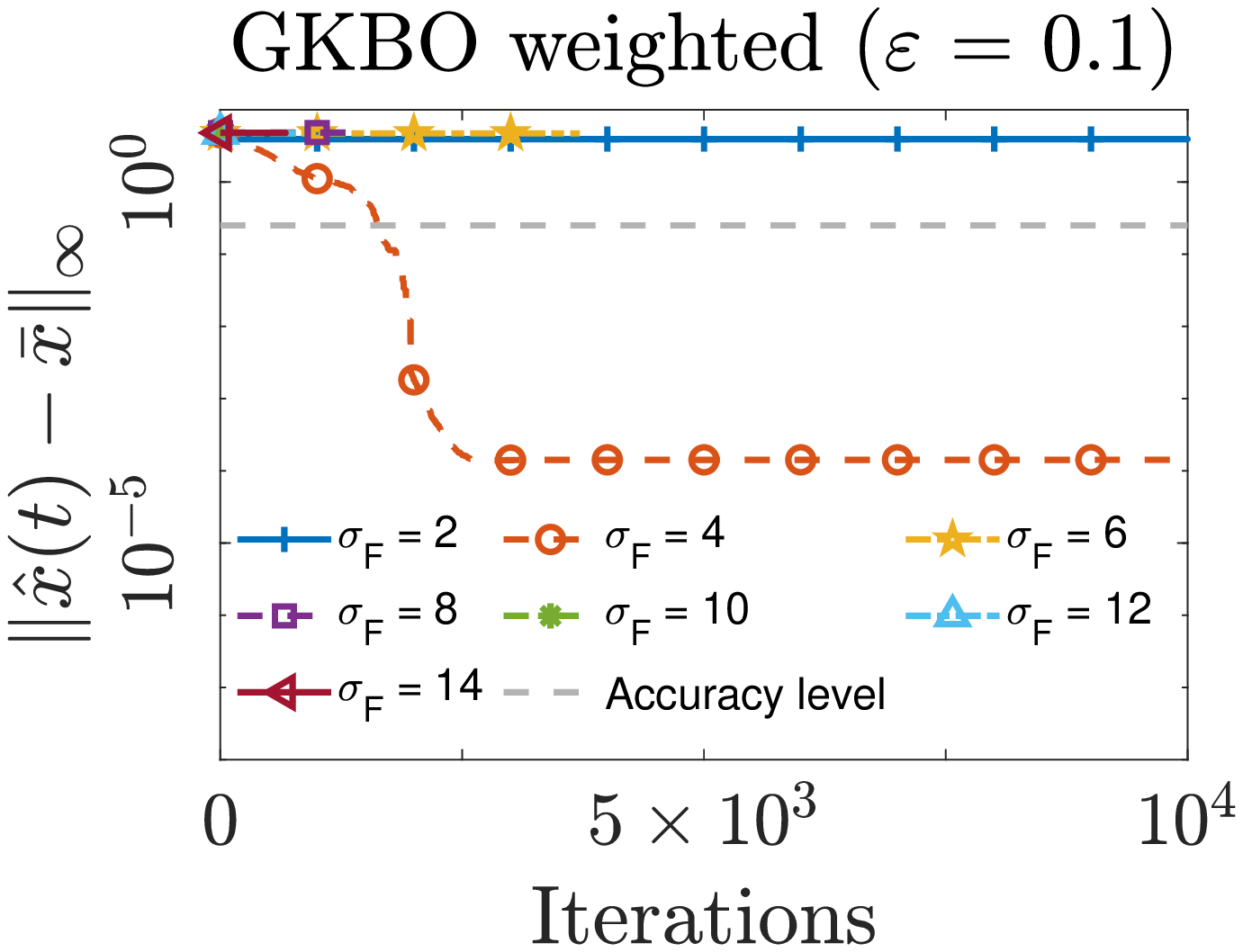}
	\includegraphics[width=0.327\linewidth]{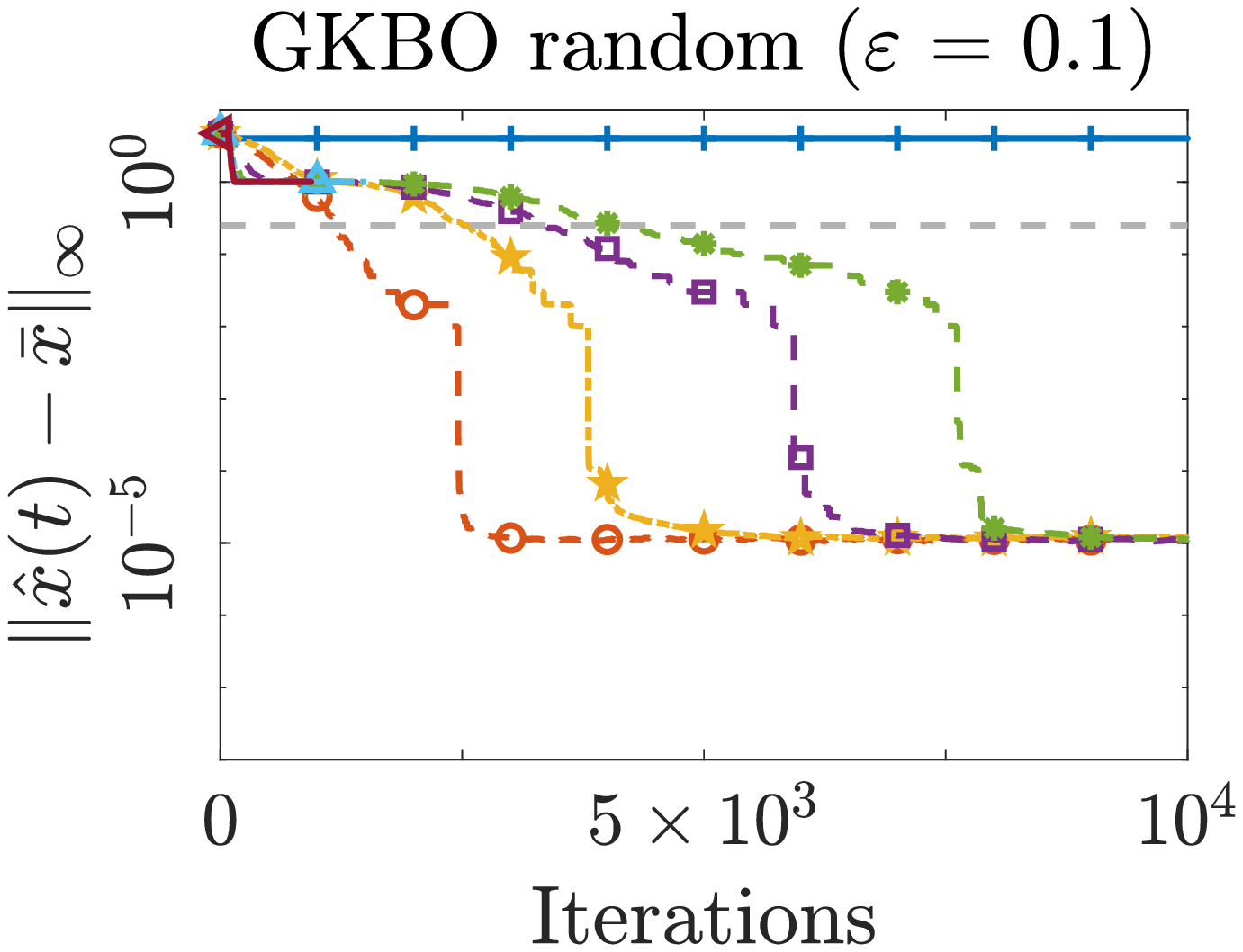}
	\vspace*{8pt}
	\caption{Accuracy of the KBO and of the GKBO algorithm as $\sigma_F$ varies  for $d=20$ and $\varepsilon=0.01$ (first row), $\varepsilon = 0.1$ (second row) for the translated Rastrigin function. From the left to the right, KBO, GKBO with random leader emergence and GKBO with weighted leader emergence. The markers have been added to distinguish the lines.    
	}
	\label{fig:accuracy}
\end{figure}
In Figure \ref{fig:KBO_GKBO_hvaries} the results in terms of success rate and number of iterations needed for different values of $\varepsilon$ are shown.  If $\varepsilon = 0.01$ the success rate are of the GKBO methods is smaller than the one of the KBO but the number of needed iterations is reduced. If $\varepsilon =0.1$, the success rate area is enlarged for the strategy with random leader emergence. With this test we confirm the results obtained in Figure \ref{fig:accuracy}. Moreover, the number of iterations is reduced with respect to the KBO and the GKBO method with weighted leader emergence.
\begin{figure}[!htb]
	\centering
	\includegraphics[width=0.495\linewidth]{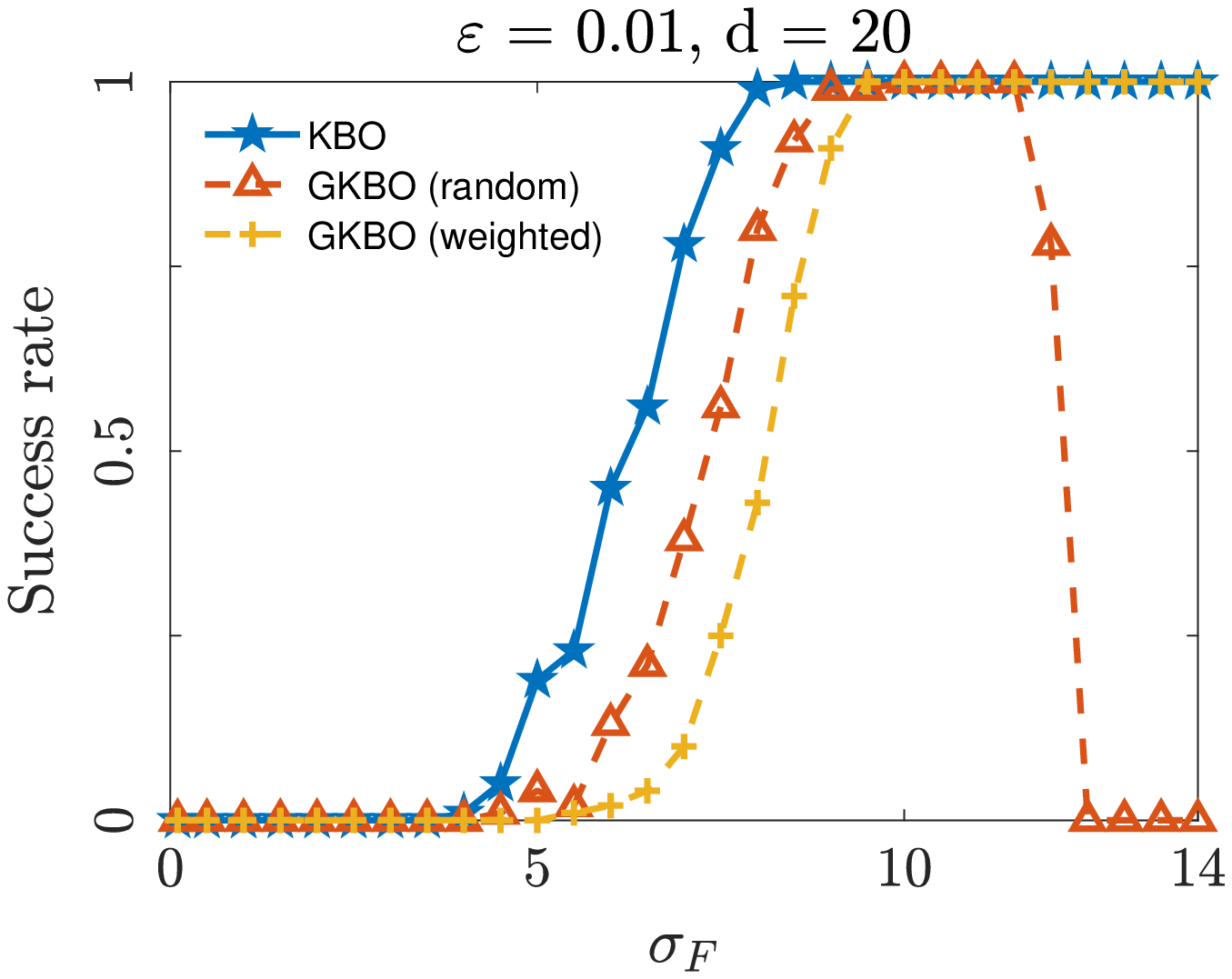}
	\includegraphics[width=0.495\linewidth]{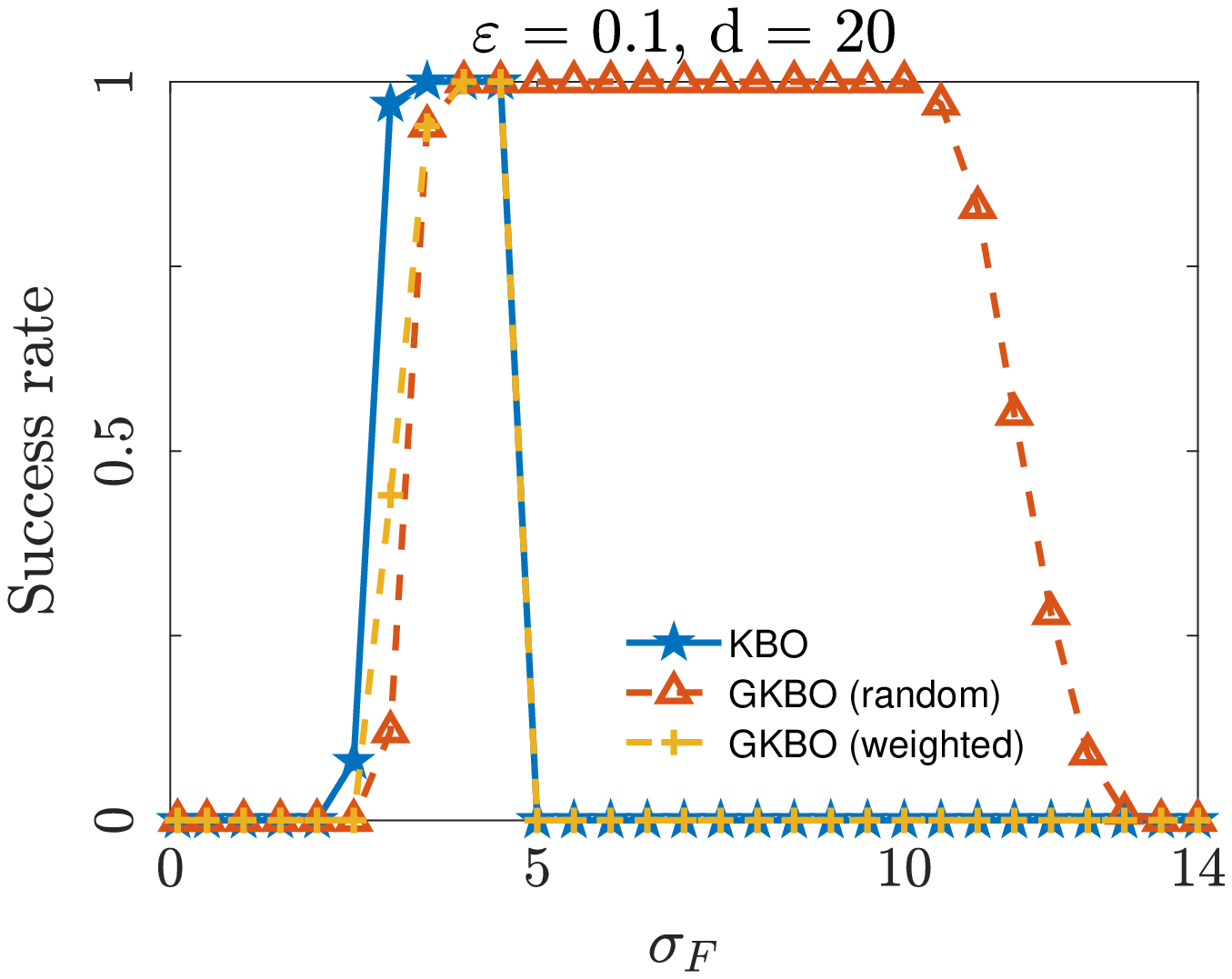}\\
	\includegraphics[width=0.495\linewidth]{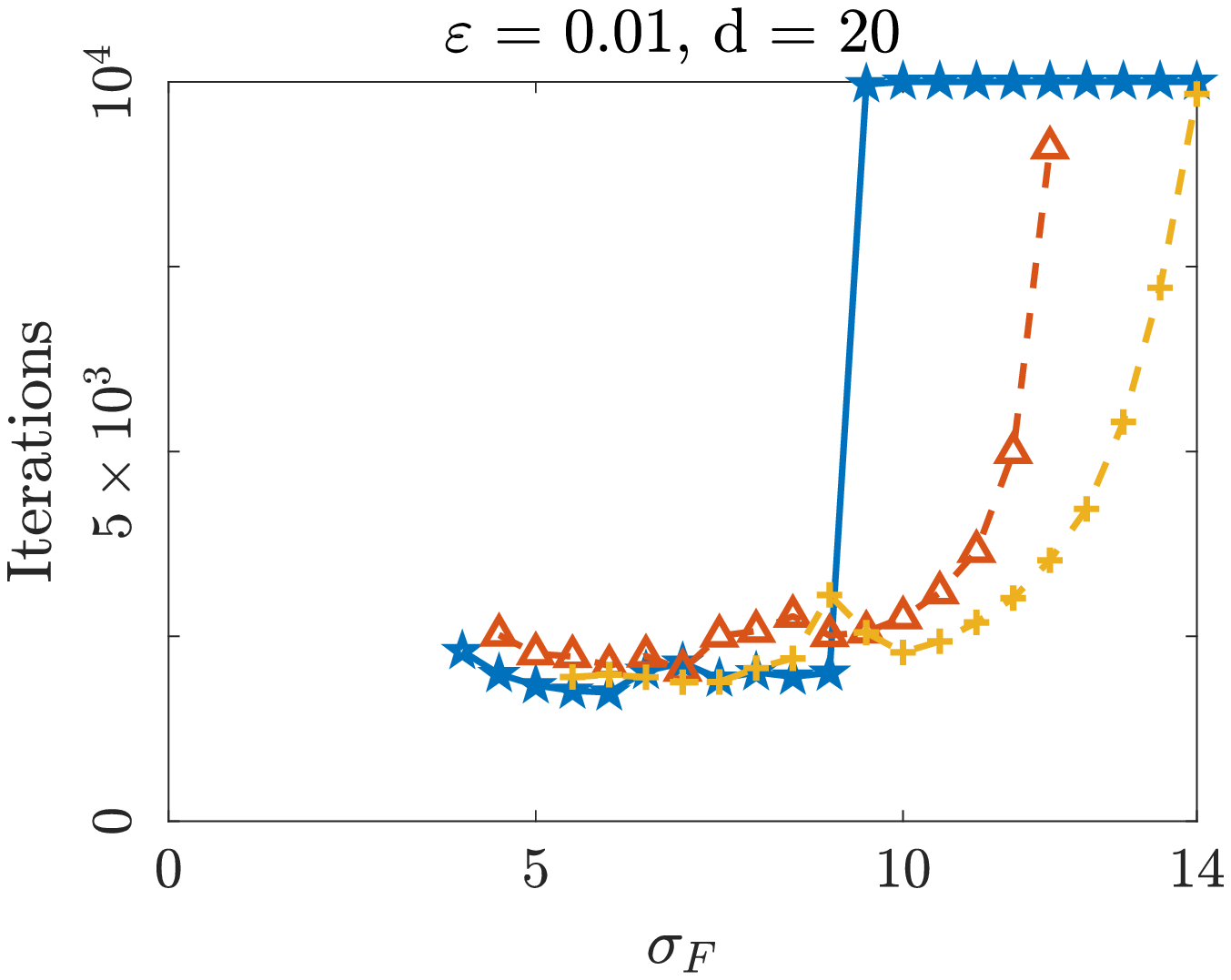}
	\includegraphics[width=0.495\linewidth]{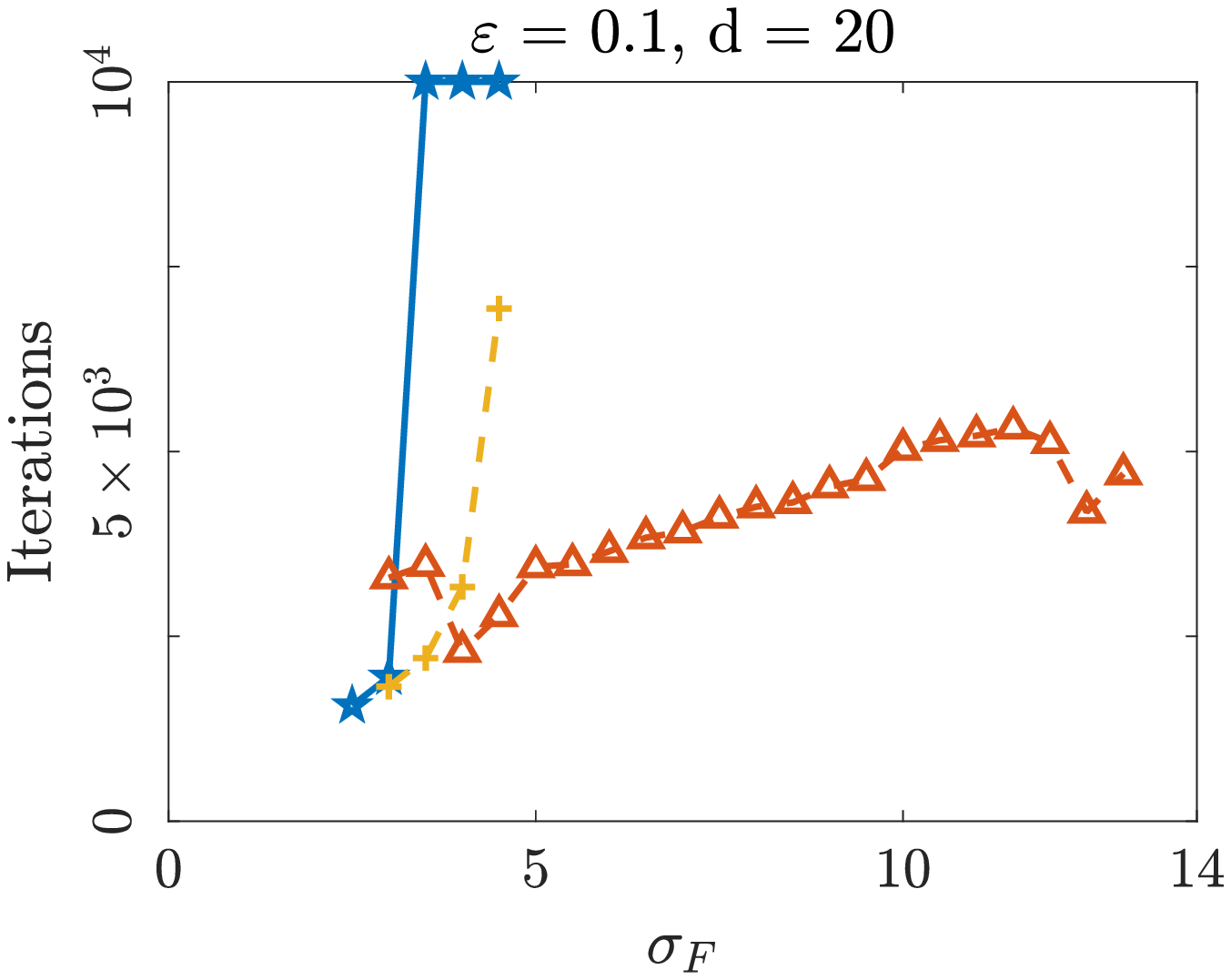}
	\vspace*{8pt}
	\caption{Success rates and means of the number of iterations as $\sigma_F$ varies for $\varepsilon=0.01$ (first column), $\varepsilon = 0.1$ (second column) for the translated Rastrigin function.   The markers have been added to distinguish the lines.    
	}
	\label{fig:KBO_GKBO_hvaries}
\end{figure}

\subsection{Test 7: Comparison of different benchmark functions}\label{sec:test7}
In the previous subsection we tested the different algorithms and different parameter sets with the translated Rastrigin function. Now, we choose $\sigma_F$ such that both the KBO and the GKBO algorithms have success rate equal to one in the previous studies and test different benchmark functions in $20$ dimensions. 
In Figure \ref{fig:funct_in0} the comparison of KBO and GKBO in terms of success rate and mean number of iterations are shown. GKBO with both variants of leader emergence outperforms KBO in terms of the number of iterations.
\begin{figure}[!htb]
\centering
	\includegraphics[width=0.495\linewidth]{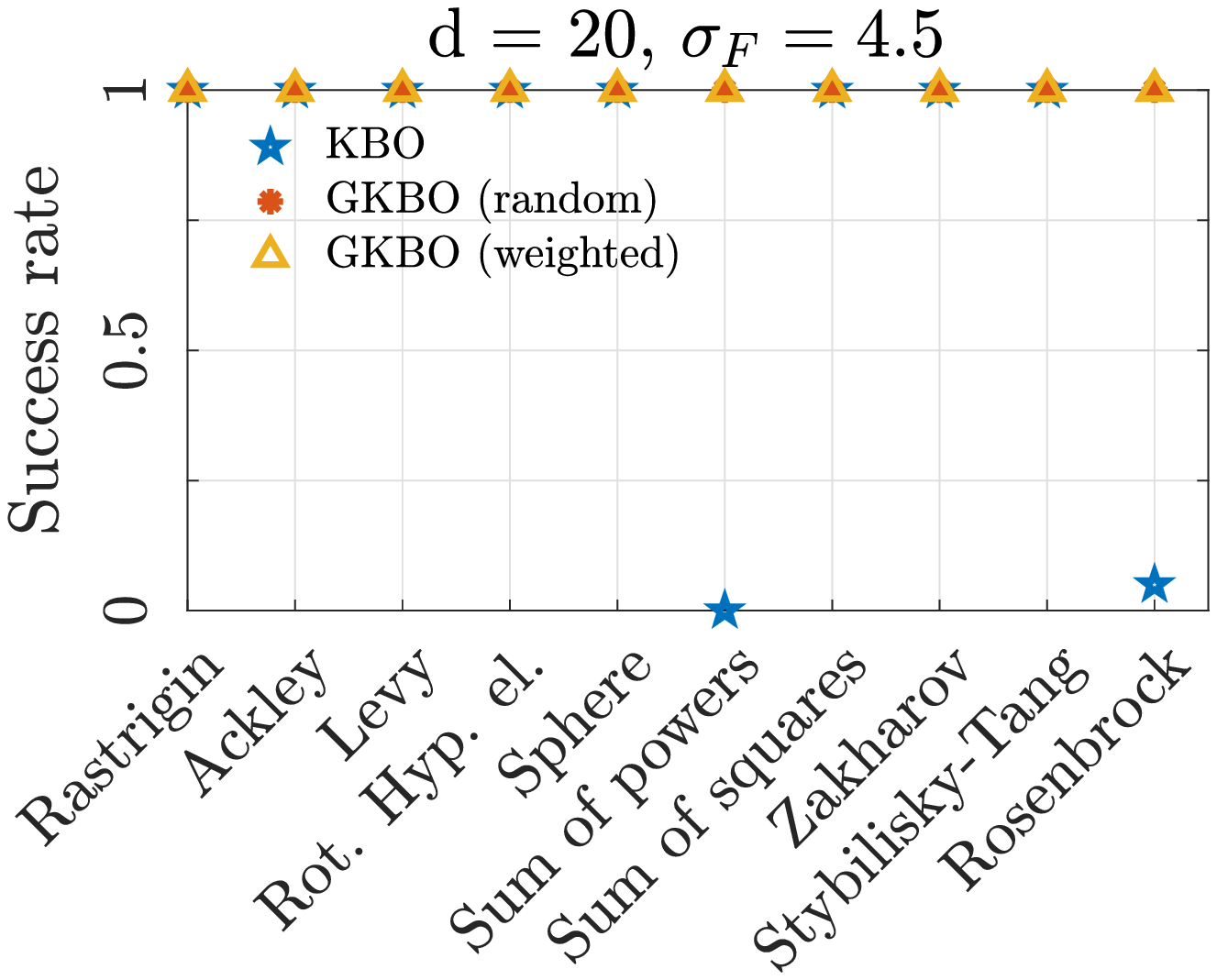}
	\includegraphics[width=0.495\linewidth]{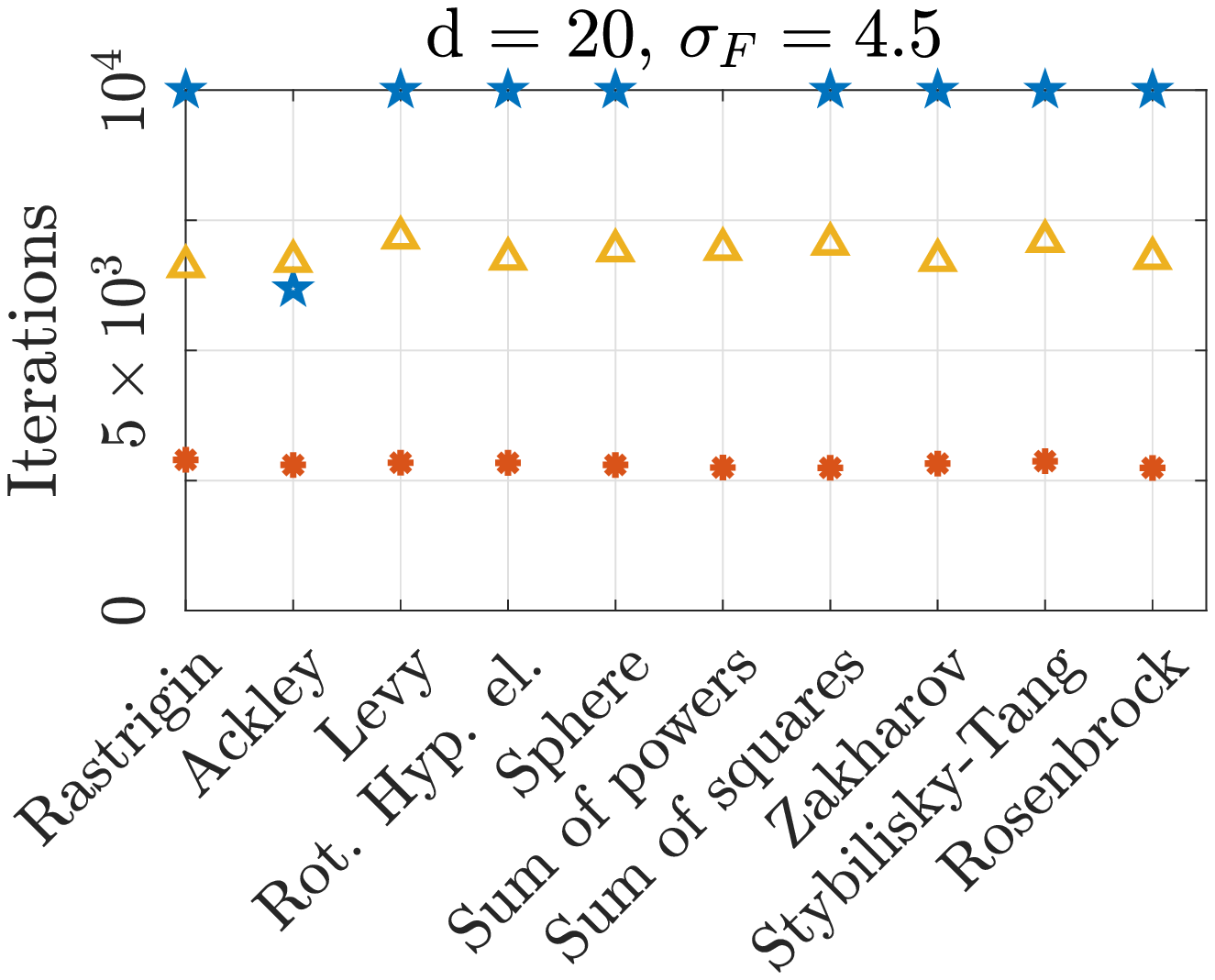}
\vspace*{8pt}
	\caption{ Success rate and mean iterations number for the different benchmark function  for fixed $d=20$ and $\sigma_F = 4.5$.  Markers denotes the value of the success rate and of the iterations number for the different benchmark functions.
	}
	\label{fig:funct_in0}
\end{figure}

\section{Conclusion}\label{sec:conclusion}
We propose a variant of the KBO method for global optimization which is enhanced by a transition process, inspired by genetic dynamics. These lead to a population divided into two species which we call followers and leaders. We adapt the convergence analysis to the new method and show in particular that the population concentrates in the long-time limit arbitrarily close to the global minimizer of the cost function. Numerical results show the feasibility of the approach and the improvement of the proposed generalization in terms of numerical effort. In particular, we measure the numerical effort in terms of the number of iterations which are proportional to the number of function evaluations used in the optimization process. Since GKBO outperforms KBO in this measure, the proposed variant is especially attractive for optimization problems with expensive cost function evaluations. This can be the case for example in engineering applications where each cost function evaluation requires a time-consuming simulation, such as the solution of high-dimensional PDEs  with commercial codes.

\section*{Acknowledgment}
GA and FF were partially supported by the MIUR-PRIN
Project 2022, No. 2022N9BM3N ``{\em Efficient numerical schemes and optimal control methods for time-dependent PDEs}''. 
\newpage


\begin{thebibliography}{00}

\bibitem{albi2019vehicular}
Giacomo Albi, Nicola Bellomo, Luisa Fermo, S-Y Ha, J~Kim, Lorenzo Pareschi,
David Poyato, and Juan Soler.
\newblock Vehicular traffic, crowds, and swarms: From kinetic theory and
multiscale methods to applications and research perspectives.
\newblock {\em Mathematical Models and Methods in Applied Sciences},
29(10):1901--2005, 2019.

\bibitem{albi2019leader}
Giacomo Albi, Mattia Bongini, Francesco Rossi, and Francesco Solombrino.
\newblock Leader formation with mean-field birth and death models.
\newblock {\em Mathematical Models and Methods in Applied Sciences},
29(04):633--679, 2019.

\bibitem{albi2013binary}
Giacomo Albi and Lorenzo Pareschi.
\newblock Binary interaction algorithms for the simulation of flocking and
swarming dynamics.
\newblock {\em Multiscale Modeling \& Simulation}, 11(1):1--29, 2013.

\bibitem{bellomo2021life}
Nicola Bellomo, Diletta Burini, Giovanni Dosi, Livio Gibelli, Damian Knopoff,
Nisrine Outada, Pietro Terna, and Maria~Enrica Virgillito.
\newblock What is life? a perspective of the mathematical kinetic theory of
active particles.
\newblock {\em Mathematical Models and Methods in Applied Sciences},
31(09):1821--1866, 2021.

\bibitem{bellomo2022towards}
Nicola Bellomo, Livio Gibelli, Annalisa Quaini, and Alessandro Reali.
\newblock Towards a mathematical theory of behavioral human crowds.
\newblock {\em Mathematical Models and Methods in Applied Sciences},
32(02):321--358, 2022.

\bibitem{benfenati2022binary}
Alessandro Benfenati, Giacomo Borghi, and Lorenzo Pareschi.
\newblock Binary interaction methods for high dimensional global optimization
and machine learning.
\newblock {\em Applied Mathematics \& Optimization}, 86(1):9, 2022.

\bibitem{borghi2022consensus}
Giacomo Borghi, Michael Herty, and Lorenzo Pareschi.
\newblock A consensus-based algorithm for multi-objective optimization and its
mean-field description.
\newblock In {\em 2022 IEEE 61st Conference on Decision and Control (CDC)},
pages 4131--4136. IEEE, 2022.

\bibitem{borghi2023constrained}
Giacomo Borghi, Michael Herty, and Lorenzo Pareschi.
\newblock Constrained consensus-based optimization.
\newblock {\em SIAM Journal on Optimization}, 33(1):211--236, 2023.

\bibitem{bottou2018optimization}
L{\'e}on Bottou, Frank~E Curtis, and Jorge Nocedal.
\newblock Optimization methods for large-scale machine learning.
\newblock {\em SIAM review}, 60(2):223--311, 2018.

\bibitem{carrillo2018analytical}
Jos{\'e}~A Carrillo, Young-Pil Choi, Claudia Totzeck, and Oliver Tse.
\newblock An analytical framework for consensus-based global optimization
method.
\newblock {\em Mathematical Models and Methods in Applied Sciences},
28(06):1037--1066, 2018.

\bibitem{carrillo2010particle}
Jos{\'e}~A Carrillo, Massimo Fornasier, Giuseppe Toscani, and Francesco Vecil.
\newblock Particle, kinetic, and hydrodynamic models of swarming.
\newblock {\em Mathematical modeling of collective behavior in socio-economic
	and life sciences}, pages 297--336, 2010.

\bibitem{carrillo2021consensus}
Jos{\'e}~A Carrillo, Shi Jin, Lei Li, and Yuhua Zhu.
\newblock A consensus-based global optimization method for high dimensional
machine learning problems.
\newblock {\em ESAIM: Control, Optimisation and Calculus of Variations}, 27:S5,
2021.

\bibitem{Chen2022adamCBO}
Shi Chen, JingrunJin and Liyao Lyu.
\newblock A consensus-based global optimization method with adaptive momentum
estimation.
\newblock {\em Communications in Computational Physics}, 31(4):1296--1316,
2022.

\bibitem{dembo1998zeitouni}
A~Dembo.
\newblock Zeitouni, 0. large deviations techniques and applications.
\newblock {\em Applications of Mathematics}, 38, 1998.

\bibitem{fornasier2020consensus}
Massimo Fornasier, Hui Huang, Lorenzo Pareschi, and Philippe S{\"u}nnen.
\newblock Consensus-based optimization on hypersurfaces: Well-posedness and
mean-field limit.
\newblock {\em Mathematical Models and Methods in Applied Sciences},
30(14):2725--2751, 2020.

\bibitem{fornasier2021consensus}
Massimo Fornasier, Hui Huang, Lorenzo Pareschi, and Philippe S{\"u}nnen.
\newblock Consensus-based optimization on the sphere: Convergence to global
minimizers and machine learning.
\newblock {\em The Journal of Machine Learning Research}, 22(1):10722--10776,
2021.

\bibitem{golberg1989genetic}
David~E. Golberg.
\newblock {\em Genetic algorithms in search, optimization, and machine
	learning}, volume 1989.
\newblock 1989.

\bibitem{grassi2021particle}
Sara Grassi and Lorenzo Pareschi.
\newblock From particle swarm optimization to consensus based optimization:
stochastic modeling and mean-field limit.
\newblock {\em Mathematical Models and Methods in Applied Sciences},
31(08):1625--1657, 2021.

\bibitem{haskovec2013flocking}
Jan Haskovec.
\newblock Flocking dynamics and mean-field limit in the {C}ucker--{S}male-type
model with topological interactions.
\newblock {\em Physica D: Nonlinear Phenomena}, 261:42--51, 2013.

\bibitem{jamil2013literature}
Momin Jamil and Xin-She Yang.
\newblock A literature survey of benchmark functions for global optimisation
problems.
\newblock {\em International Journal of Mathematical Modelling and Numerical
	Optimisation}, 4(2):150--194, 2013.

\bibitem{kalise2023jump}
Dante Kalise, Akash Sharma, and Michael~V. Tretyakov.
\newblock Consensus-based optimization via jump-diffusion stochastic
differential equations.
\newblock {\em Mathematical Models and Methods in Applied Sciences}, 33(2):289--339, 2023.

\bibitem{kennedy1995particle}
James Kennedy and Russell Eberhart.
\newblock Particle swarm optimization.
\newblock In {\em Proceedings of ICNN'95-international conference on neural
	networks}, volume~4, pages 1942--1948. IEEE, 1995.

\bibitem{klamroth2022consensus}
Kathrin Klamroth, Michael Stiglmayr, and Claudia Totzeck.
\newblock Consensus-based optimization for multi-objective problems: A
multi-swarm approach.
\newblock {\em arXiv preprint}, 2022.

\bibitem{loy2021boltzmann}
Nadia Loy and Andrea Tosin.
\newblock Boltzmann-type equations for multi-agent systems with label
switching.
\newblock {\em Kinetic and Related Models}, 14(5):867--894, 2021.

\bibitem{zbigniew1993ga}
Zbigniew Michalewicz.
\newblock {\em Genetic Algorithms + Data Structures = Evolution Programs}.
\newblock Springer, 1996.

\bibitem{mitchell1995genetic}
Melanie Mitchell.
\newblock Genetic algorithms: An overview.
\newblock {\em Complexity}, 1(1):31--39, 1995.

\bibitem{motsch2014heterophilious}
Sebastien Motsch and Eitan Tadmor.
\newblock Heterophilious dynamics enhances consensus.
\newblock {\em SIAM review}, 56(4):577--621, 2014.

\bibitem{nanbu1980direct}
Kenichi Nanbu.
\newblock Direct simulation scheme derived from the {B}oltzmann equation. i.
monocomponent gases.
\newblock {\em Journal of the Physical Society of Japan}, 49(5):2042--2049,
1980.

\bibitem{pareschi2001introduction}
Lorenzo Pareschi and Giovanni Russo.
\newblock An introduction to {M}onte {C}arlo method for the {B}oltzmann
equation.
\newblock In {\em ESAIM: Proceedings}, volume~10, pages 35--75. EDP Sciences,
2001.

\bibitem{pareschi2013interacting}
Lorenzo Pareschi and Giuseppe Toscani.
\newblock {\em Interacting multiagent systems: kinetic equations and {M}onte
	{C}arlo methods}.
\newblock OUP Oxford, 2013.

\bibitem{pinnau2017consensus}
Ren{\'e} Pinnau, Claudia Totzeck, Oliver Tse, and Stephan Martin.
\newblock { \em A consensus-based model for global optimization and its mean-field
limit}.
\newblock {\em Mathematical Models and Methods in Applied Sciences},
27(01):183--204, 2017.

\bibitem{poli2007particle}
Riccardo Poli, James Kennedy, and Tim Blackwell.
\newblock Particle swarm optimization: An overview.
\newblock {\em Swarm intelligence}, 1:33--57, 2007.

\bibitem{toledo2014global}
Claudio Fabiano~Motta Toledo, L~Oliveira, and Paulo~Morelato Fran{\c{c}}a.
\newblock Global optimization using a genetic algorithm with hierarchically
structured population.
\newblock {\em Journal of Computational and Applied Mathematics}, 261:341--351,
2014.

\bibitem{totzeck2021trends}
Claudia Totzeck.
\newblock Trends in consensus-based optimization.
\newblock In {\em Active Particles, Volume 3: Advances in Theory, Models, and
	Applications}, pages 201--226. Springer, 2021.

\bibitem{totzeck2020memory}
Claudia Totzeck and Marie-Therese Wolfram.
\newblock Consensus-based global optimization with personal best.
\newblock {\em Mathematical Biosciences and Engineering}, 17(5):6026--6044,
2020.

\end{thebibliography}
\end{document}